\documentclass[11pt, twoside, leqno]{article}

\usepackage{mathrsfs}
\usepackage{amssymb}
\usepackage{amsmath}
\usepackage{mathrsfs}
\usepackage{amsthm}
\usepackage{amsfonts}
\usepackage{color}
\usepackage{latexsym}
\usepackage{txfonts}
\usepackage{indentfirst}
\usepackage{soul}

\allowdisplaybreaks

\def\rr{{\mathbb R}}
\def\rn{{\mathbb{R}^n}}

\def\nn{{\mathbb N}}
\def\zz{{\mathbb Z}}

\def\CM{{\mathcal M}}

\def\fz{\infty }
\def\az{\alpha}

\def\ls{\lesssim}
\def\noz{\nonumber}

\def\loc{{\mathrm{loc}}}

\def \essinf {\mathop\mathrm{\,ess\,inf\,}}
\def \esssup {\mathop\mathrm{\,ess\,sup\,}}

\def\XXint#1#2#3{{\setbox0=\hbox{$#1{#2#3}{\int}$ }
\vcenter{\hbox{$#2#3$ }}\kern-.6\wd0}}

\def\lz{{\lambda}}

\def\PBMO{{\mathrm{PBMO}}}

\def\pr{{\mathrm{pr}}}

\DeclareMathOperator{\diam}{diam}

\newcommand{\BMO}{\mathrm{BMO\,}}

\newtheorem{theorem}{Theorem}[section]
\newtheorem{lemma}[theorem]{Lemma}
\newtheorem{corollary}[theorem]{Corollary}
\newtheorem{proposition}[theorem]{Proposition}

\theoremstyle{definition}
\newtheorem{remark}[theorem]{Remark}
\newtheorem{definition}[theorem]{Definition}

\renewcommand{\appendix}{\par
\setcounter{section}{0}
\setcounter{subsection}{0}
\setcounter{subsubsection}{0}
\gdef\thesection{\@Alph\c@section}
\gdef\thesubsection{\@Alph\c@section.\@arabic\c@subsection}
\gdef\theHsection{\@Alph\c@section.}
\gdef\theHsubsection{\@Alph\c@section.\@arabic\c@subsection}
\csname appendixmore\endcsname
}

\numberwithin{equation}{section}

\begin{document}

\title{\bf\Large Parabolic Muckenhoupt Weights
on Spaces of Homogeneous Type
\footnotetext{\hspace{-0.35cm} 2020 {\it Mathematics Subject Classification}.
Primary 46E36; Secondary 42B35, 42B25, 30L99.\endgraf
{\it Key words and phrases.} space of homogeneous type,
parabolic Muckenhoupt weight,
parabolic maximal operator,
weighted norm inequality,
reverse H\"older inequality, parabolic BMO,
John--Nirenberg inequality.\endgraf
This project is supported by the Academy of Finland,
the National Natural Science Foundation of China
(Grant Nos.\ 11971058 and 12071197) and
the National Key Research and Development Program of China
(Grant No.\ 2020YFA0712900).}}
\date{}
\author{Juha Kinnunen, Kim Myyryl\"{a}inen,
Dachun Yang\footnote{Corresponding author,
E-mail: \texttt{dcyang@bnu.edu.cn}/{\color{red} August 17, 2022}/Final version.}\
\ and Chenfeng Zhu}

\maketitle

\vspace{-0.7cm}

\begin{center}
\begin{minipage}{13cm}
{\small {\bf Abstract}\quad
This work discusses parabolic Muckenhoupt weights on spaces
of homogeneous type, i.e.\  quasi-metric spaces with both a doubling measure and 
an additional monotone geodesic property.
The main results include a characterization in terms of weighted norm inequalities
for parabolic maximal operators, a reverse H\"older inequality, and a
Jones-type factorization result for this class of weights.
The connection between the space of parabolic bounded mean oscillation and parabolic
Muckenhoupt weights is studied by applying a parabolic John--Nirenberg lemma.
A  Coifman--Rochberg-type characterization of the space of parabolic bounded mean oscillation
in terms of parabolic maximal functions is also given.
The main challenges in the parabolic theory are related to the time lag in the estimates.
The results are motivated by the corresponding Euclidean theory and the regularity theory for
parabolic variational problems on metric measure spaces.
}
\end{minipage}
\end{center}
	
\vspace{0.2cm}
	
\tableofcontents

\vspace{0.2cm}

\section{Introduction}

It is well known that both the real-variable theory of function
spaces and the boundedness of Calder\'on--Zygmund operators
play a fundamental role in harmonic analysis and partial
differential equations.
The class $A_q(\rn)$ of Muckenhoupt weights  (see, for instance, Muckenhoupt \cite{m1972},
Hunt et al. \cite{hmw1973},
Coifman and Fefferman \cite{cf1974},
Garc\'ia-Cuerva and Rubio de Francia \cite{cf1985}) gives an appropriate condition for
a nonnegative locally integrable weight function
$\omega$ on $\rn$, such that some well-known operators
(for instance,
the Hardy--Littlewood maximal operator,
the Hilbert transform,
and the Calder\'on--Zygmund operators)
are bounded on the weighted Lebesgue space
$L^q(\rn,\omega)$ for any given $q\in(1,\infty)$.
Recall that, for any $q\in(1,\infty)$,
a nonnegative locally integrable function $\omega$
on $\rn$ is said to belong to the \emph{Muckenhoupt class $A_q(\rn)$} if
\begin{align*}
\sup_{Q}\left[\frac{1}{|Q|}\int_Q\omega(x)\,dx\right]
\left[\frac{1}{|Q|}\int_Q
\left[\omega(x)\right]^{1-q'}\,dx\right]^{q-1}<\infty,
\end{align*}
where the supremum is taken over all cubes
$Q\subset\mathbb{R}^n$ with their edges parallel to the coordinate 
axes and where $1/q+1/q'=1$.
For any given $q\in(1,\infty)$,
the \emph{weighted Lebesgue space $L^q(\rn,\omega)$}
is defined to be the set of all the Lebesgue
measurable functions $f$  on $\rn$ such that
$$\|f\|_{L^q(\rn,\omega)}:=\left[\int_{\rn}
|f(x)|^q\omega(x)\,dx\right]^{1/q}<\fz.$$
There exists a well-established theory related to the Muckenhoupt weights
with applications in elliptic partial differential equations on $\rn$
(see, for instance, \cite{bb2021,bbd2021,bbd20211,bbd2020,
Bbd2020,bll2021,fks1982,h2012,ll2022,lw2020}).
There has also been a lot of interest in extending both harmonic analysis
and the calculus of variations to nonsmooth spaces such as spaces of homogeneous type,
i.e.\ a quasi-metric spaces with a doubling measure; see Coifman and Weiss \cite{CW1,Cowe77}.
For more about the Muckenhoupt weights on spaces of homogeneous type
and their applications, we refer the reader to \cite{abi2005,act2014,aht2017,hpr2012,ms1981}
for more properties of these weights,
to \cite{bd2020,c1976,c1990,dlsvwy2021,dgklwy2021,ggkk1998,jlpv2020} for some applications
of these weights in the weighted boundedness of operators,
and to \cite{fmy2020,st1989,YHYY1,YHYY2} for some applications of these weights 
in the related weighted
Hardy spaces.

For the case of parabolic Muckenhoupt weights,
let us first recall the boundedness of
the one-sided maximal operator on $\rr$.
For any $q\in(1,\infty)$,
Sawyer \cite{s1986} introduced
the \emph{one-sided $A_q(\rr)$ condition}
\begin{align}\label{1707}
\sup_{x\in\mathbb{R},\,h\in(0,\infty)}
\left[\frac{1}{h}\int_{x-h}^x\omega(t)\,dt\right]
\left[\frac{1}{h}\int_x^{x+h}
\left[\omega(t)\right]^{1-q'}\,dt\right]^{q-1}<\infty
\end{align}
and showed that the \emph{one-sided maximal operator}
$M^+_{\rm{os}}$ defined by
\begin{align*}
M^+_{\rm{os}}f(x):=\sup_{h\in(0,\infty)}\frac 1h\int_x^{x+h}|f(t)|\,dt,
\end{align*}
with $x\in\rr$,
is bounded on $L^q(\rr,\omega)$
if and only if $\omega$ satisfies \eqref{1707}.
These weights and the corresponding
one-sided maximal operators
have attracted a lot of attention
(see, for instance, \cite{ac1998,a1997,cu1995,m1993,
mt1993,mt1994,mot1990,mpt1993}).
There exist also many studies related
to higher-dimensional versions of these one-sided ones
(see, for instance, \cite{b2011,f2011,ks2016,lo2010,o2005}).
Kinnunen and Saari \cite{ks2016} introduced, for any $q\in(1,\infty)$, the
\emph{parabolic Muckenhoupt $A_q^+(\mathbb{R}^{n+1})$ condition}
\begin{align}\label{paq}
\sup_{R}\left[\frac{1}{|R^-(\gamma)|}
\int_{R^-(\gamma)}\omega(x)\,dx\right]
\left[\frac{1}{|R^+(\gamma)|}\int_{R^+(\gamma)}
\left[\omega(x)\right]^{1-q'}\,dx\right]^{q-1}<\infty,
\end{align}
where $R^\pm(\gamma)\subset\mathbb{R}^{n+1}$
are space-time rectangles
with a time lag $\gamma\in[0,1)$,
and developed a theory of parabolic Muckenhoupt weights.
This approach is motivated by the regularity theory for nonlinear parabolic partial differential equations
and has a very close connection to the space of functions with parabolic bounded mean oscillation; see
\cite{ks2016,ks2016',s2016,s2018,kmy2022}.

From the point of view of partial
differential equations, an important observation in Moser's proof \cite{m1961} of
a Harnack inequality for nonnegative solutions of
\begin{align}\label{e}
\text{div}(|\nabla u|^{p-2}\nabla u)=0
\end{align}
on $\rn$ when $p\in(1,\infty)$,
is that the logarithm of a nonnegative solution of \eqref{e} belongs
to $\BMO(\rn)$, the space of functions with bounded mean oscillation,
introduced by John and Nirenberg \cite{jn1961}.
The motivation to study the parabolic $\BMO$ came from the works of Moser
\cite{m1964,m1967,m1971} for $p=2$ and Trudinger \cite{t1968} for $p\in(1,\infty)$ on the parabolic
Harnack inequality for nonnegative solutions of
the doubly nonlinear parabolic equation
\begin{align}\label{2211}
\frac{\partial(|u|^{p-2}u)}{\partial t}-
\text{div}(|\nabla u|^{p-2}\nabla u)=0
\end{align}
on $\rn\times (-\fz,\fz)$.
It is also possible to consider a more general class of equations of type \eqref{e} and \eqref{2211}.
For the behavior of solutions to \eqref{2211} see, for instance,
\cite{gv2006,kk2007,kmy2022,k2012,t1968,v1992}.
The space of parabolic bounded mean oscillation and a parabolic John--Nirenberg lemma are essential in \cite{m1967,t1968}.
Fabes and Garofalo \cite{fg1985} gave a simplified proof of the parabolic John--Nirenberg lemma originally proved in
\cite{m1964,m1967}.
The parabolic $\BMO$ and a John--Nirenberg lemma related to \eqref{2211} has been recently discussed in \cite{kmy2022}.
A very general John--Nirenberg-type lemma on spaces of homogeneous type was proved by Aimar \cite{a1988}.
This approach includes both elliptic and parabolic cases.
Parabolic $\BMO$ is instrumental in the Harnack estimates for nonnegative solutions to \eqref{2211} and, on the other hand,
it follows for the Harnack estimates that nonnegative solutions to \eqref{2211} belong to the parabolic Muckenhoupt class in \eqref{paq}
for any $q\in(1,\infty)$.

Sometimes it is advantageous to consider a variational problem instead of a partial differential equation.
The parabolic variational inequality related to \eqref{2211} is
\begin{align}\label{pvi}
p\int|u|^{p-2}u\frac{\partial\varphi}{\partial t}\,dx\,dt
+\int|\nabla u|^p\,dx\,dt
\le\int|\nabla u+\nabla\varphi|^p\,dx\,dt
\end{align}
for every compactly supported smooth function $\varphi$.
This approach goes back to Lichnewsky and Temam \cite{lt1978}, 
who employed an analogous concept in the case
of the time-dependent minimal surface equation.
Wieser \cite{w1987} showed that the variational approach gives the 
same class of weak solutions as the standard definition of a weak solution.
In particular, a function is a weak solution to \eqref{2211} if and only 
if it satisfies \eqref{pvi} for every compactly supported smooth function $\varphi$.
The main advantage of the variational approach is that, instead of partial derivatives, 
it is based on the modulus of the gradient.
Thus, it is possible to consider the corresponding variational problems on 
metric measure spaces by applying the concept of an upper gradient, see \cite{H2001}.
The theory of upper gradients and Sobolev spaces in  \cite{H2001} is developed 
in the context of a metric space with a doubling measure, but in harmonic 
analysis it is more common to
consider a slightly more general space of homogeneous type.
Harnack estimates for nonlinear parabolic problems of type \eqref{pvi} on 
metric measure spaces have been investigated in \cite{Iv2014,mm2013}.
These results imply that the negative logarithm of a nonnegative variational 
solution belongs to the parabolic BMO.
In addition, a nonnegative variational solution satisfies a parabolic Harnack 
inequality and, consequently, belongs to the parabolic Muckenhoupt class as in \eqref{paq}.

The goal of this article is to create a theory for parabolic Muckenhoupt 
weights on spaces of homogeneous type having an additional monotone geodesic property.
This is motivated by the corresponding Euclidean theory in \cite{ks2016} 
and connections to parabolic variational problems of type \eqref{pvi} on metric measure spaces.
We discuss the parabolic Muckenhoupt weights on spaces of homogeneous type, give a characterization in terms of the
weak-type and the strong-type weighted norm inequalities for parabolic maximal operators,
and prove a reverse H\"older inequality and a Jones-type factorization result.
Moreover, we discuss the parabolic space BMO
and establish its relation with the parabolic Muckenhoupt weights
by using a parabolic John--Nirenberg lemma.
As an application, we obtain a Coifman--Rochberg-type characterization of the parabolic
$\BMO$ in terms of the parabolic maximal functions.

The organization of the remainder of this article is as follows.

In Section~\ref{section2},
we recall the concepts of both a space $(X,\rho,\mu)$ of homogeneous type
and the system of dyadic cubes of $(X,\rho,\mu)$.
Some necessary assumptions,
including the monotone geodesic property are imposed on $(X,\rho,\mu)$
(see Definition~\ref{mgp} and Remark~\ref{remarkmgp}
below for more details and for some examples of spaces of homogeneous type
having this additional property).
In addition, by applying the monotone geodesic property,
we obtain a chaining ball property of $(X,\rho,\mu)$
(see Lemma~\ref{assum} below). We also prove
the Borel regularity of the product space $X\times \rr$
(see Corollary~\ref{2014} below).

In Section~\ref{section3}, we first introduce the concept
of parabolic cylinders of $(X,\rho,\mu)$ and then introduce
the class of parabolic Muckenhoupt $A_q^+(\gamma)$ weights.
Some elementary properties of parabolic Muckenhoupt weights,
such as the nestedness, dual properties, and forward-in-time doubling properties,
are presented (see Proposition~\ref{1627} below).

In Section~\ref{section4}, we
show that, for any $q\in(1,\infty)$, $A_q^+(\gamma)$ is independent of
the choice of the time lag $\gamma$ whenever $\gamma\in(0,1)$ and,
moreover, $A_q^+(0)$ is a subset of $A_q^+(\gamma)$
(see Theorem~\ref{Idp} below).
In this proof, we use the Vitali covering lemma
on a given space $(X,\rho,\mu)$ of homogeneous type
to divide a given ball of $X$
into overlapping balls with smaller radius
and then to transport the new parabolic cylinders
along both the time and the space axes
by using both the chaining ball property and the
forward-in-time doubling properties.

In Section~\ref{section5}, we introduce the concept of
parabolic Muckenhoupt weights on parabolic rectangles,
which is equivalent to that class defined on parabolic cylinders.
Then we introduce the parabolic maximal operators related to parabolic rectangles and parabolic cylinders
and prove that they are pointwise equivalent.
To prove the necessity and the
sufficiency of the parabolic Muckenhoupt
condition for the weak-type weighted boundedness of
parabolic maximal operators,
on the one hand, we use a similar proof of the
classical Muckenhoupt theory on $\rn$ to prove the necessity of
the $A_q^+(\gamma)$ condition (see Theorem~\ref{2132-} below);
on the other hand, we make use of a covering argument
similar to that used in the proof of \cite[Lemma~4.5]{ks2016}
(see also \cite[Theorem~1.2]{f2011})
and geometric properties of the system of dyadic
cubes to show the sufficiency of
the $A_q^+(\gamma)$ condition (see Theorem~\ref{weakineq} below).
These results are discussed for parabolic rectangles and parabolic cylinders.

In Section~\ref{section6},
we first use the dyadic structure of the
system of dyadic cubes of $X$
to construct a Calder\'on--Zygmund-type cover and
establish an estimate for the distribution sets
(see Lemma~\ref{1546} below).
Then we prove the reverse H\"older inequalities for
parabolic Muckenhoupt weights (see Theorems~\ref{RHI}
and~\ref{RHIball} below)
and further obtain a self-improving
property for them (see Theorem~\ref{2135} below).
Using this, the weak-type weighted
boundedness of parabolic maximal operators,
and the Marcinkiewicz interpolation theorem,
we further prove the strong-type weighted
boundedness of parabolic maximal operators
(see Theorem~\ref{opeboun} below).
Moreover, the backward-in-time reverse doubling properties
of parabolic Muckenhoupt weights are showed to be equivalent to the
reverse H\"older inequalities
(see Theorem~\ref{rdp} below).

In Section~\ref{section7}, we first introduce
the $A_1^+(\gamma)$ condition via parabolic maximal functions,
and establish the relation between
$A_1^+(\gamma)$ and $A_q^+(\gamma)$ with $q\in(1,\infty)$
(see Proposition~\ref{73} below).
Then we give a factorization of $A_q^+(\gamma)$ (see Theorem~\ref{74} below)
and use the reverse H\"older inequalities to obtain
a characterization of $A_1^+(\gamma)$ as small powers of maximal functions
multiplied by bounded functions (see Theorem~\ref{76} below).

In Section~\ref{section8}, we introduce the parabolic
space $\PBMO^+(X\times\mathbb{R})$.
Using a John--Nirenberg-type lemma in \cite{a1988},
we establish the relation between
the class $A_q^+(\gamma)$ with $q\in(1,\infty)$
and $\PBMO^+(X\times\mathbb{R})$ (see Theorem~\ref{85} below).
By this, the factorization of $A_q^+(\gamma)$
in Theorem~\ref{74},
and the characterization of $A_1^+(\gamma)$ in Theorem~\ref{76},
we give a characterization
of $\PBMO^+(X\times\mathbb{R})$ in terms of parabolic maximal functions
in the sense of Coifman and Rochberg (see Theorem~\ref{86} below).

\section{Spaces of Homogeneous Type
and Monotone Geodesic Property }\label{section2}

In this section,
we recall the concept of spaces of homogeneous type,
a necessary geometric assumption,
the system of dyadic cubes,
and the adjacent system of dyadic cubes
on spaces of homogeneous type. In addition,
from the monotone geodesic property,
we deduce the chaining ball property of
a given space $X$ of homogeneous type under consideration.
We also establish the Borel regularity of the product
space $X\times \rr$.

The following notation will be used throughout.
We always let $\nn:=\{1,2,\ldots\}$
and $\zz_+:=\nn\cup\{0\}$.
We denote by $C$ a \emph{positive constant} which is independent
of the main parameters, but may vary from line to line.
We use $C_{\az,\dots}$ to denote a positive constant depending
on the indicated parameters $\az,\, \dots$.
The symbol $f\ls g$ means $f\le Cg$
and, if $f\ls g\ls f$, then we write $f\approx g$.
If $f\le Cg$ and $g=h$ or $g\le h$,
we then write $f\ls g\approx h$ or $f\ls g\ls h$,
rather than $f\ls g=h$ or $f\ls g\le h$.
For any $r\in\mathbb{R}$, let $r_+:=\max\{0,\,r\}$.
For any $\az\in\rr$,
we denote by $\lfloor\az\rfloor$ the largest integer
not greater than $\az$
and $\lceil\az\rceil$ the smallest integer
not smaller than $\az$. For any index $q\in[1,\fz]$,
we denote by $q'$ its \emph{conjugate index}, namely, $1/q+1/q'=1$.
For any $x\in X$ and $r\in(0,\fz)$, we denote by $B(x,r)$ the \emph{ball}
of $X$ centered at $x$ with the radius $r$, namely,
$$B(x,r):=\{y\in X:\ \rho(x,y)<r\}.$$
For any ball $B$, we use $x_B$ to denote its center and $r_B$ its radius,
and denote by $\lz B$ for any $\lz\in(0,\fz)$ the ball concentric with
$B$ having the radius $\lz r_B$.
The set $E\subset X$ is said to be
\emph{bounded}
if $E$ is contained in a $\rho$-ball.
The \emph{diameter} of $E$, $\diam\,(E)$,
is defined by setting
\begin{align*}
\text{diam}\,(E):=
\sup\left\{\rho(x,y):\ x,y\in E\right\}.
\end{align*}
If $R$ is a subset of $X\times\mathbb{R}$, we denote by ${\mathbf{1}}_R$ its
\emph{characteristic function},
by $R^\complement$ the set $(X\times\mathbb{R})\setminus R$,
and by $\#R$ the \emph{cardinality} of $R$.
Also, when we prove a lemma,
proposition, theorem, or corollary, we always
use the same symbols in the wanted proved lemma,
proposition, theorem, or corollary.

We recall the concept of a quasi-metric space.

\begin{definition}\label{2057}
A \emph{quasi-metric space} $(X,\rho)$ is
a non-empty set $X$ equipped with a
\emph{quasi-metric} $\rho$, namely, a nonnegative function
defined on $X\times X$ satisfying that, for any $x,y,z\in X$,
\begin{enumerate}
\item[\textup{(i)}]
$\rho(x,y)=0$\text{ if and only of }$x=y$;
\item[\textup{(ii)}]
$\rho(x,y)=\rho(y,x);$
\item[\textup{(iii)}]
there exists a constant $K_0\in[1,\infty)$,
independent of $x$, $y$, and $z$,
such that
\begin{align*}
\rho(x,y)\leq K_0\left[\rho(x,z)+\rho(z,y)\right].
\end{align*}
\end{enumerate}
\end{definition}

\begin{definition}
A triplet $(X,\rho,\mu)$\index{$(X,\rho,\mu)$}
 is called a
\emph{space of homogeneous type}
if $(X,\rho)$ is a quasi-metric space
and $\mu$ is a nonnegative measure, defined on a $\sigma$-algebra
of subsets of $X$ which contains the $\rho$-balls,
satisfying the following
\emph{doubling condition}:
there exists a constant
$C_{\mu}\in[1,\fz)$\index{$C_{\mu}$} such that,
for any $x\in X$ and $r\in(0,\infty)$,
$$
\mu\left(B(x,2r)\right)\le C_{\mu}\mu\left(B(x,r)\right).
$$
\end{definition}

Throughout the article,
we \emph{always make} the following basic assumptions on both $(X,\rho,\mu)$
and $(X\times\mathbb{R},d,\lambda)$,
where, for any $x,y\in X$ and $t,s\in\mathbb{R}$,
\begin{align*}
d\left((x,t),(y,s)\right):=\max\left\{\rho(x,y),\,|t-s|^{1/p}\right\},
\end{align*}
and $\lambda:=\mu\times m$
with $m$ being the one-dimensional Lebesgue measure on $\mathbb{R}$:
\begin{enumerate}
\item[\textup{(i)}]
for any point $x\in X$,
the balls $\{B(x,r)\}_{r\in(0,\infty)}$ form a
basis of open neighborhoods of $x$;
\item[\textup{(ii)}]
$\mu$ is \emph{Borel regular} which means that all open sets of $X$
are measurable and every set $A\subset X$
is contained in a Borel set $E$ such that $\mu(A)=\mu(E)$;
\item[\textup{(iii)}]
for any $x\in X$ and $r\in(0,\infty)$,
$\mu(B(x,r))\in(0,\infty)$;
\item[\textup{(iv)}]
$\diam\,(X)=\infty$ and
$(X,\rho,\mu)$ is non-atomic, which means $\mu(\{x\})=0$
for any $x\in X$.
\end{enumerate}

\begin{remark}
It follows from the doubling condition that,
for any
$\lz\in\left[1,\fz\right)$, $x\in X$, and $r\in(0,\infty)$,
\begin{equation}
\label{upperdoub}
\mu\left(B(x,\lz r)\right)\le
C_{\mu}\lz^n\mu\left(B(x,r)\right),
\end{equation}
where $n:=\log_2C_{\mu}$\index{$n$} is called
the \emph{upper dimension} of $X$.
By (iii) of the above assumptions,
we conclude that $C_{\mu}\in(1,\infty)$
and hence $n\in(0,\infty)$
(see also \cite[p.\,72]{AlMi2015}).
\end{remark}

Next, we introduce the definition
of the monotone geodesic property
on quasi-metric spaces.

\begin{definition}\label{mgp}
Let $(X,\rho)$ be a quasi-metric space.
The space $(X,\rho)$ is said to have the \emph{monotone geodesic property}
if there exists a positive constant $\widetilde{K}$ such that,
for any $s\in(0,\infty)$ and $x,y\in X$ with $\rho(x,y)\ge s$,
there exists a positive integer $m\in\mathbb{N}$
and a finite set $\{x_i\}_{i=0}^m$ of points
with both $x_0:=x$ and $x_m:=y$ such that, for any $i\in\{0,...,m-1\}$,
\begin{enumerate}
\item[\textup{(i)}]
$\rho(x_i,x_{i+1})\leq \widetilde{K}s$;
\item[\textup{(ii)}]
$\rho(x_{i+1},x_m)\leq\rho(x_i,x_m)-s$.
\end{enumerate}
\end{definition}

\begin{remark}\label{remarkmgp}
\begin{enumerate}
\item[\textup{(i)}]
To study the volume of spheres
on doubling metric measure spaces
and on groups of polynomial growth,
the monotone geodesic property on metric spaces
was originally introduced by Tessera in \cite{t2007},
where Tessera in \cite[Theirem 4]{t2007} proved
that this property implies
the F\o lner property for balls (see, for instance,
\cite[p.\,48]{t2007} for the definition of the F\o lner property).
To establish the boundedness
of the Littlewood--Paley operators on
localized $\BMO$ spaces on
doubling metric measure spaces,
a slightly stronger scale invariant version of
this monotone geodesic property was introduced
by Lin et al. \cite[Definition 4.1(II)]{lny2011},
which, in the same article \cite[Theorem 4.1]{lny2011}, was proved
to be equivalent to the chain ball property
introduced by Buckley in \cite{b1999}.
Moreover, this monotone geodesic property
was also used by Auscher and Routin \cite{ar2013}
to ensure the Hardy property (see \cite[Definition~3.4]{ar2013}
for its definition) satisfied by metric measure
spaces of homogeneous type under consideration.
For more equivalent properties and details on
the monotone geodesic property
over doubling metric measure spaces,
we refer the reader to \cite{ar2013,lny2011,t2007}.

\item[\textup{(ii)}]
If $(X,\rho)$
is a metric space, then
Definition~\ref{mgp} coincides with
\cite[Definition 4.1(II)]{lny2011}.
If $(X,\rho)$
is a metric space and $s=1$, then
Definition~\ref{mgp} coincides with the second properties of
\cite[Proposition 2]{t2007}.
\item[\textup{(iii)}]
Obviously, the Euclidean space $\mathbb{R}^n$ equipped
with the standard Euclidean distance has the
monotone geodesic property.
As was pointed out in \cite[p.\,890]{r2013},
this property in Definition~\ref{mgp} is satisfied by complete
doubling Riemannian manifolds and
also satisfied by both geodesics spaces and length
spaces.
\end{enumerate}
\end{remark}

For any quasi-metric space $(X,\rho)$ having
the monotone geodesic property in Definition~\ref{mgp},
we have the following \emph{chaining ball property}.

\begin{lemma}\label{assum}
Let $(X,\rho)$ be a quasi-metric space
satisfying the monotone geodesic property in Definition~\ref{mgp}.
Then
there exists a constant $\Lambda\in(1,\infty)$
such that, for any $x,y\in X$ and $r\in(0,\infty)$,
there exists a positive integer $N(x,y,r)\in\mathbb{N}$,
only depending on $x$, $y$, and $r$,
and a sequence $\{x_i\}_{i=0}^{N(x,y,r)}$ of points
with both $x_0:=x$ and $x_{N(x,y,r)}:=y$ such that
\begin{enumerate}
\item[\textup{(i)}]
for any $i\in\{0,...,N(x,y,r)-1\}$,
there exists a ball $D_i\subset B(x_i,r)\cap B(x_{i+1},r)$
such that
$$B(x_i,r)\cup B(x_{i+1},r)\subset\Lambda D_i;$$
\item[\textup{(ii)}]
for any $x,y\in X$ and $p\in(1,\infty)$,
\begin{align*}
\lim_{r\to0}N(x,y,r)r^p=0.
\end{align*}
\end{enumerate}
\end{lemma}

\begin{proof}
Let $\widetilde{K}\in[1,\infty)$ be
the positive constant in Definition~\ref{mgp}.
Let $x,y\in X$, $r\in(0,\infty)$, and $s:=r/(2K_0\widetilde{K})$.
We consider the following two cases on $s$.

\emph{Case 1)} $\rho(x,y)\ge s$. In this case, by
Definition~\ref{mgp}, we know that
there exists a positive integer $m$ and a finite set $\{x_i\}_{i=0}^m$ of points
with both $x_0:=x$ and $x_m:=y$ such that, for any $i\in\{0,...,m-1\}$,
both (i) and (ii) of Definition~\ref{mgp} hold true.
Let $N(x,y,r):=m$.
Now, for any $i\in\{0,...,m\}$, define $B_i:=B(x_i,r)$.
For any $i\in\{0,...,m-1\}$, define $D_i:=B(x_{i+1},r/(2K_0))$
and $\Lambda:=4K_0^2$. From this and both Definitions~\ref{2057}(iii)
and~\ref{mgp}(i),
we deduce that, for any $z\in D_i$ with $i\in\{0,...,m-1\}$,
\begin{align*}
\rho(x_i,z)\leq K_0\left[\rho(x_i,x_{i+1})+
\rho(x_{i+1},z)\right]<K_0\left[\widetilde{K}s+r/(2K_0)\right]=r,
\end{align*}
which further implies that $D_i\subset B_i$
and hence
\begin{align}\label{i1}
D_i\subset B_i\cap B_{i+1}.
\end{align}
Notice that both Definitions~\ref{2057}(iii) and~\ref{mgp}(i)
imply that,
for any $z\in B_i$ with $i\in\{0,...,m-1\}$,
\begin{align*}
\rho(x_{i+1},z)\leq K_0\left[\rho(x_{i+1},x_i)+\rho(x_i,z)\right]
<K_0(\widetilde{K}s+r)<2K_0r,
\end{align*}
which further implies that $B_i\subset\Lambda D_i$.
By both this and an obvious fact that $B_{i+1}\subset\Lambda D_i$
which is deduced from their definitions, we find that,
for any given $i\in\{0,...,m-1\}$,
\begin{align*}
B_i\cup B_{i+1}\subset\Lambda D_i,
\end{align*}
which, together with \eqref{i1}, then completes the proof of (i).

Next, we show (ii).
To this end, from Definition~\ref{mgp}(ii),
we deduce that
\begin{align*}
0&\leq\rho(x_{m-1},x_m)\leq\rho(x_{m-2},x_m)-s
\leq\cdots\leq\rho(x_0,x_m)-(m-1)s\\
&=\rho(x,y)-ms+s,
\end{align*}
which implies that
\begin{align*}
m\leq\frac{\rho(x,y)}{s}+1.
\end{align*}
By this, we further conclude that, for
any $p\in(1,\infty)$,
\begin{align*}
\lim_{r\to0}N(x,y,r)r^p
&=\lim_{r\to0}mr^p
\leq\lim_{r\to0}\left[\frac{\rho(x,y)}{s}+1\right]r^p\\
&=\lim_{r\to0}\left[\frac{2K_0\widetilde{K}\rho(x,y)}{r}+1\right]r^p
=0,
\end{align*}
which completes the proof of (ii) and hence Lemma~\ref{assum}
in the  case that $\rho(x,y)\ge s$.

\emph{Case 2)} $\rho(x,y)<s$. In this case, let $N(x,y,r):=1$,
$D_0:=B(x,r/(2K_0))$,
and $\Lambda:=4K_0^2$.
Since $N(x,y,r)=1$, it is easy to deduce that (ii) holds true
automatically in this case.
By an argument similar to that used in the
proof of Case 1), we find that (i) also holds true in this case.
This finishes the proof of the case that $\rho(x,y)<s$
and hence Lemma~\ref{assum}.
\end{proof}

In what follows, we \emph{always assume} that
$(X,\rho)$ is a quasi-metric space
having the monotone geodesic property in Definition~\ref{mgp},
which ensures that Lemma~\ref{assum} holds true.

The following system of dyadic cubes of $(X,\rho,\mu)$
was constructed by Hyt\"onen and Kairema in \cite[Theorem 2.2]{hk2012}.

\begin{lemma}\label{1516}
Let $(X,\rho,\mu)$ be a space of homogeneous type.
Suppose that constants $0<c_0\leq C_0<\infty$ and $\delta\in(0,1)$
satisfy $12K_0^3C_0\delta\leq c_0$.
Given a set of points,
$\{z_\alpha^k:\ k\in\mathbb{Z},\,\alpha\in I_k\}\subset X$
with $I_k$, for any $k\in\mathbb{Z}$, being a set of indices,
which has the following properties: for any $k\in\mathbb{Z}$,
\begin{align*}
\rho(z_\alpha^k,z_\beta^k)\ge c_0\delta^k
\ \text{if}\ \alpha\neq\beta,
\ \text{and}\
\min_{\alpha\in I_k}\rho(x,z_\alpha^k)<C_0\delta^k\ \text{for any}
\  x\in X,
\end{align*}
then there exists a family of sets,
$\{Q_\alpha^k:\ k\in\mathbb{Z},\,\alpha\in I_k\}$,
satisfying
\begin{enumerate}
\item[\textup{(i)}]
if $l,k\in\mathbb{Z}$ with $l\ge k$,
then, for any $\alpha\in I_k$ and $\beta\in I_l$,
either $Q_\beta^l\subset Q_\alpha^k$
or $Q_\beta^l\cap Q_\alpha^k=\emptyset$;
\item[\textup{(ii)}]
for any $k\in\mathbb{Z}$,
$X=\bigcup_{\alpha\in I_k}Q_\alpha^k$ is a disjoint union;
\item[\textup{(iii)}]
for any $k\in\mathbb{Z}$ and $\alpha\in I_k$,
\begin{align*}
B(z_\alpha^k,c_1\delta^k)\subset Q_\alpha^k
\subset B(z_\alpha^k,C_1\delta^k)=:B(Q_\alpha^k),
\end{align*}
where $c_1:=(3K_0^2)^{-1}c_0$ and $C_1:=2K_0C_0$;
\item[\textup{(iv)}]
if $l,k\in\mathbb{Z}$ with $l\ge k$, and both
$\alpha\in I_k$ and $\beta\in I_l$ satisfy
$Q_\beta^l\subset Q_\alpha^k$,
then $B(Q_\beta^l)\subset B(Q_\alpha^k)$.
\end{enumerate}
\end{lemma}

Recall that a quasi-metric space $(X,\rho)$ is said to be \emph{geometric
doubling} if there exists a positive integer $A_0\in\mathbb{N}$
such that, for any $x\in X$ and $r\in(0,\infty)$,
the ball $B(x,r)$ can be covered by at most $A_0$ balls
$B(x_i,r/2)$. It is well known that any space of homogeneous type
is also a geometric doubling quasi-metric space (see also \cite[p.\,53]{AlMi2015}).
The following adjacent system of dyadic cubes of $(X,\rho,\mu)$
was established by Hyt\"onen and Kairema in \cite[Theorem 4.1]{hk2012}.

\begin{lemma}\label{1517}
Let $(X,\rho,\mu)$ be a space of homogeneous type.
Suppose that a constant $\delta\in(0,1)$ satisfies $96K_0^6\delta\leq1$.
Then there exists a positive integer $K:=K(K_0,A_0,\delta)$,
a countable set of points,
$$\left\{z_\alpha^{k,\tau}:\ \tau\in\{1,...,K\},\,k\in\mathbb{Z},\,
\alpha\in I_{k,\tau}\right\}$$
with $I_{k,\tau}$, for any $\tau\in\{1,...,K\}$ and $k\in\mathbb{Z}$,
being a set of indices,
and a finite number of dyadic grids,
$$\left\{\mathfrak{D}^\tau:=\left\{Q_\alpha^{k,\tau}:\ k\in\mathbb{Z},\,
\alpha\in I_{k,\tau}\right\}:\ \tau\in\{1,...,K\}\right\},$$
where, for any $\tau\in\{1,...,K\}$,
$\mathfrak{D}^\tau$ is a collection of dyadic cubes,
associated to dyadic points
$\{z_\alpha^{k,\tau}:\ k\in\mathbb{Z},\,\alpha\in I_{k,\tau}\}$,
satisfying Lemma~\ref{1516}.
In addition, there exists a positive constant $C$,
only depending on both $K_0$ and $\delta$,
such that, for any $x\in X$ and $r\in(0,\infty)$,
there exists a $\tau\in\{1,...,K\}$ and a dyadic cube
$Q_\alpha^{k,\tau}\in\mathfrak{D}^\tau$ with both $k\in\mathbb{Z}$
and $\alpha\in I_{k,\tau}$
such that
\begin{align}\label{1721}
B(x,r)\subset Q_\alpha^{k,\tau}
\quad\text{and}\quad
\diam\,(Q_\alpha^{k,\tau})\leq Cr.
\end{align}
\end{lemma}

The following proposition is very useful, which might be well known;
we present the details here for the convenience of the reader.

\begin{proposition}\label{borelregular}
Let $(X,\rho_1,\mu)$ and $(Y,\rho_2,\nu)$ be two spaces of homogeneous type,
$\tau_X$ and $\tau_Y$ be, respectively,
the topology on $X$  induced by $\rho_1$
and the topology on $Y$ induced by $\rho_2$,
and $\mathscr{A}$ and $\mathscr{B}$,
respectively, the $\sigma$-algebra of $\mu$-measurable sets
and the $\sigma$-algebra of $\nu$-measurable sets.
If both $\mu$ and $\nu$ are Borel regular,
then the product measure $\lambda:=\mu\times\nu$ of the
product space $(X\times Y,\tau_{X\times Y},\lambda)$ is also Borel regular,
where $\tau_{X\times Y}$ denotes the product topology generated by
the topology basis $\tau_X\times\tau_Y$.
\end{proposition}

To prove Proposition~\ref{borelregular},
we need the following lemma.

\begin{lemma}\label{1949}
Any space of homogeneous type is separable and
has a countable topological basis.
\end{lemma}

\begin{proof}
Let $(X,\rho,\mu)$ be a space of homogeneous type.
We first prove that $(X,\rho,\mu)$ is separable.
To this end, it is suffices to show that the set
$Z:=\{z_\alpha^k\}_{k\in\mathbb{Z},\,\alpha\in I_k}$
of dyadic points in Lemma~\ref{1516} is a countable dense subset of $X$.
Fix an $x_0\in X$. On the one hand,
by both (i) and (iii) of Lemma~\ref{1516}
and the fact that
$(X,\rho)$ is geometric doubling, we find that,
for any $k\in\mathbb{Z}$ and $N\in\mathbb{N}$,
the cardinality of the set
\begin{align*}
Z_{k,N}:=\left\{z\in Z:\ B(z,c_1\delta^k)\subset B(x_0,N)\right\}
\end{align*}
is finite, where both $c_1$ and $\delta$ are the same as in Lemma~\ref{1516}.
From this and both (ii) and (iii) of Lemma~\ref{1516}, we further deduce that
\begin{align*}
Z=\bigcup_{k\in\mathbb{Z}}\bigcup_{N\in\mathbb{N}}Z_{k,N}
\end{align*}
is a countable subset of $X$.
On the other hand,
by Lemma~\ref{1516}(iii), we find that,
for any $x\in X$ and $\varepsilon\in(0,\infty)$,
there exists a $k\in\mathbb{Z}$ satisfying that $C_1\delta^k<\varepsilon$
with $C_1$ in Lemma~\ref{1516},
and an $\alpha\in I_k$ such that
$x\in Q_\alpha^k\subset B(z_\alpha^k,C_1\delta^k)$
and hence $\rho(x,z_\alpha^k)<C_1\delta^k<\varepsilon$.
This further implies that $Z$ is dense in $X$.
Thus, $Z$ is a countable dense subset of $X$.

Next, we show that $(X,\rho,\mu)$, with its topology
induced by $\rho$, has a countable topological
basis. Notice that, from an argument similar to that
used in the proof of
\cite[Theorem 2]{MaSe79i}, we deduce that
there exists a $\theta\in(0,1)$ and a metric $\varrho$
such that $\varrho\approx\rho^\theta$
and hence the topologies induced, respectively, by
$\varrho$ and $\rho$ coincide.
By this, $X$ is separable,
and \cite[p.\,204, Proposition 25]{r2010},
we conclude that the metric space $(X,\varrho)$
has a countable topological
basis and hence so does $(X,\rho)$.
This finishes the proof of Lemma~\ref{1949}.
\end{proof}

\begin{proof}[Proof of Proposition~\ref{borelregular}]
To show that $\lambda$ is Borel regular, we first prove that,
for any open subset $R$ of $X\times Y$,
$R$ is $\lambda$-measurable.
Let $\mathscr{A}\otimes\mathscr{B}$ be the $\sigma$-algebra
generated by the Cartesian product $\mathscr{A}\times\mathscr{B}$.
Since both $\mu$ and $\nu$ are Borel regular, it follows that $\tau_X\subset\mathscr{A}$ and $\tau_Y\subset\mathscr{B}$.
By this, the definitions of both $\tau_{X\times Y}$
and $\mathscr{A}\otimes\mathscr{B}$,
and Lemma~\ref{1949},
it is easy to show that,
for any open subset $R$ of $X\times Y$,
$$
R\in\tau_{X\times Y}\subset
\left(\mathscr{A}\otimes\mathscr{B}\right).
$$
From this and the fact that
all the elements of $\mathscr{A}\otimes\mathscr{B}$
are $\lambda$-measurable,
we deduce that, for any $R\in\tau_{X\times Y}$,
$R$ is $\lambda$-measurable.

Next, we show that, for any set $A\subset X\times Y$,
there exists a Borel subset $F$ of $X\times Y$
such that $A\subset F$
and $\lambda(A)=\lambda(F)$.
Recall that the product measure $\lambda$ on
the product space $X\times Y$ is defined by setting,
for any subset $A$ of $X\times Y$,
\begin{align*}
\lambda(A):=\inf\left\{\sum_{j\in\mathbb{N}}\lambda(P_j):\
P_j\in\mathscr{R},\ A\subset\bigcup_{j\in\mathbb{N}} P_j\right\},
\end{align*}
where $\mathscr{R}$ is defined to be the set of all the
subsets $P$ of $X\times Y$ satisfying that there exist
$P^X\in\mathscr{A}$ and $P^{Y}\in\mathscr{B}$ such that
$P=P^X\times P^{Y}$.
Let $A\subset X\times Y$.
We consider the following two cases on $\lambda(A)$.

\emph{Case 1)} $\lambda(A)=\infty$. In this case,
it is easy to show that both $A\subset X\times Y$
and $\infty=\lambda(A)\leq\lambda(X\times Y)$,
which further implies that $\lambda(A)=\lambda(X\times Y)$.
Notice that $X\times Y$ is obviously a Borel set.
This finishes the proof of Case 1).

\emph{Case 2)} $\lambda(A)<\infty$. In this case,
by the definition of
$\lambda(A)$, we know that, for any $n\in\mathbb{N}$,
there exist $\{A_{n,j}^X:\ A_{n,j}^X\subset X\}_{j\in\mathbb{N}}$ and
$\{A^{Y}_{n,j}:\ A^{Y}_{n,j}
\subset Y\}_{j\in\mathbb{N}}$ such that
\begin{align}\label{1939}
A\subset\bigcup_{j\in\mathbb{N}}\left(A^X_{n,j}
\times A^{Y}_{n,j}\right)=:F_n
\end{align}
and
\begin{align}\label{1907}
0\leq\lambda(F_n)-\lambda(A)<\frac{1}{n}.
\end{align}
Since $\mu$ is Borel regular on $X$, it follows that,
for any $n,j\in\mathbb{N}$,
there exists a Borel subset $\widetilde{A}^X_{n,j}$ of $X$ such that
\begin{align}\label{1906}
A^X_{n,j}\subset\widetilde{A}^X_{n,j}
\quad\text{and}\quad
\mu\left(A^X_{n,j}\right)=\mu\left(\widetilde{A}^X_{n,j}\right).
\end{align}
Similarly, for any $n,j\in\mathbb{N}$,
we can find a Borel subset
$\widetilde{A}^{Y}_{n,j}$
of $Y$ such that
\begin{align}\label{2015}
A^Y_{n,j}\subset\widetilde{A}^Y_{n,j}
\quad\text{and}\quad
\nu\left(A^Y_{n,j}\right)=\nu\left(\widetilde{A}^Y_{n,j}\right).
\end{align}
For any $n\in\mathbb{N}$, we let
\begin{align*}
\widetilde{F}_n:=\bigcup_{j\in\mathbb{N}}
\left(\widetilde{A}^X_{n,j}
\times\widetilde{A}^{Y}_{n,j}\right)
\quad\text{and}\quad
F:=\lim_{k\to\infty}\bigcap_{n\ge k}\widetilde{F}_n.
\end{align*}
On the one hand, by the definitions of both $F_n$ and $\widetilde{F}_n$,
\eqref{1906}, \eqref{2015}, and \eqref{1907}, we conclude that,
for any $n\in\mathbb{N}$,
\begin{align*}
0&\leq\lambda(\widetilde{F}_n)-\lambda(A)
=\sum_{j\in\mathbb{N}}\mu\left(\widetilde{A}^X_{n,j}\right)
\nu\left(\widetilde{A}^{Y}_{n,j}\right)-\lambda(A)
\\
&=\sum_{j\in\mathbb{N}}\mu\left(A^X_{n,j}\right)
\nu\left(A^{Y}_{n,j}\right)-\lambda(A)
=\lambda(F_n)-\lambda(A)<\frac{1}{n},
\end{align*}
which, together with both the definition of $F$ and \eqref{1939},
further implies that $\lambda(F)=\lambda(A)$
and $A\subset F$.
On the other hand, since $\widetilde{A}^X_{n,j}$
is a Borel subset of $X$
and $\widetilde{A}^{Y}_{n,j}$ is a Borel subset of $Y$,
it follows that $\widetilde{A}^X_{n,j}\times \widetilde{A}^{Y}_{n,j}$
is a Borel subset of $X\times Y$ and hence so is $\widetilde{F}_n$.
This, combined with the definition of $F$, further implies that
$F$ is a Borel subset of $X\times Y$.
This finishes the proof of Case 2)
and hence Proposition~\ref{borelregular}.
\end{proof}

Using Proposition~\ref{borelregular},
we immediately obtain the following conclusion
which is used to prove Lemma \ref{1546} below;
we omit the details here.

\begin{corollary}\label{2014}
Let $(X,\rho,\mu)$ be a
space of homogeneous type and
$m$ the one-dimensional Lebesgue measure on $\mathbb{R}$.
Then the product measure $\lambda:=\mu\times m$
is Borel regular on the product measure space
$(X\times\mathbb{R},\lambda)$.
\end{corollary}

\section{Parabolic Muckenhoupt
Weights\\ on Spaces of Homogeneous Type}\label{section3}

Let $(X,\rho,\mu)$ be a space of homogeneous type
having the monotone geodesic property.
In this section, we introduce the
parabolic Muckenhoupt weight class on $X\times\mathbb{R}$
and present some related basic properties.
We say that $X$ is \emph{spatial} and $\mathbb{R}$
\emph{temporal} related to the product space $X\times\mathbb{R}$.
For any given set $E\subset X\times\mathbb{R}$
and any $a\in\mathbb{R}$,
the \emph{temporal translation $E+a$} of $E$
is defined by setting
\begin{align*}
E+a:=\left\{(x,t+a):\ (x,t)\in E\right\}.
\end{align*}
To introduce the definition of the parabolic Muckenhoupt weight class,
we first introduce the following concept of parabolic space-time
cylinders. Throughout this article, we \emph{always fix}
a parameter $p\in(1,\infty)$.

\begin{definition}\label{3.1}
Let $(X,\rho,\mu)$ be a space of homogeneous type
and $\gamma\in[0,1)$.
For any $x\in X$, $t\in\mathbb{R}$, and $l\in(0,\infty)$,
let
$$R(x,t,l):=B(x,l)\times(t-l^p,t+l^p),$$
$$R^-(x,t,l,\gamma):=B(x,l)\times(t-l^p,t-\gamma l^p),$$
and
$$R^+(x,t,l,\gamma):
=R^-(x,t,l,\gamma)+(1+\gamma)l^p
=B(x,l)\times(t+\gamma l^p,t+l^p).$$
The set $R(x,t,l)$ is called an \emph{$(x,t)$-centered parabolic
cylinder} with edge length $l$ of $X\times\mathbb{R}$.

Simply write $R:=R(x,t,l)$,
$R^-(\gamma):=R^-(x,t,l,\gamma)$,
$R^+(\gamma):=R^+(x,t,l,\gamma)$,
$$R^-:=R^-(x,t,l,0),$$
and $R^+:=R^+(x,t,l,0)$ if there exists no confusion.
Moreover, $R^{\pm}(\gamma)$ are called, respectively,
the \emph{upper} and the \emph{lower parts} of $R$, and $\gamma$
is called the \emph{time lag}.
\end{definition}

Recall that, for any $q\in(0,\infty)$,
the \emph{Lebesgue space}
$L^q(X\times\mathbb{R})$
is defined to be the set of
all the $\lambda$-measurable functions $f$
on $X\times\mathbb{R}$ such that
\begin{align*}
\|f\|_{L^q(X\times\mathbb{R})}:=
\left[\int_{X\times\mathbb{R}}|f(x,t)|^q\,d\mu(x)\,dt
\right]^{1/q}<\fz.
\end{align*}
For any $q\in(0,\fz)$,
the set of all the locally $q$-integrable
functions on $X\times\mathbb{R}$,
$L_\loc^q(X\times\mathbb{R})$,
is defined to be the set of all the
$\lambda$-measurable functions
$f:\ X\times\mathbb{R}\to\mathbb{C}$ satisfying that,
for any $(x,t)\in X\times\mathbb{R}$,
there exists an $r\in(0,\infty)$ such that
\begin{align*}
\left\|f\mathbf{1}_{R(x,t,r)}\right\|_{L^q(X\times\mathbb{R})}
<\infty.
\end{align*}
A \emph{weight} $\omega$ always means a real-valued
$\lambda$-almost everywhere positive locally integrable function
on $X\times\mathbb{R}$. For any
measurable set $E\subset X\times\mathbb{R}$,
we let
\begin{align*}
\omega(E):=\int_E\omega
\end{align*}
and, for any $f\in L^1_\loc(X\times\mathbb{R})$,
\begin{align*}
f_E:=\fint_Ef
:=\frac{1}{\lambda(E)}\int_Ef.
\end{align*}
Here and thereafter, we \emph{always omit} the differential $d\mu(x)\,dt$
in all integral representations to simplify the
presentation if there exists no confusion.

We now introduce the parabolic Muckenhoupt weight
class via parabolic cylinders as follows.

\begin{definition}\label{defballpmw}
Let $(X,\rho,\mu)$ be a space of homogeneous type,
$q\in(1,\infty)$, and $\gamma\in[0,1)$.
\begin{enumerate}
\item[\textup{(i)}]
A weight $\omega$ is said to belong
to the \emph{parabolic Muckenhoupt weight class via parabolic cylinders},
$A^+_q(\gamma)$, if
\begin{align}\label{A+}
[\omega]_{A^+_q(\gamma)}:=
\sup_{\{R(x,t,l):\ (x,t)\in X\times\mathbb{R},\,
l\in(0,\infty)\}}\left[\fint_{R^-(\gamma)}\omega\right]
\left[\fint_{R^+(\gamma)}\omega^{1-q'}\right]^{q-1}<\infty.
\end{align}
\item[\textup{(ii)}]
A weight $\omega$ is said to belong
to the \emph{parabolic Muckenhoupt weight class via parabolic cylinders},
$A^-_q(\gamma)$, if
\begin{align}\label{A-}
[\omega]_{A^-_q(\gamma)}:=
\sup_{\{R(x,t,l):\ (x,t)\in X\times\mathbb{R},\,
l\in(0,\infty)\}}\left[\fint_{R^+(\gamma)}\omega\right]
\left[\fint_{R^-(\gamma)}\omega^{1-q'}\right]^{q-1}<\infty,
\end{align}
\end{enumerate}
where $q'\in(1,\infty)$ denotes the conjugate index of $q$, namely,
$1/q+1/q'=1$, and the suprema in both \eqref{A+} and \eqref{A-}
are taken over all the parabolic cylinders $R(x,t,l)$ of $X\times\mathbb{R}$ with
both $(x,t)\in X\times\mathbb{R}$ and $l\in(0,\infty)$.
\end{definition}

\begin{remark}
If $X=\mathbb{R}^n$, $$\rho(x,y):=
\max\left\{2|x_i-y_i|:\ i\in\{1,...,n\}\right\}$$
for any $x:=(x_1,...,x_n),y:=(y_1,...,y_n)\in\mathbb{R}^n$,
and
$\mu=\mathfrak{L}^n$ is the Lebesgue measure on $\mathbb{R}^n$, then,
in this case,
Definition~\ref{defballpmw} coincides with \cite[Definition 3.2]{ks2016}.
\end{remark}

Next, we present some elementary and useful properties
of parabolic Muckenhoupt weights;
since their proofs are similar to those of
\cite[Propositions 3.3 and 3.4]{ks2016}, we omit the details here.

\begin{proposition}\label{1627}
Let $(X,\rho,\mu)$ be a space of homogeneous type,
$q\in(1,\infty)$, and $\gamma\in[0,1)$.
\begin{enumerate}
\item[\textup{(i)}]
Assume that $u,v\in A^+_q(\gamma)$.
Then both $f:=\min\{u,\,v\}$ and $g:=\max\{u,\,v\}$
belong to $A^+_q(\gamma)$.
Moreover, there exists a positive constant $C_q$ such that
\begin{align*}
[f]_{A^+_q(\gamma)}\leq C_q
\left([u]_{A^+_q(\gamma)}+[v]_{A^+_q(\gamma)}\right)
\ \text{and}\
[g]_{A^+_q(\gamma)}\leq C_q
\left([u]_{A^+_q(\gamma)}+[v]_{A^+_q(\gamma)}\right).
\end{align*}
\item[\textup{(ii)}]
For any $r\in[q,\infty)$, $A^+_q(\gamma)\subset A^+_r(\gamma)$.
\item[\textup{(iii)}]
$\omega\in A^+_q(\gamma)$
if and only if $\omega^{1-q'}\in A^-_{q'}(\gamma)$.
\item[\textup{(iv)}]
Let $\omega\in A^+_q(\gamma)$. Then,
for any parabolic cylinder $R(x,t,l)$
with both $(x,t)\in X\times\mathbb{R}$
and $l\in(0,\infty)$,
any $S\subset R^+(\gamma)$,
and any $P\subset R^-(\gamma)$,
\begin{align}\label{1826}
\omega(R^-(\gamma))\leq
[\omega]_{A^+_q(\gamma)}
\left[\frac{\lambda(R^-(\gamma))}{\lambda(S)}\right]^q\omega(S)
\end{align}
and
\begin{align}\label{1827}
\omega^{1-q'}(R^+(\gamma))\leq
[\omega]^{q'-1}_{A^+_q(\gamma)}
\left[\frac{\lambda(R^+(\gamma))}
{\lambda(P)}\right]^{q'}\omega^{1-q'}(P).
\end{align}
\end{enumerate}
\end{proposition}

\section{Independent of the Choice of the Time Lag}
\label{section4}

Let $(X,\rho,\mu)$ be a space of
homogeneous type
having the monotone geodesic property.
In this section, we establish the relation between the
parabolic Muckenhoupt weight class with different time lag.
More precisely, we aim to prove the following theorem.

\begin{theorem}\label{Idp}
Let $(X,\rho,\mu)$ be a space of homogeneous type,
$q\in(1,\infty)$, and $\gamma,\gamma'\in[0,1)$.
\begin{enumerate}
\item[\textup{(i)}]
If $0\leq\gamma\leq\gamma'<1$,
then $A^+_q(\gamma)\subset A^+_q(\gamma')$.
\item[\textup{(ii)}]
If $0<\gamma'<\gamma<1$, then $A^+_q(\gamma)\subset A^+_q(\gamma')$.
\end{enumerate}
\end{theorem}

To prove Theorem~\ref{Idp}, we need the following proposition.

\begin{proposition}\label{1633}
Let $(X,\rho,\mu)$ be a space of homogeneous type,
$q\in(1,\infty)$, and $\gamma\in[0,1)$.
Assume that $\omega\in A^+_q(\gamma)$ and $\theta\in(0,\infty)$. Then,
for any parabolic cylinder $R(x,t,l)$
with both $(x,t)\in X\times\mathbb{R}$
and $l\in(0,\infty)$,
\begin{align}\label{312-1}
\fint_{R^-(\gamma)}\omega\leq
[\omega]_{A^+_q(\gamma)}^{\lceil\theta\rceil}
\fint_{R^-(\gamma)+\theta(1+\gamma)l^p}\omega
\end{align}
and
\begin{align}\label{312-2}
\fint_{R^+(\gamma)}\omega^{1-q'}\leq
[\omega]_{A^+_q(\gamma)}^{\lceil\theta\rceil}
\fint_{R^+(\gamma)-\theta(1+\gamma)l^p}\omega^{1-q'}.
\end{align}
\end{proposition}

To prove Proposition~\ref{1633},
we need the following
Vitali covering theorem which is just \cite[Lemma~2.1]{a1988}
and can also be found in \cite{CW1} with
a slightly different statement.

\begin{lemma}\label{winner}
Let $(X,\rho,\mu)$ be a space of homogeneous type.
Let
$$\mathfrak{B}:=\left\{B_\alpha:=B(x_\alpha,r_\alpha):\
\alpha\in\Gamma\right\}$$
be a family of balls in $X$
such that $\bigcup_{\alpha\in\Gamma}B_\alpha$
is bounded.
Then there exists a sequence
$\{B_i\}_i\subset\mathfrak{B}$
of disjoint balls
such that, for any $\alpha\in\Gamma$,
there exists an $i$ satisfying that $r_\alpha\leq 2r_i$
and
$$B_\alpha\subset B(x_i,5K_0^2r_i).$$
\end{lemma}

\begin{remark}\label{winner1}
In Lemma~\ref{winner}, since $\mu$ is doubling,
we find that, if, for any $\alpha,\beta\in\Gamma$,
$r_{\alpha}\approx r_{\beta}$
uniformly, then $\{B_i\}_i$ is a finite set.
Moreover, if there exists a positive constant $r$
such that, for any $\alpha\in\Gamma$,
$r_{\alpha}\approx r\approx\diam\,(\bigcup_{\alpha\in\Gamma}B_\alpha)$
uniformly, then $\{B_i\}_i$ is a finite set with its cardinality
independent of $r$.
\end{remark}

\begin{proof}[Proof of Proposition~\ref{1633}]
Let $R:=R(x,t,l)$, with both
$(x,t)\in X\times\mathbb{R}$
and $l\in(0,\infty)$, be a parabolic cylinder.
We first prove \eqref{312-1}. To this end,
we consider the following three cases on $\theta$.

\emph{Case 1)} $\theta=1$. In this case,
it is easy to show that
$R^-(\gamma)+(1+\gamma)l^p=R^+(\gamma)$.
From \eqref{A+}, the fact that $\varphi(s):=s^{1-q}$ for any $s\in(0,\infty)$
with $q\in(1,\infty)$ is a convex function,
and the Jensen inequality, we deduce that
\begin{align*}
\fint_{R^-(\gamma)}\omega
&\leq[\omega]_{A^+_q(\gamma)}
\left[\fint_{R^+(\gamma)}\omega^{1-q'}\right]^{1-q}
\leq[\omega]_{A^+_q(\gamma)}
\fint_{R^+(\gamma)}\omega\\
&=[\omega]_{A^+_q(\gamma)}
\fint_{R^-(\gamma)+(1+\gamma)l^p}\omega.
\end{align*}
This finishes the proof of \eqref{312-1}
in the case that $\theta=1$.

\emph{Case 2)} $\theta\in(0,1)$. In this case,
$R^-(\gamma)+\theta(1+\gamma)l^p\subset R(x,t,l)$.
Let
$$
\mathfrak{B}:=\left\{
B\left(y,(5K_0^2)^{-1}\theta^{1/p}l\right):\ y\in B(x,l)\right\}.
$$
By Definition~\ref{2057}(iii), we find that,
for any given $y\in B(x,l)$ and for any
$z\in B(y,(5K_0^2)^{-1}\theta^{1/p}l)$,
\begin{align*}
\rho(x,z)\leq K_0\left[\rho(x,y)+\rho(y,z)\right]
<K_0l+(5K_0)^{-1}\theta^{1/p}l,
\end{align*}
which further implies that
$B(y,(5K_0^2)^{-1}\theta^{1/p}l)\subset
B(x,K_0l+(5K_0)^{-1}\theta^{1/p}l)$
and hence
\begin{align}\label{1526}
\bigcup_{B\in\mathfrak{B}}B\subset
B\left(x,K_0l+(5K_0)^{-1}\theta^{1/p}l\right)
\end{align}
is a bounded set. From this, Lemma~\ref{winner}, and Remark~\ref{winner1},
it follows that there exists a positive integer $N$
and a sequence
$\{B_i:=B(y_i,(5K_0^2)^{-1}\theta^{1/p}l)\}_{i=1}^N\subset\mathfrak{B}$
of disjoint balls such that
\begin{align}\label{1603}
B(x,l)\subset\bigcup_{i=1}^N 5K_0^2B_i
\end{align}
and, for any $\tau\in\{1,...,N\}$,
\begin{align}\label{1124}
\mu\left(B(x,l)\setminus\left[\bigcup_{i=1}^N 5K_0^2B_i\setminus
5K_0^2B_\tau\right]\right)>0.
\end{align}
Now, we let
$M:=\lceil\theta^{-1}\rceil\in[2,\infty)$
and, for any $j\in\{1,...,M\}$,
\begin{align*}
T_j:=\left(t-l^p+(j-1)\theta(1-\gamma)l^p,
t-l^p+j\theta(1-\gamma)l^p\right).
\end{align*}
From this, it follows that
\begin{align}\label{1604}
\bigcup_{j=1}^MT_j
=(t-l^p,t-l^p+M\theta(1-\gamma)l^p)
\supset(t-l^p,t-\gamma l^p)
\end{align}
in the sense that it is allowed to differ
by a set of zero measure of $\mathbb{R}$.
For any $i\in\{1,...,N\}$ and $j\in\{1,...,M\}$,
let
\begin{align*}
R^-_{i,j}(\gamma):=(5K_0^2B_i)\times T_j
\quad\text{and}\quad
R^+_{i,j}(\gamma):=(5K_0^2B_i)\times T_j+\theta(1+\gamma)l^p,
\end{align*}
which, combined with both \eqref{1603} and \eqref{1604}, implies that
\begin{align}\label{1035}
R^-(\gamma)\subset\bigcup_{i=1}^N\bigcup_{j=1}^MR^-_{i,j}(\gamma)
\quad\text{and}\quad
R^-(\gamma)+\theta(1+\gamma)l^p\subset
\bigcup_{i=1}^N\bigcup_{j=1}^MR^+_{i,j}(\gamma)
\end{align}
in the sense that it is allowed to differ
by a set of zero measure of $X\times\mathbb{R}$.

Next, we reconstruct $\{R^+_{i,j}(\gamma)\}_{i\in\{1,...,N\}
,\,j\in\{1,...,M\}}$ such
that the union in the second formula of \eqref{1035} is a disjoint union.
To this end, for any $j\in\{1,...,M\}$, let
\begin{align*}
\widetilde{R}^+_{1,j}(\gamma):=R^+_{1,j}(\gamma)
\end{align*}
and, for any $i\in\{2,...,N\}$, let
\begin{align*}
\widetilde{R}^+_{i,j}(\gamma):=R^+_{i,j}(\gamma)\setminus
\bigcup_{\tau=1}^{i-1}R^+_{\tau,j}(\gamma).
\end{align*}
This implies that, for any $i,\tau\in\{1,...,N\}$ and $j,k\in\{1,...,M\}$
with $i\neq\tau$ or $j\neq k$,
\begin{align}\label{1109}
\widetilde{R}^+_{i,j}(\gamma)\cap\widetilde{R}^+_{\tau,k}(\gamma)=\emptyset
\quad\text{and}\quad
\bigcup_{i=1}^N\bigcup_{j=1}^M\widetilde{R}^+_{i,j}(\gamma)
=\bigcup_{i=1}^N\bigcup_{j=1}^MR^+_{i,j}(\gamma).
\end{align}
Moreover, \eqref{1124} and the definitions of both
$\widetilde{R}^+_{i,j}(\gamma)$
and ${R}^+_{i,j}(\gamma)$ imply that, for any
$i\in\{1,...,N\}$ and $j\in\{1,...,M\}$,
\begin{align}\label{1129}
\lambda\left(\widetilde{R}^+_{i,j}(\gamma)
\cap[R^-(\gamma)+\theta(1+\gamma)l^p]\right)>0.
\end{align}
Notice that, for any $i\in\{1,...,N\}$ and $j\in\{1,...,M\}$,
\begin{align*}
\lambda\left(R^-_{i,j}(\gamma)\right)
=\mu(B_i)\times\theta(1-\gamma)l^p
=\lambda\left(R^+_{i,j}(\gamma)\right)
\end{align*}
and
\begin{align*}
\lambda\left(R^-(\gamma)\right)
=\mu(B(x,l))\times(1-\gamma)l^p
=\lambda\left(R^-(\gamma)+\theta(1+\gamma)l^p\right).
\end{align*}
From this, \eqref{1035}, $\omega\in A^+_q(\gamma)$,
the fact that $f(s):=s^{1-q}$ for any $s\in(0,\infty)$
is a decreasing convex function with
$q\in(1,\infty)$, \eqref{1129}, the Jensen inequality, and \eqref{1109},
we further deduce that
\begin{align*}
\fint_{R^-(\gamma)}\omega
&\leq\sum_{i=1}^N\sum_{j=1}^M
\frac{\lambda(R^-_{i,j}(\gamma))}{\lambda(R^-(\gamma))}
\fint_{R^-_{i,j}(\gamma)}\omega\\
&\leq[\omega]_{A^+_q(\gamma)}\sum_{i=1}^N\sum_{j=1}^M
\frac{\lambda(R^-_{i,j}(\gamma))}{\lambda(R^-(\gamma))}
\left[\fint_{R^+_{i,j}(\gamma)}\omega^{1-q'}\right]^{1-q}\\
&\leq[\omega]_{A^+_q(\gamma)}\sum_{i=1}^N\sum_{j=1}^M
\frac{\lambda(R^-_{i,j}(\gamma))}{\lambda(R^-(\gamma))}
\left[\frac{1}{\lambda(R^+_{i,j}(\gamma))}
\int_{\widetilde{R}^+_{i,j}(\gamma)
\cap[R^-(\gamma)+\theta(1+\gamma)l^p]}\omega^{1-q'}\right]^{1-q}\\
&\leq[\omega]_{A^+_q(\gamma)}\sum_{i=1}^N\sum_{j=1}^M
\frac{\lambda(R^-_{i,j}(\gamma))}{\lambda(R^-(\gamma))}
\frac{1}{\lambda(R^+_{i,j}(\gamma))}
\int_{\widetilde{R}^+_{i,j}(\gamma)
\cap[R^-(\gamma)+\theta(1+\gamma)l^p]}\omega\\
&=[\omega]_{A^+_q(\gamma)}
\frac{1}{\lambda(R^-(\gamma)+\theta(1+\gamma)l^p)}
\sum_{i=1}^N\sum_{j=1}^M
\int_{\widetilde{R}^+_{i,j}(\gamma)
\cap[R^-(\gamma)+\theta(1+\gamma)l^p]}\omega\\
&=[\omega]_{A^+_q(\gamma)}
\fint_{R^-(\gamma)+\theta(1+\gamma)l^p}\omega.
\end{align*}
This finishes the proof of \eqref{312-1}
in the case that $\theta\in(0,1)$.

\emph{Case 3)} $\theta\in(1,\infty)$. In this case,
let $a:=\lfloor\theta\rfloor$. If $\theta=a$, then,
from case 1), we deduce that
\begin{align}\label{2100}
\fint_{R^-(\gamma)}\omega
&\leq[\omega]_{A^+_q(\gamma)}
\fint_{R^-(\gamma)+(1+\gamma)l^p}\omega
\leq\cdots
\leq[\omega]_{A^+_q(\gamma)}^a
\fint_{R^-(\gamma)+a(1+\gamma)l^p}\omega.
\end{align}
If $\theta>a$, then,
by both cases 1) and 2), we conclude that
\begin{align*}
\fint_{R^-(\gamma)}\omega
&\leq[\omega]_{A^+_q(\gamma)}
\fint_{R^-(\gamma)+(1+\gamma)l^p}\omega
\leq\cdots
\leq[\omega]_{A^+_q(\gamma)}^a
\fint_{R^-(\gamma)+a(1+\gamma)l^p}\omega\\
&\leq[\omega]_{A^+_q(\gamma)}^{a+1}
\fint_{R^-(\gamma)+\theta(1+\gamma)l^p}\omega,
\end{align*}
which, together with \eqref{2100},
then completes the proof of \eqref{312-1} in the case
that $\theta\in(1,\infty)$.

In summary, we finish the proof of \eqref{312-1}.
By Proposition~\ref{1627}(iii), \eqref{A-},
and an argument similar to that used in the above proof
of \eqref{312-1}, we conclude that \eqref{312-2} also
holds true.
This finishes the proof of Proposition~\ref{1633}.
\end{proof}

Using Proposition~\ref{1633}, we immediately obtain the following
conclusion; we omit the details here.

\begin{corollary}\label{1451}
Let $(X,\rho,\mu)$ be a space of homogeneous type,
$q\in(1,\infty)$, and $\gamma\in[0,1)$.
Assume that $\omega\in A^+_q(\gamma)$ and
$\theta_1,\theta_2\in(0,\infty)$.
Then
\begin{align*}
&\sup_{\{R(x,t,l):\ (x,t)\in X\times\mathbb{R},\,l\in(0,\infty)\}}
\left[\fint_{R^-(\gamma)-\theta_1(1+\gamma)l^p}\omega\right]
\left[\fint_{R^+(\gamma)+\theta_2(1+\gamma)l^p}\omega^{1-q'}\right]^{q-1}\\
&\quad\leq
[\omega]_{A^+_q(\gamma)}^{\lceil\theta_1\rceil+\lceil\theta_2\rceil+1}.
\end{align*}
\end{corollary}

Now, we prove Theorem~\ref{Idp}.

\begin{proof}[Proof of Theorem~\ref{Idp}]
We first prove (i). Let $\omega\in A^+_q(\gamma)$
and $R:=R(x,t,l)$, with both
$(x,t)\in X\times\mathbb{R}$ and $l\in(0,\infty)$,
be a parabolic cylinder.
Since $\gamma'>\gamma$,
it follows that $R^-(\gamma')\subset R^-(\gamma)$
and $R^+(\gamma')\subset R^+(\gamma)$. By this,
$q\in(1,\infty)$, and \eqref{A+}, we find that
\begin{align*}
\left[\fint_{R^-(\gamma')}\omega\right]
\left[\fint_{R^+(\gamma')}\omega^{1-q'}\right]^{q-1}
&\leq\left(\frac{1-\gamma}{1-\gamma'}\right)^q
\left[\fint_{R^-(\gamma)}\omega\right]
\left[\fint_{R^+(\gamma)}\omega^{1-q'}\right]^{q-1}\\
&\leq\left(\frac{1-\gamma}{1-\gamma'}\right)^q
[\omega]_{A^+_q(\gamma)}<\infty
\end{align*}
and hence $\omega\in A^+_q(\gamma')$.
This finishes the proof of (i).

Next, we prove (ii). To this end, we first claim that,
for any $\gamma\in(0,1)$, $A^+_q(\gamma)\subset A^+_q(\gamma/2)$.
Assume for the moment that the claim holds true.
For any $\gamma'\in(0,\gamma)$, we find that
there exists a $K\in\mathbb{N}$
such that $\gamma/2^K\leq\gamma'$,
which, combined with both the claim and (i) of the present theorem,
further implies that
\begin{align*}
A^+_q(\gamma)\subset A^+_q(\gamma/2)\subset\cdots\subset A^+_q(\gamma/2^K)
\subset A^+_q(\gamma').
\end{align*}
To complete the proof of (ii) of the present theorem,
it remains to prove the above claim. Assume that $\omega\in A^+_q(\gamma)$.
Let $R:=R(x,t,l)$, with both
$(x,t)\in X\times\mathbb{R}$ and $l\in(0,\infty)$,
be a parabolic cylinder.
We divide this proof into the following three steps.

\emph{Step 1}. Division process.
Fix a sufficiently small positive constant $k\in(0,1)$
which is determined later.
Let
$$
\mathfrak{B}:=\left\{B\left(y,(5K_0^2)^{-1}kl\right):\ y\in B(x,l)\right\}.
$$
Using an argument similar to that used in the proof of
\eqref{1526}, we find that
$\bigcup_{B\in\mathfrak{B}}B$
is bounded. From this, Lemma~\ref{winner}, and Remark~\ref{winner1},
it follows that there exists a positive integer $N$, depending on $k$,
and a sequence
$\{B(y_i,(5K_0^2)^{-1}kl)\}_{i=1}^N\subset\mathfrak{B}$
of disjoint balls such that
\begin{align}\label{15593}
\bigcup_{i=1}^NB(y_i,kl)\supset B(x,l).
\end{align}
Let $M:=\lceil(2-\gamma)/[2(1-\gamma)k^p]\rceil$.
For any $j\in\{1,...,M\}$, we define
\begin{align*}
T^-_j:=\left(t-\gamma l^p/2-j(1-\gamma)k^pl^p,
t-\gamma l^p/2-(j-1)(1-\gamma)k^pl^p\right)
\end{align*}
and
\begin{align*}
T^+_j:=\left(t+\gamma l^p/2+(j-1)(1-\gamma)k^pl^p,
t+\gamma l^p/2+j(1-\gamma)k^pl^p\right).
\end{align*}
By this, it is easy to show that
\begin{align}\label{15591}
\bigcup_{j=1}^MT^-_j
=\left(t-\gamma l^p/2-M(1-\gamma)k^pl^p,t-\gamma l^p/2\right)
\supset(t-l^p,t-\gamma l^p/2)
\end{align}
and
\begin{align}\label{15592}
\bigcup_{j=1}^MT^+_j
=\left(t+\gamma l^p/2,
t+\gamma l^p/2+M(1-\gamma)k^pl^p\right)
\supset(t+\gamma l^p/2,t+l^p)
\end{align}
in the sense that it is allowed to differ
by a set of zero measure of $\mathbb{R}$.
For any $i\in\{1,...,N\}$ and $j\in\{1,...,M\}$,
let
\begin{align*}
R^-_{i,j}(\gamma):=B(y_i,kl)\times T^-_j
\quad\text{and}\quad
R^+_{i,j}(\gamma):=B(y_i,kl)\times T^+_j.
\end{align*}
From this, \eqref{15593}, \eqref{15591}, and \eqref{15592},
we deduce that
\begin{align*}
R^-(\gamma/2)\subset
\bigcup_{i=1}^N\bigcup_{j=1}^MR^-_{i,j}(\gamma)
\quad\text{and}\quad
R^+(\gamma/2)\subset
\bigcup_{i=1}^N\bigcup_{j=1}^MR^+_{i,j}(\gamma)
\end{align*}
in the sense that it is allowed to differ
by a set of zero measure of $X\times\mathbb{R}$.

\emph{Step 2}. Translation process.
In this step, we show how each element of
$\{R^\pm_{i,j}(\gamma)\}_{i\in\{1,...,N\},\,j\in\{1,...,M\}}$
can be transported to the same spatially central position
$B(x,kl)$ without losing too much information about their measures.
By Lemma~\ref{assum}, we
find that,
for any given $i\in\{1,...,N\}$,
there exist two constants $K_i\in\mathbb{N}$
and $\Lambda\in(1,\infty)$, $\{x_v\}_{v=0}^{K_i}\subset X$,
and $\{D_v\}_{v=0}^{K_i-1}$
satisfying both (i) and (ii) of Lemma~\ref{assum}
with $x_0:=y_i$, $x_{K_i}:=x$, and $r:=kl$.
From this, \eqref{1826} with $R^-(\gamma)$ and
$S$ replaced, respectively, by $B(x_0,kl)\times T^-_j$
and $D_0\times T^-_j+(1+\gamma)k^pl^p$,
and \eqref{upperdoub},
we deduce that, for any $j\in\{1,...,M\}$,
\begin{align*}
\omega\left(R^-_{i,j}(\gamma)\right)
&=\omega\left(B(x_0,kl)\times T^-_j\right)\\
&\lesssim\left[\frac{\lambda(B(x_0,kl)\times T^-_j)}
{\lambda(D_0\times T^-_j+(1+\gamma)k^pl^p)}\right]^q
\omega\left(D_0\times T^-_j+(1+\gamma)k^pl^p\right)\\
&\approx\left[\frac{\mu(B(x_0,kl))}
{\mu(D_0)}\right]^q
\omega\left(D_0\times T^-_j+(1+\gamma)k^pl^p\right)\\
&\lesssim\left[\frac{\mu(\Lambda D_0)}
{\mu(D_0)}\right]^q
\omega\left(D_0\times T^-_j+(1+\gamma)k^pl^p\right)\\
&\lesssim\omega\left(B(x_1,kl)\times T^-_j+(1+\gamma)k^pl^p\right),
\end{align*}
where the implicit positive constants
only depend on $n$, $q$, and $\Lambda$.
Repeating this argument for
$\{B(x_v,kl)\times T^-_j+v(1+\gamma)k^pl^p\}_{v=1}^{K_i-1}$,
we further obtain,
for any $i\in\{1,...,N\}$ and $j\in\{1,...,M\}$,
\begin{align}\label{1737}
\omega\left(R^-_{i,j}(\gamma)\right)
&\lesssim\omega\left(B(x_1,kl)\times T^-_j+(1+\gamma)k^pl^p\right)\\
&\lesssim\cdots
\lesssim
\omega\left(B(x_{K_i},kl)\times T^-_j+(1+\gamma)K_ik^pl^p\right)
\noz\\
&\approx\omega\left(B(x,kl)\times T^-_j+(1+\gamma)K_ik^pl^p\right).\noz
\end{align}
By both \eqref{1827}
and an argument similar to that used in the estimation
of \eqref{1737} with $\omega$ and $R^-_{i,j}(\gamma)$ replaced, respectively,
by $\omega^{1-q'}$ and $R^+_{i,j}(\gamma)$, we find that,
for any $i\in\{1,...,N\}$ and $j\in\{1,...,M\}$,
\begin{align}\label{1738}
\omega^{1-q'}\left(R^+_{i,j}(\gamma)\right)
\lesssim
\omega^{1-q'}\left(B(x,kl)\times T^+_j-(1+\gamma)K_ik^pl^p\right).
\end{align}
Now, for any $i\in\{1,...,N\}$ and $j\in\{1,...,M\}$,
we let
$$
\widetilde{R}^-_{i,j}(\gamma):=
B(x,kl)\times T^-_j+(1+\gamma)K_ik^pl^p
$$
and
$$
\widetilde{R}^+_{i,j}(\gamma):=
B(x,kl)\times T^+_j-(1+\gamma)K_ik^pl^p.
$$
From this and \eqref{upperdoub}, we deduce that,
for any $i\in\{1,...,N\}$ and $j\in\{1,...,M\}$,
\begin{align}\label{1751}
\frac{\lambda(\widetilde{R}^-_{i,j}(\gamma))}
{\lambda(R^-_{i,j}(\gamma))}
=\frac{\mu(B(x,kl))}{\mu(B(y_i,kl))}
\leq\frac{\mu(B(y_i,K_0(kl+l)))}{\mu(B(y_i,kl))}
\lesssim1,
\end{align}
where, in the penultimate inequality,
we used Definition~\ref{2057}(iii) and the fact that $\rho(x,y_i)<l$.
Similarly, we have
\begin{align}\label{1752}
\frac{\lambda(\widetilde{R}^+_{i,j}(\gamma))}
{\lambda(R^+_{i,j}(\gamma))}\lesssim1,
\end{align}
where the implicit positive constants are independent of
$x$, $t$, and $l$.
By Lemma~\ref{assum}(ii), we choose $k$ sufficiently small such that
\begin{align*}
\min_{i\in\{1,...,N\}}\left\{\gamma l^p-2(1+\gamma)K_ik^pl^p\right\}
>2\gamma k^pl^p,
\end{align*}
which, together with the definitions of
both $\widetilde{R}^-_{i,j}(\gamma)$
and $\widetilde{R}^+_{\tau,\sigma}(\gamma)$,
further implies that, for any $i,\tau\in\{1,...,N\}$
and $j,\sigma\in\{1,...,M\}$,
\begin{align}\label{1739}
&\min_{(x,t_1)\in\widetilde{R}^-_{i,j}(\gamma),\,
(x,t_2)\in\widetilde{R}^+_{\tau,\sigma}(\gamma)}
\left\{t_2-t_1\right\}\\
&\quad>t+\gamma l^p/2
-\max_{i\in\{1,...,N\}}\left\{(1+\gamma)K_ik^pl^p\right\}
\noz\\
&\qquad-\left(t-\gamma l^p/2+
\max_{i\in\{1,...,N\}}\left\{(1+\gamma)K_ik^pl^p\right\}\right)
\noz\\
&\quad>2\gamma k^pl^p.
\noz
\end{align}
Moreover,
from the definitions of both $\widetilde{R}^-_{i,j}(\gamma)$ and
$\widetilde{R}^+_{\tau,\sigma}(\gamma)$,
the choice of $k$,
and the fact that $M<(2-\gamma)/[(1-\gamma)k^p]$,
it follows that,
for any $i,\tau\in\{1,...,N\}$
and $j,\sigma\in\{1,...,M\}$,
\begin{align}
\label{1740}
&\max_{(x,t_1)\in\widetilde{R}^-_{i,j}(\gamma), \,
(x,t_2)\in\widetilde{R}^+_{\tau,\sigma}(\gamma)}
\left\{t_2-t_1\right\}\\
&\quad<t+\gamma l^p/2+M(1-\gamma)k^pl^p
-\left[t-\gamma l^p/2-M(1-\gamma)k^pl^p\right]
\noz\\
&\quad=\gamma l^p+2M(1-\gamma)k^pl^p<(4-\gamma)l^p.
\noz
\end{align}

\emph{Step 3}. In this step, we aim to
prove that $\omega\in A^+_q(\gamma/2)$.
By Step 1 and \eqref{upperdoub}, we find that
\begin{align*}
\fint_{R^-(\gamma/2)}\omega
&\leq\sum_{i=1}^N\sum_{j=1}^M
\frac{\lambda(R^-_{i,j}(\gamma))}{\lambda(R^-(\gamma/2))}
\fint_{R^-_{i,j}(\gamma)}\omega\\
&\leq\sum_{i=1}^N\sum_{j=1}^M
\frac{(1-\gamma)k^pl^p\mu(B(y_i,kl))}{(1-\gamma/2)l^p\mu(B(x,l))}
\fint_{R^-_{i,j}(\gamma)}\omega\\
&\leq\frac{(1-\gamma)k^p}{1-\gamma/2}
\sum_{i=1}^N\sum_{j=1}^M
\frac{\mu(B(x,(k+1)K_0l))}{\mu(B(x,l))}
\fint_{R^-_{i,j}(\gamma)}\omega\\
&\lesssim\sum_{i=1}^N\sum_{j=1}^M
\fint_{R^-_{i,j}(\gamma)}\omega
\end{align*}
and, similarly,
\begin{align*}
\left[\fint_{R^+(\gamma/2)}\omega^{1-q'}\right]^{q-1}
&\leq\left[\sum_{i=1}^N\sum_{j=1}^M
\frac{\lambda(R^+_{i,j}(\gamma))}{\lambda(R^+(\gamma/2))}
\fint_{R^+_{i,j}(\gamma)}\omega^{1-q'}\right]^{q-1}\\
&\lesssim\sum_{i=1}^N\sum_{j=1}^M\left[
\fint_{R^+_{i,j}(\gamma)}\omega^{1-q'}\right]^{q-1},
\end{align*}
where the implicit positive constants are independent of
$x$, $t$, and $l$.
From this, \eqref{1751}, \eqref{1752}, \eqref{1737}, \eqref{1738},
and Corollary~\ref{1451} combined with both
\eqref{1739} and \eqref{1740},
we deduce that
\begin{align*}
&\left[\fint_{R^-(\gamma/2)}\omega\right]
\left[\fint_{R^+(\gamma/2)}\omega^{1-q'}\right]^{q-1}\\
&\quad\lesssim\left[\sum_{i=1}^N\sum_{j=1}^M
\fint_{R^-_{i,j}(\gamma)}\omega\right]
\sum_{\tau=1}^N\sum_{\sigma=1}^M\left[
\fint_{R^+_{\tau,\sigma}(\gamma)}\omega^{1-q'}\right]^{q-1}\\
&\quad\approx\sum_{i=1}^N\sum_{j=1}^M
\sum_{\tau=1}^N\sum_{\sigma=1}^M
\left[\fint_{R^-_{i,j}(\gamma)}\omega\right]\left[
\fint_{R^+_{\tau,\sigma}(\gamma)}\omega^{1-q'}\right]^{q-1}\\
&\quad\lesssim\sum_{i=1}^N\sum_{j=1}^M
\sum_{\tau=1}^N\sum_{\sigma=1}^M
\left[\frac{\lambda(\widetilde{R}^-_{i,j}(\gamma))}
{\lambda(R^-_{i,j}(\gamma))}
\fint_{\widetilde{R}^-_{i,j}(\gamma)}\omega\right]\left[
\frac{\lambda(\widetilde{R}^+_{\tau,\sigma}(\gamma))}
{\lambda(R^+_{\tau,\sigma}(\gamma))}
\fint_{\widetilde{R}^+_{\tau,\sigma}(\gamma)}\omega^{1-q'}\right]^{q-1}\\
&\quad\lesssim\sum_{i=1}^N\sum_{j=1}^M
\sum_{\tau=1}^N\sum_{\sigma=1}^M
\left[\fint_{\widetilde{R}^-_{i,j}(\gamma)}\omega\right]\left[
\fint_{\widetilde{R}^+_{\tau,\sigma}(\gamma)}\omega^{1-q'}\right]^{q-1}
\lesssim1,
\end{align*}
where the implicit positive constants are independent of
$x$, $t$, and $l$.
By this and the arbitrariness of the parabolic cylinder $R(x,t,l)$,
we further conclude that $\omega\in A^+_q(\gamma/2)$.

Combining these three steps, we finish the proof of the above claim,
which then completes the proof of (ii) and hence Theorem~\ref{Idp}.
\end{proof}

\begin{remark}\label{1511}
\begin{enumerate}
\item[\textup{(i)}]
By Theorem~\ref{Idp}, we find that,
for any $q\in(1,\infty)$, $A^+_q(\gamma)$
is independent of the choice of $\gamma$ whenever $\gamma\in(0,1)$.
\item[\textup{(ii)}]
Notice that Lemma~\ref{assum}, which follows from
the monotone geodesic property,
is only used to prove Theorem~\ref{Idp}.
If $(X,\rho,\mu)$ is a
space of homogeneous type satisfying (i)
and (ii) of Lemma~\ref{assum} instead of the
monotone geodesic property, but with Lemma~\ref{assum}(ii)
replaced by
\begin{enumerate}
\item[\textup{(a)}]
there exists a constant $p_0\in(0,\infty)$,
only depending on $(X,\rho)$, such that,
for any $x,y\in X$ and $p\in(p_0,\infty)$,
\begin{align*}
\lim_{r\to0}N(x,y,r)r^p=0,
\end{align*}
\end{enumerate}
and define $p_{X}$ to be the smallest $p_0$
satisfied by $(X,\rho)$ in (a),
then we can choose any $p\in(p_X,\infty)$ to define
parabolic cylinders in Definition~\ref{3.1}
and all the results of this article still hold true.
It can be seen that whether (i) of the present remark
holds true depends on both
the parameter $p$ in Definition~\ref{3.1}
and the geometric properties of $(X,\rho)$ reflected by $p_X$.
For instance, let $(X,\rho)=(\mathbb{R}^n,|\cdot-\cdot|)$,
where $|\cdot-\cdot|$ is the standard Euclidean distance.
Then, in this case,
$p_X=1$ and we can use any $p\in(1,\infty)$ to define the
parabolic cylinders to develop a corresponding
theory. This interesting phenomenon has already been
mentioned in \cite[p.\,1723]{ks2016}.

\item[\textup{(iii)}]
Let $(X,\rho,\mu)$ be a space of homogeneous type,
$q\in(1,\infty)$, and $\gamma\in[0,1)$.
Assume that $\omega\in A^+_q(\gamma)$ and $\theta_1,\theta_2\in\mathbb{R}$
satisfy $\theta_1+\theta_2>-\frac{2\gamma}{1+\gamma}$.
By both Corollary~\ref{1451} and
an argument similar to that used in the proof of Theorem~\ref{Idp}(ii),
we find that there exists a positive constant $C$
such that
\begin{align}\label{1510}
&\sup_{\{R(x,t,l):\ (x,t)\in X\times\mathbb{R},\,l\in(0,\infty)\}}
\left[\fint_{R^-(\gamma)-\theta_1(1+\gamma)l^p}\omega\right]
\left[\fint_{R^+(\gamma)+\theta_2(1+\gamma)l^p}\omega^{1-q'}\right]^{q-1}
\leq C.
\end{align}
Moreover, assuming that \eqref{1510} holds true,
it is possible to show that $\omega\in A^+_q(\gamma)$ by
an argument similar to that used in the proof of \eqref{1510}.
This implies that we can replace \eqref{A+} by \eqref{1510}
in the definition of the parabolic Muckenhoupt weight class.
The new condition is satisfied with a different
constant which depends on $q$, $\gamma$,
$[\omega]_{A^+_q(\gamma)}$, $\theta_1$, and $\theta_2$.
\end{enumerate}
\end{remark}

\section{Parabolic Maximal Operators\\ and Parabolic
Weighted Norm Inequalities}\label{section5}

In this section, we introduce the
parabolic maximal functions
and show the weak-type parabolic
weighted norm inequalities.
To this end, we first introduce the parabolic space-time
rectangles as follows.

\begin{definition}\label{defrectpmw}
Let $(X,\rho,\mu)$ be a space of homogeneous type,
$p\in(1,\infty)$, and $\gamma\in[0,1)$.
Assume that
$$
\left\{\mathfrak{D}^\tau:=\left\{Q_\alpha^{k,\tau}:\
k\in\mathbb{Z},\,\alpha\in I_{k,\tau}\right\}
\right\}_{\tau=1}^K
$$
is the adjacent system of dyadic cubes and $\delta\in(0,1)$
in Lemma~\ref{1517}.
For any $\tau\in\{1,...,K\}$, $k\in\mathbb{Z}$, $\alpha\in I_{k,\tau}$,
and $t\in\mathbb{R}$,
let
$$
P(\tau,k,\alpha,t):=
Q_\alpha^{k,\tau}\times(t-\delta^{kp},t+\delta^{kp}),
$$
$$
P^-(\tau,k,\alpha,t,\gamma):=
Q_\alpha^{k,\tau}\times(t-\delta^{kp},t-\gamma\delta^{kp}),
$$
and
$$
P^+(\tau,k,\alpha,t,\gamma):=P^-(\tau,k,\alpha,t,\gamma)
+(1+\gamma)\delta^{kp}=
Q_\alpha^{k,\tau}\times(t+\gamma\delta^{kp},t+\delta^{kp}).
$$
The set $P(\tau,k,\alpha,t)$ is called
a \emph{$(z_\alpha^{k,\tau},t)$-centered parabolic
rectangle} of $X\times\mathbb{R}$
with edge length $l(P(\tau,k,\alpha,t)):=\delta^k$,
where $z_\alpha^{k,\tau}$ is the same as
in Lemma~\ref{1516} related to the system $\mathfrak{D}^\tau$
of dyadic cubes.

Simply write $P:=P(\tau,k,\alpha,t)$,
$P^-(\gamma):=P^-(\tau,k,\alpha,t,\gamma)$,
$P^+(\gamma):=P^+(\tau,k,\alpha,t,\gamma)$,
$$P^-:=P^-(\tau,k,\alpha,t,0),$$
and $P^+:=P^+(\tau,k,\alpha,t,0)$ if there exists no confusion.
Then $P^{\pm}(\gamma)$ are called, respectively,
the \emph{upper} and the \emph{lower parts} of $P$.
\end{definition}

We now introduce the parabolic Muckenhoupt weight
class via parabolic rectangles as follows.

\begin{definition}\label{1618}
Let $(X,\rho,\mu)$ be a space of homogeneous type,
$q\in(1,\infty)$, and $\gamma\in[0,1)$.
\begin{enumerate}
\item[\textup{(i)}]
A weight $\omega$ is said to belong
to the \emph{parabolic Muckenhoupt weight class via parabolic rectangles},
$\widetilde{A}^+_q(\gamma)$, if
\begin{align}\label{Ar+}
[\omega]_{\widetilde{A}^+_q(\gamma)}:=
\sup_{P(\tau,k,\alpha,t)}\left[\fint_{P^-(\gamma)}\omega\right]
\left[\fint_{P^+(\gamma)}\omega^{1-q'}\right]^{q-1}<\infty.
\end{align}
\item[\textup{(ii)}]
A weight $\omega$ is said to belong
to the \emph{parabolic Muckenhoupt weight class via parabolic rectangles},
 $\widetilde{A}^-_q(\gamma)$, if
\begin{align}\label{Ar-}
[\omega]_{\widetilde{A}^-_q(\gamma)}:=
\sup_{P(\tau,k,\alpha,t)}\left[\fint_{P^+(\gamma)}\omega\right]
\left[\fint_{P^-(\gamma)}\omega^{1-q'}\right]^{q-1}<\infty,
\end{align}
\end{enumerate}
where $q'\in(1,\infty)$ is the conjugate index of $q$, namely,
$1/q+1/q'=1$, and the suprema in both \eqref{Ar+} and \eqref{Ar-}
are taken over all the parabolic rectangles
$P(\tau,k,\alpha,t)$ of $X\times\mathbb{R}$ with
$\tau\in\{1,...,K\}$, $k\in\mathbb{Z}$, $\alpha\in I_{k,\tau}$,
and $t\in\mathbb{R}$.
\end{definition}

\begin{remark}\label{5.3}
\begin{enumerate}
\item[\textup{(i)}]
Let $q\in(1,\infty)$ and $\gamma\in[0,1$).
Then (i), (ii), and (iii) of Proposition~\ref{1627} still hold true
for $\widetilde{A}^+_q(\gamma)$.
Proposition~\ref{1627}(iv) also holds true
for any $\omega\in\widetilde{A}^+_q(\gamma)$
with parabolic cylinders replaced by parabolic rectangles.
\item[\textup{(ii)}]
From both an argument similar to that
used in the proof of Proposition~\ref{1633}
and the forward-in-time doubling
property of $\omega\in\widetilde{A}^+_q(\gamma)$
[see Proposition~\ref{1627}(iv) with parabolic
cylinders replaced by parabolic rectangles],
we deduce that Proposition~\ref{1633}
also holds true
with parabolic cylinders replaced by parabolic rectangles.
By this and an argument similar to that used in the proof
of Theorem~\ref{Idp}, we find that
Theorem~\ref{Idp} also holds true for
$\widetilde{A}^+_q(\gamma)$.
In addition, we conclude that $\widetilde{A}^+_q(\gamma)$
is independent of the choice of $\gamma$ whenever $\gamma\in(0,1)$.
\item[\textup{(iii)}]
Part (ii) of the present remark implies that
Remark~\ref{1511}(iii) also holds true
for any $\omega\in\widetilde{A}^+_q(\gamma)$
with parabolic cylinders replaced by parabolic rectangles.
\item[\textup{(iv)}]
By Definitions~\ref{defballpmw} and~\ref{defrectpmw} and
Lemmas~\ref{1516} and~\ref{1517},
it is easy to show that, for any $q\in(1,\infty)$ and
$\gamma_1,\gamma_2,\gamma\in[0,1)$ satisfying that
\begin{align*}
0\leq\gamma_1\leq\min\left\{1,\,\frac{1}{(2K_0C_1)^p}\right\}\gamma_2\leq
\min\left\{\left(\frac{c_1}{C}\right)^p,
\left(\frac{c_1\delta}{2K_0C_1C}\right)^p\right\}\gamma<1
\end{align*}
with both $c_1$ and $C_1$ in Lemma~\ref{1516}
and both $\delta$ and $C$ in Lemma~\ref{1517},
\begin{align*}
A^+_q(\gamma_1)\subset\widetilde{A}^+_q(\gamma_2)\subset A^+_q(\gamma),
\end{align*}
which, together with both Remark~\ref{1511}(i) and (ii)
of the present remark,
further implies that,
for any $q\in(1,\infty)$ and $\gamma,\gamma'\in(0,1)$,
\begin{align*}
\widetilde{A}^+_q(0)=A^+_q(0)\subset
A^+_q(\gamma)=\widetilde{A}^+_q(\gamma').
\end{align*}
Thus, the definitions, respectively,
via parabolic cylinders and parabolic rectangles
of $A^+_q(\gamma)$ and $\widetilde{A}^+_q(\gamma)$ coincide.
\end{enumerate}
\end{remark}

Next, we introduce the definition
of parabolic maximal operators
related to parabolic rectangles
as follows.

\begin{definition}\label{M^gamma}
Let $(X,\rho,\mu)$ be a space of homogeneous type
and $\gamma\in[0,1)$. Assume that $f\in L^1_{\loc}(X\times\mathbb{R})$.
The \emph{parabolic maximal functions},
related to parabolic rectangles, $\mathcal{M}^{\gamma+}f$
and $\mathcal{M}^{\gamma-}f$ of $f$ are defined, respectively,
by setting, for any $(x,t)\in X\times\mathbb{R}$,
\begin{align*}
\mathcal{M}^{\gamma+}f(x,t):=\sup_{\substack{P(\tau,k,\alpha,t)\ni (x,t)\\
\tau\in\{1,...,K\},\,k\in\mathbb{Z},\,\alpha\in I_{k,\tau}}}
\fint_{P^+(\gamma)}\left|f\right|
\end{align*}
and
\begin{align*}
\mathcal{M}^{\gamma-}f(x,t):=\sup_{\substack{P(\tau,k,\alpha,t)\ni (x,t)\\
\tau\in\{1,...,K\},\,k\in\mathbb{Z},\,\alpha\in I_{k,\tau}}}
\fint_{P^-(\gamma)}\left|f\right|.
\end{align*}
Moreover, simply write $\mathcal{M}^{+}:=\mathcal{M}^{0+}$
and $\mathcal{M}^{-}:=\mathcal{M}^{0-}$.
\end{definition}

Let $q\in(0,\infty)$ and $\omega$ be a weight.
Recall that the \emph{weighted Lebesgue space}
$L^q(X\times\mathbb{R},\omega)$
is defined to be the set of
all the $\lambda$-measurable functions $f$
on $X\times\mathbb{R}$ such that
\begin{align*}
\|f\|_{L^q(X\times\mathbb{R},\omega)}:=
\left(\int_{X\times\mathbb{R}}|f|^q\omega
\right)^{1/q}<\fz.
\end{align*}
The \emph{weighted weak Lebesgue space}
$L^{q,\infty}(X\times\mathbb{R},\omega)$
is defined to be the set of
all the $\lambda$-measurable functions $f$
on $X\times\mathbb{R}$ such that
\begin{align*}
\|f\|_{L^{q,\infty}(X\times\mathbb{R},\omega)}:=
\sup_{\xi\in(0,\infty)}\xi
\left[\omega\left(\left\{(x,t)\in X\times\mathbb{R}:\
|f(x,t)|>\xi\right\}\right)\right]^{1/q}<\fz.
\end{align*}
Especially, when $\omega\equiv1$, $L^{q,\infty}(X\times\mathbb{R},1)$ is just
the \emph{weak Lebesgue space} and we simply write
$$L^{q,\infty}(X\times\mathbb{R}):=L^{q,\infty}(X\times\mathbb{R},1).$$

Now, we show the necessity of the $A_q^+(\gamma)$ condition.

\begin{theorem}\label{2132-}
Let $(X,\rho,\mu)$ be a space of homogeneous type,
$q\in(1,\infty)$, and $\gamma\in[0,1)$.
Assume that $\omega$ is a weight such that
the parabolic maximal operator
$$\mathcal{M}^{\gamma+}:\ L^q(X\times\mathbb{R},\omega)
\to L^{q,\infty}(X\times\mathbb{R},\omega)$$
is bounded. Then
$\omega\in A^+_q(\gamma)$.
\end{theorem}

\begin{proof}
Without loss of generality, we may assume that
$f\in L^q(X\times\mathbb{R},\omega)$
is a nonnegative function and $P:=P(\tau,k,\alpha,t)$
a parabolic rectangle with
$\tau\in\{1,...,K\}$, $k\in\mathbb{Z}$, $\alpha\in I_{k,\tau}$,
and $t\in\mathbb{R}$.
Let $\widetilde{k}:=k+\lfloor\log_{\delta}2^{1/p}\rfloor<k$,
where $\delta\in(0,1)$ is the same as in Lemma~\ref{1517}.
By this and Lemma~\ref{1516}, we find that there exists
a unique $\beta\in I_{\widetilde{k},\tau}$ such that
$Q_\alpha^{k,\tau}\subset Q_\beta^{\widetilde{k},\tau}$,
which implies that,
for any $(y,s)\in P^-(\tau,k,\alpha,t,\gamma)=:P^-(\gamma)$,
$(y,s)\in P(\tau,\widetilde{k},\beta,s)$
and
\begin{align*}
S^+(\gamma):=P^+(\gamma)+
\gamma\left(\delta^{\widetilde{k}p}-2\delta^{kp}\right)
\subset P^+\left(\tau,\widetilde{k},\beta,s,\gamma\right),
\end{align*}
where $P^+(\gamma):=P^+(\tau,k,\alpha,t,\gamma)$.
From this, both (iii) and (iv) of Lemma~\ref{1516}, Definition~\ref{2057}(iii),
and \eqref{upperdoub},
we deduce that, for any $(y,s)\in P^-(\gamma)$,
\begin{align*}
\mathcal{M}^{\gamma+}f(y,s)
&\ge \fint_{P^+(\tau,\widetilde{k},\beta,s,\gamma)}f\ge
\frac{\lambda(S^+(\gamma))}{\lambda(P^+(\tau,\widetilde{k},\beta,s,\gamma))}
f_{S^+(\gamma)}\\
&=\frac{\mu(Q_\alpha^{k,\tau})\delta^{kp}}
{\mu(Q_\beta^{\widetilde{k},\tau})\delta^{\widetilde{k}p}}f_{S^+(\gamma)}
\ge\frac{\delta^p\mu(B(z_\alpha^{k,\tau},c_1\delta^k))}
{2\mu(B(z_\beta^{\widetilde{k},\tau},C_1\delta^{\widetilde{k}}))}f_{S^+(\gamma)}\\
&\ge\frac{\delta^p\mu(B(z_\alpha^{k,\tau},c_1\delta^k))}
{2\mu(B(z_\alpha^{k,\tau},2K_0C_1\delta^{\widetilde{k}}))}f_{S^+(\gamma)}\\
&\ge\frac{\delta^{p}}{2C_{\mu}}
\left(\frac{2K_0C_12^{1/p}}{c_1\delta}\right)^{-n}f_{S^+(\gamma)},
\end{align*}
and hence
\begin{align*}
P^-(\gamma)\subset
\left\{(x,t)\in X\times\mathbb{R}:\
\mathcal{M}^{\gamma+}f(x,t)>\frac{\delta^{p}}{2C_{\mu}}
\left(\frac{2K_0C_12^{1/p}}{c_1\delta}\right)^{-n}f_{S^+(\gamma)}\right\}.
\end{align*}
By this and the boundedness of $\mathcal{M}^{\gamma+}$
from $L^q(X\times\mathbb{R},\omega)$ to
$L^{q,\infty}(X\times\mathbb{R},\omega)$,
we find that, for any $\xi\in(0,\frac{\delta^{p}}{2C_{\mu}}
[\frac{2K_0C_12^{1/p}}{c_1\delta}]^{-n}f_{S^+(\gamma)})$,
\begin{align}\label{2029}
\omega(P^-(\gamma))=\int_{P^-(\gamma)}\omega\leq
\int_{\{(x,t)\in X\times\mathbb{R}:\
\mathcal{M}^{\gamma+}f(x,t)>\xi\}}\omega
\lesssim\frac{1}{\xi^q}\int_{X\times\mathbb{R}}
f^q\omega.
\end{align}
Let $f:=\omega^{1-q'}\mathbf{1}_{S^+(\gamma)}$ and
$\xi\to\frac{\delta^{p}}{2C_{\mu}}
(\frac{2K_0C_12^{1/p}}{c_1\delta})^{-n}f_{S^+(\gamma)}$.
Then, by \eqref{2029} and the fact that
$\lambda(S^+(\gamma))=\lambda(P^+(\gamma))
=\lambda(P^-(\gamma))$,
we conclude that
\begin{align}\label{1046}
\fint_{P^-(\gamma)}\omega
&\lesssim\frac{1}{\lambda(S^+(\gamma))}
\left[\fint_{S^+(\gamma)}\omega^{1-q'}\right]^{-q}
\left[\int_{S^+(\gamma)}
\omega^{(1-q')q+1}\right]\\
&\approx
\left[\fint_{S^+(\gamma)}\omega^{1-q'}\right]^{1-q}.\noz
\end{align}
On the one hand, if $\gamma=0$, then $S^+(0)=P^+$.
This, combined with \eqref{1046}, implies that
\begin{align*}
\left(\fint_{P^-}\omega\right)\left(\fint_{P^+}\omega^{1-q'}\right)^{q-1}
\lesssim1,
\end{align*}
which, together with the arbitrariness of $P$,
further implies that $\omega\in A_q^+(0)$.
On the other hand, if $\gamma\in(0,1)$, then,
from \eqref{1046}, the arbitrariness of $P$,
and Remark~\ref{5.3}(iii),
we deduce that, for any parabolic rectangle $U$,
\begin{align*}
\left[\fint_{U^-(\gamma)}\omega\right]
\left[\fint_{U^+(\gamma)}\omega^{1-q'}\right]^{q-1}
\lesssim\sup_{P}
\left[\fint_{P^-(\gamma)}\omega\right]
\left[\fint_{S^+(\gamma)}\omega^{1-q'}\right]^{q-1}
\lesssim1
\end{align*}
and hence $\omega\in A_q^+(\gamma)$.
This finishes the proof of
Theorem~\ref{2132-}.
\end{proof}

Next, we show the reverse conclusion of Theorem~\ref{2132-}.
To this end, we need two technical lemmas.

\begin{lemma}\label{924}
Let $\tau\in\{1,...,K\}$ with $K$ in Lemma~\ref{1517},
and let $F:=\{P_i=Q_i\times I_i:\ i\in\Lambda\}$,
with $\Lambda$ being a set of indices,
be a countable collection of
parabolic rectangles satisfying that
\begin{enumerate}
\item[\textup{(i)}]
there exists a positive constant
$C$ such that, for any $k\in\mathbb{Z}$,
\begin{align*}
\sum_{P\in F,\,l(P)=\delta^k}
\mathbf{1}_{P^-}\leq C;
\end{align*}
\item[\textup{(ii)}]
for any $P,U\in F$ with $P\neq U$, $P^-\nsubseteq U^-$;
\item[\textup{(iii)}]
for any $i\in\Lambda$,
$Q_i\in\mathfrak{D}^\tau$.
\end{enumerate}
Then there exists another positive constant $C'$ such that,
for any parabolic rectangle $P_0:=Q_0\times I_0$
with $Q_0\in\mathfrak{D}^\tau$,
\begin{align}\label{F}
\sum_{P\in\mathcal{F}}
\lambda(P)\leq C'\lambda(P_0),
\end{align}
where $\mathcal{F}$, related to $P_0$, is defined by setting
\begin{align*}
\mathcal{F}:=
\left\{P\in F:\ P^+\cap P_0^+\neq\emptyset,\,l(P)<l(P_0)\right\}.
\end{align*}
\end{lemma}

\begin{proof}
Assume that
\begin{align}\label{P0}
P_0:=Q_{\alpha_0}^{k_0,\tau}\times(t_0-\delta^{k_0p},t_0+\delta^{k_0p})
\end{align}
with $k_0\in\mathbb{Z}$, $\alpha_0\in I_{k_0,\tau}$,
and $t_0\in\mathbb{R}$.
We divide $\mathcal{F}$ into
\begin{align*}
\mathcal{F}=\left\{P\in \mathcal{F}:\
P\nsubseteq P_0\right\}
\cup\left\{P\in \mathcal{F}:\
P\subset P_0\right\}=:\mathcal{F}_1\cup\mathcal{F}_2.
\end{align*}
Thus, to prove \eqref{F},
it suffices to show
\begin{align}\label{1058}
\sum_{P\in\mathcal{F}_1}
\lambda(P)\lesssim\lambda(P_0)
\quad\text{and}\quad
\sum_{P\in\mathcal{F}_2}
\lambda(P)\lesssim\lambda(P_0).
\end{align}

We first prove the first inequality of \eqref{1058}.
By both (i) and (ii) of Lemma~\ref{1516}, we find that,
for any $k\in\mathbb{Z}\cap(k_0,\infty)$, there exists
a finite set $\widetilde{I}_{k,\tau}\subset I_{k,\tau}$ of indices such that,
for any $\alpha\in\widetilde{I}_{k,\tau}$,
\begin{align}\label{2020}
Q_\alpha^{k,\tau}\subset Q_{\alpha_0}^{k_0,\tau}.
\end{align}
Assume that $I$ is an interval of $\mathbb{R}$
such that $P=Q_\alpha^{k,\tau}\times I\in\mathcal{F}_1$.
Then \eqref{2020}, the definition of $\mathcal{F}_1$, (ii)
of the present lemma, and \eqref{P0} imply
that $I=(t_I-\delta^{kp},t_I+\delta^{kp})$
with
\begin{align*}
t_I\in(t_0+\delta^{k_0p}-\delta^{kp},t_0+\delta^{k_0p}),
\end{align*}
which, combined with (i) of the present lemma,
further implies that
\begin{align*}
&\#\left\{P:\ P=Q_\alpha^{k,\tau}\times I\in\mathcal{F}_1\right\}\\
&\quad\leq\sum_{P\in\mathcal{F}_1,\,l(P)=\delta^k}
\mathbf{1}_{P^-}\left(z_\alpha^{k,\tau},
t_0+\delta^{k_0p}-\frac{2\delta^{kp}}{3}\right)\\
&\qquad+\sum_{P\in\mathcal{F}_1,\,l(P)=\delta^k}
\mathbf{1}_{P^-}\left(z_\alpha^{k,\tau},
t_0+\delta^{k_0p}-\frac{4\delta^{kp}}{3}\right)\\
&\quad\lesssim1,
\end{align*}
where $z_\alpha^{k,\tau}$ is any point in $Q_\alpha^{k,\tau}$.
From this, \eqref{2020}, and both (i) and
(ii) of the present lemma, we deduce that
\begin{align}\label{F1}
\sum_{P\in\mathcal{F}_1}
\lambda(P)&=\sum_{k=k_0+1}^\infty\sum_{\alpha\in\widetilde{I}_{k,\tau}}
\sum_{\substack{
P=Q_\alpha^{k,\tau}\times I\in\mathcal{F}_1\\ I
\text{ is any interval}}}
\lambda(P)
\lesssim\sum_{k=k_0+1}^\infty\sum_{\alpha\in\widetilde{I}_{k,\tau}}
\mu(Q_\alpha^{k,\tau})\delta^{kp}\\
&\approx\sum_{k=k_0+1}^\infty
\mu(Q_{\alpha_0}^{k_0,\tau})\delta^{kp}
\approx\mu(Q_{\alpha_0}^{k_0,\tau})\delta^{k_0p}
\approx\lambda(P_0).\noz
\end{align}

Next, we prove the second inequality of \eqref{1058}.
Notice that, for any $P\in\mathcal{F}_2$,
we have $l(P)<l(P_0)<\infty$.
We now let
\begin{align*}
\Delta_1:=\left\{P\in\mathcal{F}_2:\
l(P)=\max_{U\in\mathcal{F}_2}\{l(U)\}\right\}
\end{align*}
and, for any $k\in\mathbb{N}\setminus\{1\}$,
\begin{align*}
\widetilde{\Delta}_k:=\left\{P\in\mathcal{F}_2:\
P^-\cap \left(\bigcup_{U\in\bigcup_{j=1}^{k-1}
\Delta_j}U^-\right)=\emptyset\right\}
\end{align*}
and
\begin{align*}
\Delta_k:=\left\{P\in\widetilde{\Delta}_k:\
l(P)=\max_{U\in\widetilde{\Delta}_k}\{l(U)\}\right\}.
\end{align*}
Let $\Delta:=\bigcup_{k\in\mathbb{N}}\Delta_k$
and it is easy to show that, for any
$U\in\mathcal{F}_2\setminus\Delta$,
there exists a parabolic rectangle $P\in\Delta$
such that
$P^-\cap U^-\neq\emptyset$.
This implies that
\begin{align}\label{2125}
\sum_{P\in\mathcal{F}_2}\lambda(P)
\leq\sum_{P\in\Delta}\left[\lambda(P)+
\sum_{U\in\mathcal{F}_2\setminus\Delta,\,
U^-\cap P^-\neq\emptyset,\,l(U)<l(P)}
\lambda(U)\right].
\end{align}
Notice that, for any $P,U\in\mathcal{F}_2$ with $P\neq U$,
$U^-\nsubseteq P^-$.
By this and an argument similar to that used in the estimation of
\eqref{F1}, we find that, for any $P\in\Delta$,
\begin{align*}
\sum_{U\in\mathcal{F}_2\setminus\Delta,\,
U^-\cap P^-\neq\emptyset,\,l(U)<l(P)}
\lambda(U)\lesssim\lambda(P).
\end{align*}
From this, \eqref{2125}, the definition of $\Delta$,
(i) of the present lemma, and
the definitions of both $\widetilde{\Delta}_k$ and $\mathcal{F}_2$,
we deduce that
\begin{align*}
\sum_{P\in\mathcal{F}_2}\lambda(P)&\lesssim
\sum_{P\in\Delta}\lambda(P)
\approx\sum_{k\in\mathbb{N}}\sum_{P\in\Delta_k}\lambda(P)
\approx\sum_{k\in\mathbb{N}}\sum_{P\in\Delta_k}\lambda(P^-)\\
&\approx\sum_{k\in\mathbb{N}}\lambda\left(
\bigcup_{P\in\Delta_k}P^-\right)
\approx\lambda\left(
\bigcup_{k\in\mathbb{N}}
\bigcup_{P\in\Delta_k}P^-\right)\\
&\lesssim\lambda\left(
\bigcup_{P\in\mathcal{F}_2}P\right)\lesssim\lambda(P_0),
\end{align*}
which, together with both \eqref{1058}
and \eqref{F1}, further implies \eqref{F}.
This then finishes the proof of Lemma~\ref{924}.
\end{proof}

\begin{lemma}\label{2133}
Let $\gamma\in[0,1)$
and $\mathfrak{D}^\tau$ be the system of dyadic cubes
in Lemma~\ref{1516},
where $\tau\in\{1,...,K\}$ and $K\in\mathbb{N}$
are the same as in Lemma~\ref{1517}.
Let $j\in\mathbb{N}$,
$\xi\in(0,\infty)$, and $f\in L^1_{\loc}(X\times\mathbb{R})$
be a nonnegative function.
For any given $x_0\in X$ and $r_0\in(0,\infty)$,
consider the set $A:=R\cap E$,
where $R:=B(x_0,r_0)\times[-r_0,r_0]$ and
$$E:=\left\{(x,t)\in X\times\mathbb{R}:\ \xi\leq
\mathcal{M}^{\gamma+}_{\tau,j} f(x,t)\leq2\xi\right\}$$
with
\begin{align}\label{958}
\mathcal{M}^{\gamma+}_{\tau,j}f(x,t)
:=\sup_{\substack{P(\tau,k,\alpha,t)\ni (x,t)\\
k\in\mathbb{Z},\,|k|\leq j,\,\alpha\in I_{k,\tau}}}
\fint_{P^+(\gamma)}f.
\end{align}
For any $(x,t)\in E$,
let $P_{(x,t)}:=Q_\alpha^{k,\tau}\times(t-\delta^{kp},
t+\delta^{kp})$, for some $|k|\leq j$ and $\alpha\in I_{k,\tau}$,
be a parabolic rectangle satisfying that
\begin{align}\label{1817}
\xi\leq\fint_{P_{(x,t)}^+(\gamma)}f\leq2\xi.
\end{align}
Then there exists a set $\Omega\subset A$ such that, for any $(x,t)\in\Omega$,
there exists a set $F_{(x,t)}\subset P_{(x,t)}^+(\gamma)$
such that
\begin{enumerate}
\item[\textup{(i)}]
$A\subset\bigcup_{(x,t)\in\Omega}\overline{P_{(x,t)}^-}$,
where, for any $(x,t)\in\Omega$,
$\overline{P_{(x,t)}^-}$ denotes the \emph{closure} of ${P_{(x,t)}^-}$;
\item[\textup{(ii)}]
$\frac{1}{\lambda(P_{(x,t)}^+(\gamma))}
\int_{F_{(x,t)}}f\gtrsim\xi$;
\item[\textup{(iii)}]
$\sum_{(x,t)\in\Omega}\mathbf{1}_{F_{(x,t)}}\lesssim1$,
\end{enumerate}
where the implicit positive constants
are independent of $\xi$, $f$, $x$, $t$, $x_0$, and $r_0$.
\end{lemma}

\begin{proof}
We first construct $\Omega$. To this end,
by the definitions of both $E$ and $\mathcal{M}^{\gamma+}_{\tau,j}$,
we know that, for any given $(x,t)\in E$, there
exists a $|k|\leq j$, an $\alpha\in I_{k,\tau}$,
and a parabolic rectangle $P_{(x,t)}$ such that
$$P_{(x,t)}=Q_\alpha^{k,\tau}\times
\left(t-\delta^{kp},t+\delta^{kp}\right)$$
and that \eqref{1817} holds true.
Notice that the edge length of parabolic
rectangles taken in \eqref{958}
is not smaller than $\delta^j$. By both (i) and (ii) of Lemma~\ref{1516},
we find that there exists a unique $\beta\in I_{j,\tau}$
such that $x\in Q_\beta^{j,\tau}\subset Q_\alpha^{k,\tau}$.
Thus, for any $y\in Q_\beta^{j,\tau}$,
\eqref{958} implies that
$$\mathcal{M}^{\gamma+}_{\tau,j}f(x,t)
=\mathcal{M}^{\gamma+}_{\tau,j}f(y,t)$$
and hence
$(y,t)\in E$.
This, combined with the fact that
$\fint_{P^+(\tau,k,\alpha,t,\gamma)}f$
is a continuous function of $t$, further implies that,
for any dyadic cube $Q\in\mathfrak{D}^\tau$ at level $j$,
$\{t:\ Q\times\{t\}\subset E\}$
is a closed set.
By this and the assumption that
$A\subset R$ is bounded,
we find that there exists a closed subset $I_\alpha^{j,\tau}\subset[-r_0,r_0]$,
for any $\alpha\in I_{j,\tau}$, such that
\begin{align*}
A=\bigcup_{\alpha\in I_{j,\tau}}\left[\left(Q_\alpha^{j,\tau}\times
I_\alpha^{j,\tau}\right)\cap R\right]
\end{align*}
becomes a union of a finite number of disjoint subsets of $X\times\mathbb{R}$.
Thus, we can choose an $(x_1,t_1)\in A$
such that $t_1=\max_{(x,t)\in A}\{t\}$.
For any $k\in\mathbb{N}\setminus\{1\}$,
choose an $(x_k,t_k)\in A$ such that
\begin{align}\label{1859}
(x_k,t_k)\notin\bigcup_{i=1}^{k-1}
Q_{\alpha_i}^{k_i,\tau}\times(t_i-\delta^{k_ip},t_i]
\end{align}
with $P_{(x_i,t_i)}=Q_{\alpha_i}^{k_i,\tau}
\times(t_i-\delta^{k_ip},t_i+\delta^{k_ip})$ for any $i\in\{1,...,k-1\}$,
and that
\begin{align*}
t_k=\max\left\{t:\ (x,t)\in A\text{ satisfies }\eqref{1859}\right\}.
\end{align*}
Define $\widetilde{\Omega}:=\{(x_i,t_i)\}_{i\in\mathbb{N}}$.
Since, for any $(x,t)\in A$,
the edge length of the parabolic rectangle $P_{(x,t)}$
is not greater than $\delta^{-j}$,
we can choose a $(y_1,s_1)\in\widetilde{\Omega}$
such that
$$l(P_{(y_1,s_1)})=\max_{(y,s)\in\widetilde{\Omega}}\{l(P_{(y,s)})\},$$
where, for any $(y,s)\in\widetilde{\Omega}$,
$l(P_{(y,s)})$ denotes the \emph{edge length}
of the parabolic rectangle $P_{(y,s)}$.
For any $k\in\mathbb{N}\setminus\{1\}$,
choose a $(y_k,s_k)\in\widetilde{\Omega}$
such that, for any $i\in\{1,...,k-1\}$,
\begin{align}\label{1}
P^-_{(y_k,s_k)}\nsubseteq P^-_{(y_i,s_i)}
\end{align}
and
$$l\left(P^-_{(y_k,s_k)}\right)
=\max\left\{l\left(P_{(y,s)}\right):\ P^-_{(y,s)}
\text{ satisfies }\eqref{1}\right\}.$$
Thus, we obtain a sequence $\Omega:=\{(y_i,s_i)\}_{i\in\mathbb{N}}$.
By the above construction of $\Omega$, it is easy to show that,
for any $(x,t),(y,s)\in\Omega$,
\begin{align}\label{1607}
P^-_{(x,t)}\nsubseteq P^-_{(y,s)}
\end{align}
and
\begin{align*}
A\subset\bigcup_{(x,t)\in\Omega}
\overline{P^-_{(x,t)}}.
\end{align*}
Thus, $\Omega$ satisfies (i).

Next, we show that $\Omega$ satisfies (ii). To this end,
we first claim that Lemma~\ref{924}(i) holds true for $\Omega$.
Indeed,
by both (i) and (ii) of Lemma~\ref{1516}, we find that,
for any $(x,t),(y,s)\in\Omega$ with both
$l(P_{(x,t)})=l(P_{(y,s)})$
and $P_{(x,t)}\cap P_{(y,s)}\neq\emptyset$,
\begin{align*}
\pr_X(P_{(x,t)})=\pr_X(P_{(y,s)}),
\end{align*}
where $\pr_X$ denotes the \emph{projection} to $X$.
For any given $(x,t)\in\Omega$,
we consider the set
$$
\Omega_{(x,t)}:=\left\{(y,s)\in\Omega:\ P_{(y,s)}=
\pr_X(P_{(x,t)})\times
\left(s-\left[l(P_{(x,t)})\right]^{p},s+
\left[l(P_{(x,t)})\right]^{p}\right)
\right\}.
$$
Hence, we may identify $\Omega_{(x,t)}$
with a collection of intervals.
This, together with \eqref{1859}, implies that
\begin{align*}
\sum_{(y,s)\in\Omega_{(x,t)}}\mathbf{1}_{P_{(y,s)}}\leq2,
\end{align*}
which, combined with both (i) and (ii) of Lemma~\ref{1516},
further implies that, for any $k\in\mathbb{Z}$,
\begin{align}\label{1640}
\sum_{(x,t)\in\Omega,\,
l(P_{(x,t)})=\delta^k}\mathbf{1}_{P^-_{(x,t)}}\leq
\sum_{(x,t)\in\Omega,\,
l(P_{(x,t)})=\delta^k}\mathbf{1}_{P_{(x,t)}}\leq2.
\end{align}
Thus, $\Omega$ satisfies Lemma~\ref{924}(i). This finishes the
proof of the above claim.
Combining this claim, \eqref{1607}, and Lemma~\ref{924},
we find that, for any given $(x,t)\in\Omega$,
\begin{align*}
\sum_{(y,s)\in\Gamma_{(x,t)}}\lambda(P_{(y,s)})\lesssim\lambda(P_{(x,t)}),
\end{align*}
where
\begin{align*}
\Gamma_{(x,t)}:=\left\{(y,s)\in\Omega:\ P^+_{(y,s)}\cap
P^+_{(x,t)}\neq\emptyset
,\,l(P_{(y,s)})<l(P_{(x,t)})\right\}.
\end{align*}
From this, we deduce that,
for any $(x,t)\in\Omega$,
\begin{align*}
\sum_{(y,s)\in\Gamma_{(x,t)}}\lambda(P^+_{(y,s)}(\gamma))
\approx
\sum_{(y,s)\in\Gamma_{(x,t)}}\lambda(P_{(y,s)})
\lesssim
\lambda(P_{(x,t)})
\approx
\lambda(P^+_{(x,t)}(\gamma)),
\end{align*}
which, together with \eqref{1817}, further implies that
\begin{align*}
\sum_{(y,s)\in\Gamma_{(x,t)}}\int_{P^+_{(y,s)}(\gamma)}f
\lesssim\xi\sum_{(y,s)\in\Gamma_{(x,t)}}
\lambda(P^+_{(y,s)}(\gamma))
\lesssim\xi\lambda(P^+_{(x,t)}(\gamma))
\lesssim\int_{P^+_{(x,t)}(\gamma)}f.
\end{align*}
By this, we find that there exists a positive
constant $\widetilde{C}$ such that
\begin{align}\label{1614}
\sum_{(y,s)\in\Gamma_{(x,t)}}\int_{P^+_{(y,s)}(\gamma)}f
\leq\widetilde{C}\int_{P^+_{(x,t)}(\gamma)}f.
\end{align}
Observe that, for any $(y,s)\in\Omega$,
$\delta^{j}\leq l(P_{(y,s)})\leq\delta^{-j}$.
From this, Lemma~\ref{1516}(iii), $\text{pr}_XA\subset B(x_0,r_0)$,
and the fact that
$(X,\rho,\mu)$ is geometric doubling,
it follows that $\{\text{pr}_X(P_{(y,s)})\}_{(y,s)\in\Omega}$
is a finite set,
which, combined with both \eqref{1640} and the
assumption that $A\subset R$ is bounded,
further implies that, for any given $(x,t)\in\Omega$,
$\#\Gamma_{(x,t)}$
is finite.
Now, we
construct $F_{(x,t)}$ by considering the
following two cases on $\#\Gamma_{(x,t)}$.

\emph{Case 1)} $\#\Gamma_{(x,t)}\leq 2\widetilde{C}$.
In this case, let $F_{(x,t)}:=P^+_{(x,t)}(\gamma)$.
By this and \eqref{1817}, it is easy to show that
(ii) holds true.

\emph{Case 2)} $\#\Gamma_{(x,t)}>2\widetilde{C}$. In this case,
for any $i\in\mathbb{N}\cap[1,\#\Gamma_{(x,t)}]$, let
\begin{align}\label{5.20x}
E_{(x,t)}^{(i)}:=\left\{(z,\tau)\in P^+_{(x,t)}(\gamma):\
\sum_{(y,s)\in\Omega,\,l_{(y,s)}<l_{(x,t)}}
\mathbf{1}_{P^+_{(y,s)}(\gamma)}(z,\tau)\ge i\right\}.
\end{align}
From this and the definition of $\Gamma_{(x,t)}$, it follows that,
for any $(z,\tau)\in P^+_{(x,t)}(\gamma)$,
\begin{align}\label{1606}
\sum_{i=1}^{\#\Gamma_{(x,t)}}
\mathbf{1}_{E_{(x,t)}^{(i)}}(z,\tau)
=\sum_{(y,s)\in\Gamma_{(x,t)}}
\mathbf{1}_{P^+_{(y,s)}(\gamma)}(z,\tau).
\end{align}
By the definition of $E_{(x,t)}^{(i)}$,
we know that, for any $i,j\in\mathbb{N}$
with $i<j$, $E_{(x,t)}^{(j)}\subset E_{(x,t)}^{(i)}$,
which, together with both \eqref{1606} and \eqref{1614},
further implies that
\begin{align}\label{1040}
2\widetilde{C}\int_{E_{(x,t)}^{(2\widetilde{C})}}
f(z,\tau)\,d\mu(z)\,d\tau
&\leq\sum_{i=1}^{\#\Gamma_{(x,t)}}\int_{E_{(x,t)}^{(i)}}
f(z,\tau)\,d\mu(z)\,d\tau\\
&=\int_{P^+_{(x,t)}(\gamma)}f(z,\tau)\sum_{i=1}^{\#\Gamma_{(x,t)}}
\mathbf{1}_{E_{(x,t)}^{(i)}}(z,\tau)\,d\mu(z)\,d\tau
\noz\\
&=\int_{P^+_{(x,t)}(\gamma)}f(z,\tau)\sum_{(y,s)\in\Gamma_{(x,t)}}
\mathbf{1}_{P^+_{(y,s)}(\gamma)}(z,\tau)\,d\mu(z)\,d\tau
\noz\\
&\leq\sum_{(y,s)\in\Gamma_{(x,t)}}\int_{P^+_{(y,s)}(\gamma)}f
(z,\tau)\,d\mu(z)\,d\tau
\noz\\
&\leq\widetilde{C}\int_{P^+_{(x,t)}(\gamma)}f(z,\tau)\,d\mu(z)\,d\tau,\noz
\end{align}
where $E_{(x,t)}^{(2\widetilde{C})}$ is as in \eqref{5.20x} with
$i$ replaced by $2\widetilde{C}$.
Let $F_{(x,t)}:=P^+_{(x,t)}(\gamma)\setminus E_{(x,t)}^{(2\widetilde{C})}$.
From this, \eqref{1040}, and \eqref{1817},
we deduce that
\begin{align*}
\int_{F_{(x,t)}}f=\int_{P^+_{(x,t)}(\gamma)}f-\int_{E_{(x,t)}^{(2C)}}f
\ge\frac{1}{2}\int_{P^+_{(x,t)}(\gamma)}f
\gtrsim\xi\lambda(P^+_{(x,t)}(\gamma)),
\end{align*}
which further implies that (ii) holds true also in this case.
Thus, $\Omega$ satisfies (ii).

Finally, we show that $\Omega$ satisfies (iii). Let $k\in\mathbb{N}$.
Assume that $\widetilde{\Omega}\subset\Omega$
satisfies
$$\#\widetilde{\Omega}=k
\quad\text{and}\quad
\bigcap_{(x,t)\in\widetilde{\Omega}}F_{(x,t)}\neq\emptyset.$$
Let $(z,\tau)\in\bigcap_{(x,t)\in\widetilde{\Omega}}F_{(x,t)}$
and
$l_0:=\max_{(x,t)\in\widetilde{\Omega}}\{l(P_{(x,t)})\}$.
By \eqref{1640}, we find that
there exist at most two parabolic rectangles $\{P_{(x_i,t_i)}\}_{i=1}^{2}$,
with edge length $l_0$, that contain $(z,\tau)$.
Moreover, from the definition of $F_{(x,t)}$, we
deduce that, for any $P_{(x_i,t_i)}$ with $i\in\{1,2\}$,
there exist at most $\lceil2\widetilde{C}\rceil$ upper halves of smaller
parabolic rectangles that might contain $(z,\tau)$.
Thus,
\begin{align*}
k&=\sum_{(x,t)\in\widetilde{\Omega}}
\mathbf{1}_{F_{(x,t)}}(z,\tau)
=\sum_{(x,t)\in\widetilde{\Omega},\,l(P_{(x,t)})=l_0}
\mathbf{1}_{F_{(x,t)}}(z,\tau)
+\sum_{(x,t)\in\widetilde{\Omega},\,l(P_{(x,t)})<l_0}
\mathbf{1}_{F_{(x,t)}}(z,\tau)\\
&\leq2+4\widetilde{C},
\end{align*}
which completes the proof of (iii) and hence Lemma~\ref{2133}.
\end{proof}

Next, by Lemma~\ref{2133},
we show the weak-type
weighted norm
inequality.

\begin{theorem}\label{weakineq}
Let $(X,\rho,\mu)$ be a space of homogeneous type,
$q\in(1,\infty)$, and $\gamma\in[0,1)$.
Assume that $\omega\in A^+_q(\gamma)$.
Then there exists a positive constant $C$
such that, for any $f\in L^q(X\times\mathbb{R},\omega)$ and $\xi\in(0,\infty)$,
\begin{align}\label{1126}
\omega\left(\left\{(x,t)\in X\times\mathbb{R}:\
\mathcal{M}^{\gamma+}f(x,t)>\xi\right\}\right)
\leq C\xi^{-q}\int_{X\times\mathbb{R}}
\left|f\right|^q\omega.
\end{align}
\end{theorem}

\begin{proof}
Without loss of generality, we may
assume that $f\in L^q(X\times\mathbb{R},\omega)$ is a
nonnegative function.
For any given $\tau\in\{1,...,K\}$ with $K$ in Lemma~\ref{1517},
and for any $(x,t)\in X\times\mathbb{R}$,
let
\begin{align*}
\mathcal{M}^{\gamma+}_{\tau} f(x,t)
:=\sup_{\substack{P(\tau,k,\alpha,t)\ni (x,t)\\
k\in\mathbb{Z},\,\alpha\in I_{k,\tau}}}
\fint_{P^+(\gamma)}f.
\end{align*}
From this and the definition of $\mathcal{M}^{\gamma+}$,
it follows that
\begin{align*}
\left\{(x,t)\in X\times\mathbb{R}:\
\mathcal{M}^{\gamma+}f(x,t)>\xi\right\}
\subset\bigcup_{\tau=1}^K
\left\{(x,t)\in X\times\mathbb{R}:\
\mathcal{M}^{\gamma+}_\tau f(x,t)>\xi\right\}.
\end{align*}
Thus, to prove \eqref{1126}, it suffices to show
that, for any $\tau\in\{1,...,K\}$,
\begin{align*}
\omega\left(\left\{(x,t)\in X\times\mathbb{R}:\
\mathcal{M}^{\gamma+}_\tau f(x,t)>\xi\right\}\right)
\lesssim{\xi^{-q}}\int_{X\times\mathbb{R}}
f^q\omega,
\end{align*}
where the implicit positive constant is independent of $\tau$, $f$,
and $\xi$.
To this end, by the monotone convergence theorem,
it suffices to prove
that, for any $\tau\in\{1,...,K\}$ and
$j\in\mathbb{N}$,
\begin{align}\label{1704}
\omega\left(\left\{(x,t)\in X\times\mathbb{R}:\
\mathcal{M}^{\gamma+}_{\tau,j} f(x,t)>\xi\right\}\right)
\lesssim{\xi^{-q}}\int_{X\times\mathbb{R}}
f^q\omega,
\end{align}
where, for any $(x,t)\in X\times\mathbb{R}$,
$\mathcal{M}^{\gamma+}_{\tau,j}f(x,t)$ is the same as in \eqref{958}
and where the implicit positive constant
is independent of $\tau$, $j$, $f$,
and $\xi$.
To prove \eqref{1704}, it suffices to show
that, for any $\tau\in\{1,...,K\}$,
$j\in\mathbb{N}$, and $i\in\mathbb{Z}_+$,
\begin{align}\label{1706}
&\omega\left(\left\{(x,t)\in X\times\mathbb{R}:\ 2^i\xi\leq
\mathcal{M}^{\gamma+}_{\tau,j} f(x,t)\leq 2^{i+1}\xi\right\}\right)\\
&\quad\lesssim2^{-iq}{\xi^{-q}}\int_{X\times\mathbb{R}}
f^q\omega.\noz
\end{align}
Assume this for the moment.
Then
\begin{align*}
&\omega\left(\left\{(x,t)\in X\times\mathbb{R}:\
\mathcal{M}^{\gamma+}_{\tau,j} f(x,t)>\xi\right\}\right)\\
&\quad\leq\sum_{i=0}^\infty
\omega\left(\left\{(x,t)\in X\times\mathbb{R}:\ 2^i\xi\leq
\mathcal{M}^{\gamma+}_{\tau,j} f(x,t)\leq 2^{i+1}\xi\right\}\right)\\
&\quad\lesssim\sum_{i=0}^\infty
2^{-iq}{\xi^{-q}}\int_{X\times\mathbb{R}}
f^q\omega
\approx{\xi^{-q}}\int_{X\times\mathbb{R}}
f^q\omega,
\end{align*}
which completes the proof of \eqref{1704} and hence 
Theorem~\ref{weakineq} under the assumption \eqref{1706}.

It remains to show \eqref{1706}.
Since $\omega\in A_q^+(\gamma)$,
it then follows that \eqref{A+}
implies that the weighted measure $\omega(x,t)\,d\mu(x)\,dt$
is finite on any bounded subsets of $X\times\mathbb{R}$.
Therefore, it suffices to show that, for any set
$$
A:=\left[{B(x,r)}\times[-r,r]\right]\cap
\left\{(x,t)\in X\times\mathbb{R}:\ 2^i\xi\leq
\mathcal{M}^{\gamma+}_{\tau,j} f(x,t)\leq 2^{i+1}\xi\right\}
$$
with both $x\in X$ and $r\in(0,\infty)$,
\begin{align}\label{1741}
\omega(A)
\lesssim2^{-iq}{\xi^{-q}}\int_{X\times\mathbb{R}}
f^q\omega.
\end{align}
Without loss of generality, we may assume that $i=0$.
By Lemma~\ref{2133}, we find that there exists a set
$\Omega\subset A$ such that (i), (ii), and (iii) of Lemma~\ref{2133}
hold true.
From this, the H\"older inequality,
and Remark~\ref{5.3}(iii), we deduce that
\begin{align*}
\omega(A)&\leq\omega\left(\bigcup_{(x,t)\in\Omega}P^-_{(x,t)}\right)
\leq\sum_{(x,t)\in\Omega}\omega(P^-_{(x,t)})\\
&\lesssim\xi^{-q}\sum_{(x,t)\in\Omega}\omega(P^-_{(x,t)})
\left[\frac{1}{\lambda(P_{(x,t)}^+(\gamma))}
\int_{F_{(x,t)}}f\omega^{1/q}\omega^{-1/q}\right]^q\\
&\lesssim\xi^{-q}\sum_{(x,t)\in\Omega}\omega(P^-_{(x,t)})
\left[\frac{1}{\lambda(P_{(x,t)}^+(\gamma))}
\int_{P_{(x,t)}^+(\gamma)}\omega^{1-q'}\right]^{q-1}
\frac{1}{\lambda(P_{(x,t)}^+(\gamma))}
\int_{F_{(x,t)}}f^q\omega\\
&\lesssim\xi^{-q}\sum_{(x,t)\in\Omega}
\int_{F_{(x,t)}}f^q\omega
\lesssim\xi^{-q}\int_{X\times\mathbb{R}}f^q\omega,
\end{align*}
which completes the proof of \eqref{1741} and
hence \eqref{1706}.
This finishes the proof of Theorem~\ref{weakineq}.
\end{proof}

Combining Theorems~\ref{2132-} and~\ref{weakineq}
and Remark~\ref{5.3}(ii),
we immediately obtain the following
Theorems~\ref{5.6} and~\ref{5.7},
where Theorem~\ref{5.6} is on the case that $\gamma=0$
and Theorem~\ref{5.7} is on the case that $\gamma\in(0,1)$;
we omit the details here.

\begin{theorem}\label{5.6}
Let $(X,\rho,\mu)$ be a space of homogeneous type,
$q\in(1,\infty)$, and $\omega$ be a weight.
Then $\omega\in A^+_q(0)$ if and only if
$\mathcal{M}^+$
is bounded from $L^{q}(X\times\mathbb{R},\omega)$ to
$L^{q,\infty}(X\times\mathbb{R},\omega)$.
\end{theorem}

\begin{theorem}\label{5.7}
Let $(X,\rho,\mu)$ be a space of homogeneous type,
$q\in(1,\infty)$, and $\omega$ be a weight.
Then the following statements are equivalent.
\begin{enumerate}
\item[\textup{(i)}]
There exists a $\gamma\in(0,1)$ such that
$\omega\in A^+_q(\gamma)$.
\item[\textup{(ii)}]
For any $\gamma\in(0,1)$,
$\omega\in A^+_q(\gamma)$.
\item[\textup{(iii)}]
There exists a $\gamma\in(0,1)$ such that $\mathcal{M}^{\gamma+}$
is bounded from $L^{q}(X\times\mathbb{R},\omega)$ to
$L^{q,\infty}(X\times\mathbb{R},\omega)$.
\item[\textup{(iv)}]
For any $\gamma\in(0,1)$,
$\mathcal{M}^{\gamma+}$
is bounded from $L^{q}(X\times\mathbb{R},\omega)$ to
$L^{q,\infty}(X\times\mathbb{R},\omega)$.
\end{enumerate}
\end{theorem}

Next, we show that Theorems~\ref{5.6} and~\ref{5.7} hold true
also for the parabolic maximal function, related
to parabolic cylinders, defined as follows.

\begin{definition}
Let $(X,\rho,\mu)$ be a space of homogeneous type
and $\gamma\in[0,1)$. Assume that $f\in L^1_{\loc}(X\times\mathbb{R})$.
The \emph{parabolic maximal functions},
related to parabolic cylinders, $M^{\gamma+}f$
and $M^{\gamma-}f$ of $f$ are defined, respectively,
by setting, for any $(x,t)\in X\times\mathbb{R}$,
\begin{align*}
M^{\gamma+}f(x,t):=\sup_{\{R(x,t,l):\ l\in(0,\infty)\}}
\fint_{R^+(\gamma)}\left|f\right|
\end{align*}
and
\begin{align*}
M^{\gamma-}f(x,t):=\sup_{\{R(x,t,l):\ l\in(0,\infty)\}}
\fint_{R^-(\gamma)}\left|f\right|.
\end{align*}
Moreover, simply write $M^{+}:=M^{0+}$
and $M^{-}:=M^{0-}$.
\end{definition}

Now, we prove the following pointwise
estimates for parabolic maximal operators.

\begin{proposition}\label{Mfequiv}
Let $(X,\rho,\mu)$ be a space of homogeneous type
and $\gamma\in[0,1)$.
Then there exists a positive constant $c$
such that, for any $f\in L^1_{\loc}(X\times\mathbb{R})$
and $(x,t)\in X\times\mathbb{R}$,
\begin{align}\label{1135}
M^{\gamma+}f(x,t) \leq c \mathcal{M}^{\gamma_1+}f(x,t)
\end{align}
and
\begin{align}\label{1136}
\mathcal{M}^{\gamma+}f(x,t) &\leq c M^{\gamma_2+}f(x,t) ,
\end{align}
where $\gamma_1 := \gamma
\min\{(c_1/C)^p,\,c_1^p \delta^p / (2K_0C_1C)^p\}$
and $\gamma_2 := \gamma \min\{1,\,1 / (2K_0C_1)^p\}$
with both $c_1$ and $C_1$ in Lemma~\ref{1516}
and with both $\delta$ and $C$ in Lemma~\ref{1517}.
Moreover, if $\gamma=0$, then $\gamma_1=0=\gamma_2$.
\end{proposition}

\begin{proof}
We first prove \eqref{1135}. To this end,
fix an $(x,t)\in X\times\mathbb{R}$.
Let $r\in(0,\infty)$ and take any parabolic cylinder $R=B(x,r)\times
(t-r^{p},t+r^{p})$.
Lemma~\ref{1517} implies that there exists a dyadic cube
$Q_\alpha^{k,\tau}$ such that \eqref{1721} holds true.
From this and Lemma~\ref{1516}(iii), we deduce that
$B(x,r) \subset Q_\alpha^{k,\tau} \subset B(x,Cr)$,
$r \leq \diam\,(Q_\alpha^{k,\tau})\leq 2K_0C_1 \delta^k$,
and
$c_1 \delta^k \leq\diam\,(Q_\alpha^{k,\tau})\leq Cr$,
which, combined with \eqref{upperdoub}, further implies that
\begin{align}\label{1433}
\fint_{R^+(\gamma)}\left|f\right|
&\leq
\frac{\mu(Q_\alpha^{k,\tau})}{\mu(B(x,r))}
\fint_{Q_\alpha^{k,\tau}\times(t+\gamma r^p,t+r^p)}\left|f\right|\\
&\leq
\frac{\mu(B(x,Cr))}{\mu(B(x,r))}
\fint_{Q_\alpha^{k,\tau}\times(t+\gamma r^p,t+r^p)}\left|f\right|
\noz\\
&\lesssim
\fint_{Q_\alpha^{k,\tau}\times
(t+\gamma c_1^p \delta^{kp}/C^p,t+(2K_0C_1)^p\delta^{kp})}\left|f\right|.
\noz
\end{align}
Then we consider the following two cases on $2K_0C_1$.

\emph{Case 1)} $2K_0C_1\ge1$. In this case, we find that
there exists a $k'\in\mathbb{Z}_+$ such that
\begin{align}\label{1010}
\delta^{-k'+1} < 2K_0C_1 \leq \delta^{-k'}.
\end{align}
By Lemma~\ref{1516}(i), we find that
there exists a $\beta\in I_{k-k'}$ such that
$Q_\alpha^{k,\tau} \subset Q_\beta^{k-k',\tau}$.
From this and \eqref{1010},
we deduce that
\begin{align}\label{1434}
&\fint_{Q_\alpha^{k,\tau}\times
(t+\gamma c_1^p \delta^{kp}/C^p,t+(2K_0C_1)^p\delta^{kp})}
\left|f\right|\\
&\quad\lesssim \fint_{Q_\beta^{k-k',\tau}\times
(t+\gamma c_1^p \delta^{kp}/C^p,t+\delta^{(k-k')p})}\left|f\right|
\noz\\
&\quad\lesssim \fint_{Q_\beta^{k-k',\tau}\times
(t+\widetilde{\gamma}_1 \delta^{(k-k')p},t+\delta^{(k-k')p})}\left|f\right|
\lesssim \mathcal{M}^{\widetilde{\gamma}_1+}f(x,t),
\noz
\end{align}
where
$\widetilde{\gamma}_1:= \gamma c_1^p \delta^p /(2K_0C_1C)^p$.

\emph{Case 2)} $2K_0C_1<1$. In this case, we have
\begin{align}\label{1435}
\fint_{Q_\alpha^{k,\tau}\times
(t+\gamma c_1^p \delta^{kp}/C^p,t+(2K_0C_1)^p\delta^{kp})}\left|f\right|
&\lesssim \fint_{Q_\alpha^{k,\tau}\times
(t+\overline{\gamma}_1\delta^{kp},t+\delta^{kp})}\left|f\right|
\lesssim \mathcal{M}^{\overline{\gamma}_1+}f(x,t),
\end{align}
where $\overline{\gamma}_1:= \gamma c_1^p /C^p$.
Taking the supremum over all parabolic cylinders $R(x,t,r)$
with $r\in(0,\infty)$, by
\eqref{1433}, \eqref{1434}, and \eqref{1435},
we then conclude that
\begin{align*}
M^{\gamma+}f(x,t)
&=\sup_{r\in(0,\infty)}
\fint_{R^+(\gamma)}\left|f\right|
\lesssim\max\left\{\mathcal{M}^{\widetilde{\gamma}_1+}f(x,t),\,
\mathcal{M}^{\overline{\gamma}_1+}f(x,t)\right\}\\
&\approx\mathcal{M}^{\gamma_1+}f(x,t),
\end{align*}
where $\gamma_1:=\min\{\widetilde{\gamma}_1,\,
\overline{\gamma}_1\}$.
Observe that, if $\gamma=0$, then also $\gamma_1=0$.
This finishes the proof of \eqref{1135}.

Next, we prove \eqref{1136}.
To this end, fix an $(x,t)\in X\times\mathbb{R}$.
Take an arbitrary parabolic rectangle $P:=P(\tau,k,\alpha,t)=
Q_\alpha^{k,\tau}\times(t-\delta^{kp},t+\delta^{kp})$ with
$\tau\in\{1,...,K\}$,
$k\in\mathbb{Z}$,
and $\alpha\in I_{k,\tau}$ such that $(x,t)\in P$.
From this, Lemma~\ref{1516}(iii),
and Definition~\ref{2057}(iii),
it follows that $x\in Q_\alpha^{k,\tau}\subset
B(z_\alpha^{k,\tau},C_1\delta^k)$
and, for any $y\in B(z_\alpha^{k,\tau},C_1\delta^k)$,
$$
\rho(x,y)\leq K_0[\rho(x,z_\alpha^{k,\tau})+\rho(z_\alpha^{k,\tau},y)]
<2K_0C_1\delta^k,
$$
which further implies that
$B(z_\alpha^{k,\tau},C_1\delta^k)\subset B(x,2K_0C_1\delta^k)$.
By this, Lemma~\ref{1516}(iii), and \eqref{upperdoub}, we conclude that
\begin{align*}
\fint_{P^+(\gamma)}\left|f\right|
&\leq\frac{\mu(B(z_\alpha^{k,\tau},C_1\delta^k))}{\mu(Q_\alpha^{k,\tau})}
\fint_{B(z_\alpha^{k,\tau},C_1\delta^k)\times
(t+\gamma\delta^{kp},t+\delta^{kp})}\left|f\right|\\
&\leq\frac{\mu(B(z_\alpha^{k,\tau},C_1\delta^k))}
{\mu(B(z_\alpha^{k,\tau},c_1\delta^k))}
\fint_{B(z_\alpha^{k,\tau},C_1\delta^k)\times
(t+\gamma\delta^{kp},t+\delta^{kp})}\left|f\right|\\
&\lesssim \fint_{B(x,\Lambda\delta^k)\times
(t+\gamma\delta^{kp},t+\delta^{kp})}\left|f\right|
\lesssim \fint_{B(x,\Lambda\delta^k)\times
(t+\gamma_2\Lambda^p\delta^{kp},t+\Lambda^p\delta^{kp})}\left|f\right|\\
&\lesssim M^{\gamma_2 +}f(x,t),
\end{align*}
where $\Lambda:=\max\{2K_0C_1,\,1\}$
and $\gamma_2 := \gamma/\Lambda^p$.
Taking supremum over all parabolic rectangles $P\ni(x,t)$,
we further conclude that
\begin{align*}
\mathcal{M}^{\gamma+}f(x,t)
=\sup_{P\ni(x,t)}\fint_{P^+(\gamma)}\left|f\right|
\lesssim M^{\gamma_2+}f(x,t) .
\end{align*}
Also observe that, if $\gamma=0$, then $\gamma_2=0$.
This finishes the proof of \eqref{1136}
and hence Proposition~\ref{Mfequiv}.
\end{proof}

\begin{remark}\label{2137}
By Proposition~\ref{Mfequiv},
we find that Theorems~\ref{5.6} and~\ref{5.7}
hold true also for the parabolic maximal function
$M^{\gamma+}$ related to parabolic cylinders but with
$\gamma\in(0,1)$ in both (iii) and (iv) of Theorem~\ref{5.7} replaced by
\begin{align}\label{1048}
\gamma\in\left(0,\min\left\{\min\left\{1,\,\frac{1}{(2K_0C_1)^p}\right\},\,
\min\left\{\left(\frac{c_1}{C}\right)^p,
\left(\frac{c_1 \delta}{2K_0C_1C}\right)^p\right\}\right\}\right),
\end{align}
where $K_0$ is the coefficient of the quasi-triangle inequality
in Definition~\ref{2057}(iii), $c_1$ and $C_1$ are the same as in
Lemma~\ref{1516}(iii),
and $\delta$ and $C$ are the same as in Lemma~\ref{1517}.
\end{remark}

\section{Reverse H\"older Inequalities}
\label{section6}

In this section, we aim to prove the reverse H\"older inequalities
for parabolic Muckenhoupt weights.
Here is the main theorem of this section.

\begin{theorem}\label{RHI}
Let $(X,\rho,\mu)$ be a space of homogeneous type,
$q\in(1,\infty)$, and $\gamma\in[0,1)$.
Assume that $\omega\in A^+_q(\gamma)$.
Then there exist two positive constants $\kappa$
and $C$ such that, for any parabolic rectangle $P$,
\begin{align}\label{rhi}
\left(\fint_{P^-}\omega^{\kappa+1}\right)^{\frac{1}{\kappa+1}}
\leq C\fint_{P^+}\omega
\end{align}
and
\begin{align}\label{rhii}
\left[\fint_{P^+}\omega^{(1-q')(\kappa+1)}\right]^{\frac{1}{\kappa+1}}
\leq C\fint_{P^-}\omega^{1-q'}.
\end{align}
\end{theorem}

To prove Theorem~\ref{RHI}, we need the following technical lemma.

\begin{lemma}\label{1455}
Let $(X,\rho,\mu)$ be a space of homogeneous type,
$q\in(1,\infty)$, $\gamma\in[0,1)$,
and $\tau\in\{1,...,K\}$ with $K$ in Lemma~\ref{1517}.
Assume that $\omega\in A^+_q(\gamma)$.
Then there exist two constants $\varepsilon,\eta\in(0,1)$ such that,
for any parabolic rectangle $P=Q_\alpha^{k,\tau}\times
(t-\delta^{kp},t+\delta^{kp})$
with $k\in\mathbb{Z}$, $\alpha\in I_{k,\tau}$, and $t\in\mathbb{R}$,
\begin{align}\label{s1}
\lambda\left(\widehat{P}\cap
\left\{(x,t)\in X\times\mathbb{R}:\
\omega(x,t)>\varepsilon\omega_{P^-}\right\}\right)
>\eta\lambda\left(\widehat{P}\right),
\end{align}
where
\begin{align}\label{1507}
\widehat{P}:=Q_\alpha^{k,\tau}\times
\left(t+\delta^{(k+1)p}/2,t+\delta^{(k+1)p}\right).
\end{align}
\end{lemma}

\begin{proof}
By Theorem~\ref{Idp}, without loss of generality,
we may assume that $\gamma=\delta^p/2$.
By this, the Chebyshev inequality, and Remark~\ref{5.3}(iii), we find that,
for any parabolic rectangle $P=Q_\alpha^{k,\tau}\times
(t-\delta^{kp},t+\delta^{kp})$
with $k\in\mathbb{Z}$, $\alpha\in I_{k,\tau}$, and $t\in\mathbb{R}$,
\begin{align}\label{2126}
&\lambda\left(\widehat{P}\cap
\left\{(x,t)\in X\times\mathbb{R}:\
\omega(x,t)\leq\varepsilon\omega_{P^-}\right\}\right)\\
&\quad\leq\int_{\widehat{P}}
\left[\frac{\omega(x,t)}{\varepsilon\omega_{P^-}}\right]^{1-q'}\,d\mu(x)\,dt
=\varepsilon^{q'-1}\lambda\left(\widehat{P}\right)
\omega_{P^-}^{q'-1}\fint_{\widehat{P}}
\left[\omega(x,t)\right]^{1-q'}\,d\mu(x)\,dt
\noz\\
&\quad\lesssim\varepsilon^{q'-1}\lambda\left(\widehat{P}\right)
\left(\sum_{i=1}^{\lceil2\delta^{-p}\rceil}\omega_{P_i^-}^{q'-1}
\right)\fint_{\widehat{P}}\left[\omega(x,t)\right]^{1-q'}\,d\mu(x)\,dt
\lesssim\varepsilon^{q'-1}\lambda\left(\widehat{P}\right),\noz
\end{align}
where, for any $i\in\{1,...,\lceil2\delta^{-p}\rceil\}$,
\begin{align*}
P^-_i:=Q_\alpha^{k,\tau}\times
\left(t-i\delta^{(k+1)p}/2,t-(i-1)\delta^{(k+1)p}/2\right)
\end{align*}
and the implicit positive constants only depend on $\delta$, $p$,
$\varepsilon$, $q$, and $[\omega]_{A^+_q(\gamma)}$.
Choose an $\varepsilon\in(0,1)$ sufficiently small and then \eqref{2126}
implies that there exists
a constant $\eta\in(0,1)$ such that \eqref{s1} holds true,
which completes the proof of Lemma~\ref{1455}.
\end{proof}

Using Lemma~\ref{1455}, we prove the following conclusion.

\begin{lemma}\label{1546}
Let $(X,\rho,\mu)$ be a space of homogeneous type,
$q\in(1,\infty)$, $\gamma\in[0,1)$,
and $\tau\in\{1,...,K\}$ with $K$ as in Lemma~\ref{1517}.
Assume that $\omega\in A^+_q(\gamma)$,
$P_0:=P_0(\tau,k_0,\alpha_0,t_0)$ is a parabolic rectangle,
and $\widetilde{P}_0:=Q_{\alpha_0}^{k_0,\tau}
\times(t_0-\delta^{k_0p},t_0+\delta^{(k_0+1)p})$,
where $Q_{\alpha_0}^{k_0,\tau}$ is a dyadic cube and
$\delta\in(0,1)$ the same
as in Lemma~\ref{1517}.
Then there exist two constants $C\in(0,\infty)$
and $\varepsilon\in(0,1)$, independent of $P_0$,
such that, for any $\xi\in[\omega_{P_0^-},\infty)$,
\begin{align*}
&\omega\left(P_0^-\cap\left\{(x,t)\in X\times\mathbb{R}:\
\omega(x,t)>\xi\right\}\right)\\
&\quad\leq C\xi\lambda\left(\widetilde{P}_0\cap
\left\{(x,t)\in X\times\mathbb{R}:\
\omega(x,t)>\varepsilon\xi\right\}\right).
\end{align*}
\end{lemma}

\begin{proof}
Let
\begin{align*}
\mathfrak{B}:=\left\{U=Q_{\beta}^{l,\tau}
\times(t-\delta^{lp}/2,t+\delta^{lp}/2)
\subset Q_{\alpha_0}^{k_0,\tau}\times(t_0-\delta^{k_0p}
,t_0):\ l\ge k_0,\,\beta\in I_l\right\}
\end{align*}
and, for any $(x,t)\in X\times\mathbb{R}$,
\begin{align*}
\mathcal{M}_{\mathfrak{B}}(\omega)(x,t)
:=\sup_{(x,t)\in U\in\mathfrak{B}}
\fint_{U}\left|\omega\right|.
\end{align*}
From Corollary~\ref{2014} and
the Lebesgue differentiation theorem
(see, for instance, \cite[Theorem 1.8]{H2001}), we deduce that,
for $\lambda$-almost every
$(x,t)\in P_0^-\cap\{(x,t):\ \omega(x,t)>\xi\}$,
\begin{align*}
\mathcal{M}_{\mathfrak{B}}(\omega)(x,t)
=\sup_{(x,t)\in U\in\mathfrak{B}}
\fint_{U}\omega
\ge\omega(x,t)>\xi,
\end{align*}
which further implies that
\begin{align}\label{2045}
P_0^-\cap\left\{(x,t):\ \omega(x,t)>\xi\right\}
\subset\left\{(x,t):\ \mathcal{M}_{\mathfrak{B}}
(\omega)(x,t)>\xi\right\}=:E.
\end{align}

Now, we aim to construct
a Calder\'on--Zygmund type decomposition for $E$.
From the definition of $\mathfrak{B}$,
it is easy to see that, for any $U=Q_{\beta}^{l,\tau}
\times(t-\delta^{lp}/2,t+\delta^{lp}/2)\in\mathfrak{B}$,
one has $t\in(t_0-\delta^{k_0p}
,t_0)$.
Then, for any fixed $t\in(t_0-\delta^{k_0p}
,t_0)$, we define
$E_t:=E\cap(X\times\{t\})$.
Since $\xi\ge\omega_{P_0^-}$,
we infer that there exists a collection $\{\mathcal{Q}_i^t
\}_{i\in \Gamma_t}$ of
maximal dyadic cubes
and $\{U_i^t\}_{i\in \Gamma_t}\subset\mathfrak{B}$,
with $\Gamma_t$ being a set of indices, satisfying that,
for any $i\in\Gamma_t$,
\begin{enumerate}
\item
$\mathcal{Q}_i^t\times\{t\}\subset E_t$;
\item
the projection of $U_i^t$ on $X$ is
$\mathcal{Q}_i^t$
and the projection on $\mathbb R$ contains $\{t\}$;
\item
\begin{align}\label{1646}
\fint_{U_i^t}\omega>\xi;
\end{align}
\item
$\mathcal{Q}_i^t\subsetneqq Q_{\alpha_0}^{k_0,\tau}$,
\end{enumerate}
where the maximality means that, for any $x\in X$ satisfying $(x,t)\in E_t$,
there exists an $i\in\Gamma_t$ such that $x\in\mathcal{Q}_i^t$
and, for any dyadic cube $\mathcal{Q}$ which
properly contains $\mathcal{Q}_i^t$, there exists no $U\in\mathfrak{B}$
satisfying the above conditions with respect to $\mathcal{Q}$.
Indeed, by the definition of $E$, we know that
there exist $\{\mathcal{Q}_i^t\}_{i\in \Gamma_t}$
and $\{U_i^t\}_{i\in \Gamma_t}$
satisfying (i), (ii), and (iii). If (iv) does not hold true,
then $\mathcal{Q}_i^t=Q_{\alpha_0}^{k_0,\tau}$
and $U_i^t=P_0^-$,
which, together with (iii), further implies that
$$\omega_{P_0^-}=\fint_{P_0^-}\omega=\fint_{U_i^t}\omega>\xi.$$
This contradicts $\xi\ge\omega_{P_0^-}$.
Thus, the above statements hold true.

Next, we claim that
\begin{enumerate}
\item[\textup{(a)}]
$E\subset\bigcup_{t\in(t_0-\delta^{k_0p},t_0)}\bigcup_{i\in\Gamma_t}U_i^t$;
\item[\textup{(b)}]
for any $t\in(t_0-\delta^{k_0p},t_0)$ and $i\in\Gamma_t$,
\begin{align}\label{s2}
\fint_{U_i^t}\omega\approx\xi.
\end{align}
\end{enumerate}
Indeed, by the definition of $E_t$,
(i), the maximality of $\{\mathcal{Q}_i^t\}_{i\in\Gamma_t}$,
and (ii), we conclude that
\begin{align*}
E=\bigcup_{t\in(t_0-\delta^{k_0p},t_0)}E_t\times\{t\}
=\bigcup_{t\in(t_0-\delta^{k_0p},t_0)}\bigcup_{i\in\Gamma_t}
\mathcal{Q}_i^t\times\{t\}\subset
\bigcup_{t\in(t_0-\delta^{k_0p},t_0)}\bigcup_{i\in\Gamma_t}U_i^t.
\end{align*}
This finishes the proof of (a). To prove (b),
let $t\in(t_0-\delta^{k_0p},t_0)$ and $i\in\Gamma_t$, and
denote the level of $\mathcal{Q}_i^t$ by $l$.
Since $\mathcal{Q}_i^t$ is maximal and
$\mathcal{Q}_i^t\subsetneq Q_{\alpha_0}^{k_0,\tau}$,
it follows that, for any dyadic cube $Q^{l-1,\tau}_\beta$
satisfying $\mathcal{Q}_i^t\subset
Q^{l-1,\tau}_\beta\subset Q_{\alpha_0}^{k_0,\tau}$,
and for any interval $I$ satisfying that
$U_i^t\subset Q^{l-1,\tau}_\beta\times I$ with $m(I)=\delta^{(l-1)p}$,
\begin{align}\label{835}
\fint_{Q^{l-1,\tau}_\beta\times I}\omega\leq\xi.
\end{align}
If there exists a dyadic cube at level $l-1$
and an interval with length $\delta^{(l-1)p}$ such that \eqref{835} does not
hold true, then $Q^{l-1,\tau}_\beta\in\{\mathcal{Q}_i^t\}_i$,
which, combined with
$\mathcal{Q}_i^t\subset Q^{l-1,\tau}_\beta$,
contradicts the maximality of $\{\mathcal{Q}_i^t\}_i$.
Thus, both \eqref{1646} and \eqref{835} imply that
\begin{align*}
\xi<\fint_{U_i^t}\omega\lesssim
\fint_{Q^{l-1,\tau}_\beta\times I}\omega\lesssim\xi,
\end{align*}
which completes the proof of (b) and hence this claim.

Now, we collect and reorganize $\{U_i^t\}$.
To this end, we let
\begin{align}\label{494}
\mathcal{S}:=\left\{U_i^t:\ t\in(t_0-\delta^{k_0p}
,t_0),\,i\in\Gamma_t\right\}
=:\bigcup_{j\in\mathbb{N}}\mathcal{S}_j=:
\bigcup_{j\in\mathbb{N}}\bigcup_{\beta\in\Lambda_j}\mathcal{S}_{j,\beta},
\end{align}
where, for any $j\in\mathbb{N}$,
we collect all elements of $\mathcal{S}$
with same edge length $\delta^{k_0+j}$
by setting
\begin{align*}
\mathcal{S}_j:=\left\{U=Q\times I\in\mathcal{S}:\
m(I)=\delta^{(k_0+j)p}\right\}
\end{align*}
and, for any $\beta\in\Lambda_j$ with $\Lambda_j$ being a set of indices,
we collect all elements of $\mathcal{S}_j$ with the same
space projection $Q^{j,\tau}_\beta$ by setting
\begin{align*}
\mathcal{S}_{j,\beta}:=\left\{U=Q\times I\in\mathcal{S}_j:\
Q=Q^{j,\tau}_\beta\in\{\mathcal{Q}^t_i\}_{t\in(t_0-\delta^{k_0p}
,t_0),\,i\in\Gamma_t}\right\}.
\end{align*}
By this, it is easy to show that, for any $j\in\mathbb{N}$
and $\beta$, $\beta'\in\Lambda_j$ with $\beta\neq\beta'$,
\begin{align}\label{ji}
\left(\bigcup_{U\in\mathcal{S}_{j,\beta}}U\right)\cap
\left(\bigcup_{U\in\mathcal{S}_{j,\beta'}}U\right)=\emptyset.
\end{align}
Since each $U\in\mathcal{S}_{j,\beta}$
has the same space projection $Q^{j,\tau}_\beta$,
we identify $\mathcal{S}_{j,\beta}$ with a collection
$$
\mathcal{I}_{j,\beta}:=\left\{I:\ Q^{j,\tau}_\beta
\times I\in\mathcal{S}_{j,\beta}\right\}
$$
of intervals having the same length.
Thus, there exists a subset $\widetilde{\mathcal{I}}_{j,\beta}$
of $\mathcal{I}_{j,\beta}$
such that, for any $t\in\mathbb{R}$,
\begin{align*}
\sum_{I\in\widetilde{\mathcal{I}}_{j,\beta}}\mathbf{1}_I(t)\leq2
\quad\text{and}\quad
\bigcup_{I\in\widetilde{\mathcal{I}}_{j,\beta}}I
=\bigcup_{I\in{\mathcal{I}_{j,\beta}}}I,
\end{align*}
and hence there exists a subset
$\widetilde{\mathcal{S}}_{j,\beta}:=\{Q^{j,\tau}_\beta\times I:\
I\in\widetilde{\mathcal{I}}_{j,\beta}\}$
of $\mathcal{S}_{j,\beta}$
such that, for any $(x,t)\in X\times\mathbb{R}$,
\begin{align}\label{1943}
\sum_{U\in\widetilde{\mathcal{S}}_{j,\beta}}\mathbf{1}_U(x,t)\leq2
\quad\text{and}\quad
\bigcup_{U\in\widetilde{\mathcal{S}}_{j,\beta}}U
=\bigcup_{U\in{\mathcal{S}_{j,\beta}}}U,
\end{align}
which, together with \eqref{ji}, further
implies that $\widetilde{\mathcal{S}}_j:=\bigcup_{\beta\in\Lambda_j}
\widetilde{\mathcal{S}}_{j,\beta}$
satisfies \eqref{1943} with $\widetilde{\mathcal{S}}_{j,\beta}$
and ${\mathcal{S}}_{j,\beta}$ replaced, respectively, by
$\widetilde{\mathcal{S}}_{j}$ and ${\mathcal{S}}_{j}$.
This, combined with both (a) and \eqref{494},
implies that
\begin{align}\label{Esubset}
E\subset\bigcup_{t\in(t_0-\delta^{k_0p},t_0)}\bigcup_{i\in\Gamma_t}U_i^t
=\bigcup_{U\in\mathcal{S}}U
=\bigcup_{j\in\mathbb{N}}\bigcup_{U\in\mathcal{S}_j}U
=\bigcup_{j\in\mathbb{N}}\bigcup_{U\in\widetilde{\mathcal{S}}_j}U
=\bigcup_{U\in\widetilde{\mathcal{S}}}U,
\end{align}
where $\widetilde{\mathcal{S}}:=
\bigcup_{j\in\mathbb{N}}\widetilde{\mathcal{S}}_j$.
Moreover,
$\widetilde{\mathcal{S}}$ is countable and
the elements of $\widetilde{\mathcal{S}}$
are interpreted as lower halves of parabolic rectangles,
namely, for any $U\in\widetilde{\mathcal{S}}$,
there exists a parabolic rectangle $P$ such that $U=P^-$.

Next, we define $\Sigma_1$ to be the set of all the elements of
$\widetilde{\mathcal{S}}$ with maximal edge length,
that is, $\Sigma_1=\widetilde{\mathcal{S}}_1$.
For any $k\in\mathbb{N}\setminus\{1\}$, define $\Sigma_k$ to be
the set of all the elements of $\widetilde{\mathcal{S}}_{k}$
satisfying that, for any $P^-\in\Sigma_k$,
\begin{align}\label{2038}
P^+\cap\bigcup_{U^-\in\bigcup_{i=1}^{k-1}\Sigma_i}U^+=\emptyset.
\end{align}
Let $\mathcal{F}:=\bigcup_{k\in\mathbb{N}}\Sigma_k$.
By this, \eqref{2038},
and the first inequality of \eqref{1943} for
$\{\widetilde{\mathcal{S}}_{j}\}_{j\in\mathbb{N}}$, we conclude that
\begin{align}\label{s3}
\sum_{P^-\in \mathcal{F}}\mathbf{1}_{P^+}\leq2.
\end{align}
For any
parabolic rectangle $P$, let $\widehat{P}$ be the same as in \eqref{1507}.
From \eqref{2045}, \eqref{Esubset}, \eqref{s2},
the definition of $\mathcal{F}$, Lemma~\ref{924}, \eqref{s1},
\eqref{1646},
and \eqref{s3} with $\widehat{P}\subset P^+$, we deduce that,
for any $\xi\in[\omega_{P_0^-},\infty)$,
\begin{align*}
&\omega\left(P_0^-\cap\{(x,t):\ \omega(x,t)>\xi\}\right)\\
&\quad\leq\omega(E)\leq\sum_{P^-\in\widetilde{\mathcal{S}}}\omega(P^-)
\approx\xi\sum_{P^-\in\widetilde{\mathcal{S}}}\lambda(P^-)\\
&\quad\approx\xi\sum_{P^-\in \mathcal{F}}\left[\lambda(P^-)+
\sum_{U^-\in\widetilde{\mathcal{S}},\,U^+\cap P^+\neq\emptyset,\,l(U)<l(P)}
\lambda(U^-)\right]\\
&\quad\lesssim\xi\sum_{P^-\in \mathcal{F}}\lambda(P^-)\approx
\xi\sum_{P^-\in \mathcal{F}}\lambda\left(\widehat{P}\right)
\lesssim\xi\sum_{P^-\in \mathcal{F}}\lambda\left(\widehat{P}
\cap\left\{(x,t):\ \omega(x,t)>\varepsilon\omega_{P^-}\right\}\right)\\
&\quad\lesssim\xi\sum_{P^-\in \mathcal{F}}\lambda\left(\widehat{P}
\cap\left\{(x,t):\ \omega(x,t)>\varepsilon\xi\right\}\right)
\lesssim\xi\int_{\bigcup_{P^-\in \mathcal{F}}\widehat{P}
\cap\{(x,t):\ \omega(x,t)>\varepsilon\xi\}}2\\
&\quad\lesssim\xi\lambda\left(\widetilde{P}_0
\cap\left\{(x,t):\ \omega(x,t)>\varepsilon\xi\right\}\right),
\end{align*}
where, in the last inequality, we used the fact that,
for any $P^-\in F$, $\widehat{P}\subset\widetilde{P}_0$,
which completes the proof of Lemma~\ref{1546}.
\end{proof}

Now, using Lemma~\ref{1546}, we prove Theorem~\ref{RHI}.

\begin{proof}[Proof of Theorem~\ref{RHI}]
Fix any constant $m\in(0,\infty)$,
let $u:=\min\{\omega,\,m\}$ be a weight.
Since Proposition~\ref{1627}(i) and $\omega,m\in A^+_q(\gamma)$,
it follows that $u\in A^+_q(\gamma)$.
Now, we first prove that $u$ satisfies
\eqref{rhi} for some $\kappa\in(0,\infty)$.
Let $P:=P(\tau,k,\alpha,t)$, with
$\tau\in\{1,...,K\}$, $k\in\mathbb{Z}$, $\alpha\in I_k$,
and $t\in\mathbb{R}$,
be any fixed parabolic rectangle.
We divide the proof into the following three steps.

\emph{Step 1.} Division process.
Let
$P^-:=Q_\alpha^{k,\tau}\times(t-\delta^{kp},t)$
and
$$\widetilde{P}:=Q_\alpha^{k,\tau}
\times\left(t-\delta^{kp},t+\delta^{(k+1)p}\right).$$
By Lemma~\ref{1516},
we find that there exists a sequence
$\{Q^{k+1,\tau}_{\beta}\}_{\beta\in J_{k+1,\tau}}$
of dyadic cubes such that
\begin{align*}
\bigcup_{\beta\in J_{k+1,\tau}}Q^{k+1,\tau}_{\beta}=Q_\alpha^{k,\tau},
\end{align*}
where $J_{k+1,\tau}\subset I_{k+1,\tau}$
is a set of indices and satisfies
\begin{align*}
\#J_{k+1,\tau}&\leq\left\lceil\frac{\mu(Q_\alpha^{k,\tau})}
{\min_{\beta\in J_{k+1,\tau}}\mu(Q^{k+1,\tau}_{\beta})}\right\rceil
\leq\left\lceil\frac{\mu(B(z_\alpha^{k,\tau},C_1\delta^k))}
{\min_{\beta\in J_{k+1,\tau}}
\mu(B(z^{k+1,\tau}_{\beta},c_1\delta^k))}\right\rceil\\
&\leq\left\lceil\max_{\beta\in J_{k+1,\tau}}
\frac{\mu(B(z^{k+1,\tau}_{\beta},(C_1+c_1)K_0\delta^k))}
{\mu(B(z^{k+1,\tau}_{\beta},c_1\delta^k))}\right\rceil\\
&\leq\left\lceil C_{\mu}
\left[\frac{(C_1+c_1)K_0}{c_1}\right]^n\right\rceil=:M
\end{align*}
with both $c_1$ and $C_1$ in Lemma~\ref{1516}.
Now, for any $\beta\in J_{k+1,\tau}$, we let
\begin{align*}
P^{1-}_{\beta}:=
Q^{k+1,\tau}_{\beta}\times\left(t,t+\delta^{(k+1)p}\right)
\ \text{and}\
\widetilde{P}^{1}_{\beta}:=
Q^{k+1,\tau}_{\beta}
\times\left(t,t+(1+\delta^p)\delta^{(k+1)p}\right).
\end{align*}
For any given $i\in\mathbb{N}\setminus\{1\}$,
by Lemma~\ref{1516},
we find that there exists a sequence
$\{Q^{k+i,\tau}_{\beta}\}_{\beta\in J_{k+i,\tau}}$
of dyadic cubes such that
\begin{align*}
\bigcup_{\beta\in J_{k+i,\tau}}Q^{k+i,\tau}_{\beta}=Q_\alpha^{k,\tau},
\end{align*}
where $J_{k+i,\tau}\subset I_{k+i,\tau}$
is a set of indices and satisfies
$\#J_{k+i,\tau}\leq M^i$.
Next, for any $\beta\in J_{k+i,\tau}$, we let
\begin{align*}
P^{i-}_{\beta}:=Q^{k+i,\tau}_{\beta}
\times\left(t+\delta^{(k+1)p}\frac{1-\delta^{(i-1)p}}{1-\delta^p},
t+\delta^{(k+1)p}\frac{1-\delta^{ip}}{1-\delta^p}\right)
\end{align*}
and
\begin{align*}
\widetilde{P}^{i}_{\beta}:=Q^{k+i,\tau}_{\beta}
\times\left(t+\delta^{(k+1)p}\frac{1-\delta^{(i-1)p}}{1-\delta^p},
t+\delta^{(k+1)p}\frac{1-\delta^{(i+1)p}}{1-\delta^p}\right).
\end{align*}
From this and $\delta\in(0,(96K_0^6)^{-1})$, it is easy to show that
\begin{align}\label{1523}
\bigcup_{i\in\mathbb{N}}\bigcup_{\beta\in J_{k+i,\tau}}
P^{i-}_{\beta}\subset P^+
\end{align}
is a disjoint union
and, for any $i\in\mathbb{N}$,
\begin{align}\label{1051}
\bigcup_{\beta\in J_{k+i,\tau}}
\left(\widetilde{P}^{i}_{\beta}\setminus P^{i-}_{\beta}\right)
=\bigcup_{\beta\in J_{k+(i+1),\tau}}P^{(i+1)-}_{\beta}
\end{align}
is also a disjoint union.

\emph{Step 2.} Iteration process.
Define
$$
E:=\left\{(x,t)\in X\times\mathbb{R}:\ u(x,t)>u_{P^-}\right\}.
$$
By this, the Cavalieri principle,
and Lemma~\ref{1546},
we find that there exist two positive
constants $\varepsilon\in(0,1)$ and $C$,
independent of $P$,
such that, for any $\kappa\in(0,1)$,
\begin{align*}
\int_{P^-\cap E}u^{\kappa+1}
&=\kappa\int_0^\infty s^{\kappa-1}
u\left(\left\{(x,t)\in P^-\cap E:\ u(x,t)>s\right\}\right)
\,ds\\
&=\kappa u(P^-\cap E)\int_0^{u_{P^-}}s^{\kappa-1}\,ds\\
&\quad+\kappa\int_{u_{P^-}}^\infty s^{\kappa-1}
u\left(\left\{(x,t)\in P^-:\ u(x,t)>s\right\}\right)
\,ds\\
&\leq\lambda(P^-\cap E)u_{P^-}^{\kappa+1}
+C\kappa\int_{u_{P^-}}^\infty s^\kappa
\lambda\left(\left\{(x,t)\in\widetilde{P}:\ u(x,t)>\varepsilon s\right\}\right)
\,ds\\
&\leq\lambda(P^-\cap E)u_{P^-}^{\kappa+1}
+\frac{C\kappa}{(\kappa+1)\varepsilon^{\kappa+1}}
\int_{\widetilde{P}\setminus(P^-\cap E)}u^{\kappa+1}\\
&\quad+\frac{C\kappa}{(\kappa+1)\varepsilon^{\kappa+1}}
\int_{P^-\cap E}u^{\kappa+1},
\end{align*}
which implies that, for any $\kappa\in(0,1)$ satisfying
$\frac{C\kappa}{(\kappa+1)\varepsilon^{\kappa+1}}<1$,
\begin{align*}
\int_{P^-\cap E}u^{\kappa+1}&\leq
\frac{(\kappa+1)\varepsilon^{\kappa+1}}
{(\kappa+1)\varepsilon^{\kappa+1}-C\kappa}
\lambda(P^-\cap E)u_{P^-}^{\kappa+1}\\
&\quad+\frac{C\kappa}{(\kappa+1)\varepsilon^{\kappa+1}-C\kappa}
\int_{\widetilde{P}\setminus(P^-\cap E)}u^{\kappa+1}.
\end{align*}
Using this and the fact that, for any $(x,t)\in P^-\setminus E$,
$u(x,t)\leq u_{P^-}$,
we find that we can choose a $\kappa$ sufficiently small such that
\begin{align}\label{1609}
\int_{P^-}u^{\kappa+1}
&=\int_{P^-\cap E}u^{\kappa+1}
+\int_{P^-\setminus E}u^{\kappa+1}\\
&\leq
\frac{(\kappa+1)\varepsilon^{\kappa+1}}
{(\kappa+1)\varepsilon^{\kappa+1}-C\kappa}
\lambda(P^-\cap E)u_{P^-}^{\kappa+1}
+\frac{C\kappa}{(\kappa+1)\varepsilon^{\kappa+1}-C\kappa}
\int_{\widetilde{P}\setminus P^-}u^{\kappa+1}
\noz\\
&\quad+\left[\frac{C\kappa}{(\kappa+1)\varepsilon^{\kappa+1}-C\kappa}
+1\right]\int_{P^-\setminus E}u^{\kappa+1}
\noz\\
&\leq
\frac{(\kappa+1)\varepsilon^{\kappa+1}}
{(\kappa+1)\varepsilon^{\kappa+1}-C\kappa}
\lambda(P^-\cap E)u_{P^-}^{\kappa+1}
+\frac{C\kappa}{(\kappa+1)\varepsilon^{\kappa+1}-C\kappa}
\int_{\widetilde{P}\setminus P^-}u^{\kappa+1}
\noz\\
&\quad+\left[\frac{C\kappa}{(\kappa+1)\varepsilon^{\kappa+1}-C\kappa}
+1\right]\lambda(P^-\setminus E)u_{P^-}^{\kappa+1}
\noz\\
&\leq\left[\frac{C\kappa}{(\kappa+1)\varepsilon^{\kappa+1}-C\kappa}
+1\right]\lambda(P^-)u_{P^-}^{\kappa+1}
+\frac{C\kappa}{(\kappa+1)\varepsilon^{\kappa+1}-C\kappa}
\int_{\widetilde{P}\setminus P^-}u^{\kappa+1}
\noz\\
&\leq2\lambda(P^-)u_{P^-}^{\kappa+1}
+2C\kappa/\varepsilon\int_{\widetilde{P}\setminus P^-}u^{\kappa+1}.\noz
\end{align}
By an argument similar to that used in the estimation of \eqref{1609}
and by \eqref{1051},
we find that, for any $i\in\mathbb{N}$,
\begin{align}\label{t1}
\sum_{\beta\in J_{k+i,\tau}}\int_{P^{i-}_\beta}u^{\kappa+1}
&\leq\sum_{\beta\in J_{k+i,\tau}}\left[2\lambda(P^{i-}_\beta)
u_{P^{i-}_\beta}^{\kappa+1}
+2C\kappa/\varepsilon\int_{\widetilde{P}^{i-}_\beta
\setminus P^{i-}_\beta}u^{\kappa+1}\right]\\
&=\sum_{\beta\in J_{k+i,\tau}}\left[2\lambda(P^{i-}_\beta)
u_{P^{i-}_\beta}^{\kappa+1}\right]
+2C\kappa/\varepsilon\sum_{\beta\in J_{k+(i+1),\tau}}
\int_{P^{(i+1)-}_\beta}u^{\kappa+1}\noz
\end{align}
and
\begin{align}\label{t2}
\int_{\widetilde{P}\setminus P^-}u^{\kappa+1}
=\sum_{\beta\in J_{k+1,\tau}}\int_{P^{1-}_\beta}u^{\kappa+1}.
\end{align}

\emph{Step 3.}
In this step, we aim to prove \eqref{rhi}.
To this end, using \eqref{1609}, \eqref{t2}, and \eqref{t1},
we conclude that, for any $N\in\mathbb{N}$,
\begin{align}\label{952}
\int_{P^-}u^{\kappa+1}&\leq
2\lambda(P^-)u_{P^-}^{\kappa+1}
+2C\kappa/\varepsilon\int_{\widetilde{P}\setminus P^-}u^{\kappa+1}\\
&=2\lambda(P^-)u_{P^-}^{\kappa+1}
+2C\kappa/\varepsilon\sum_{\beta\in J_{k+1,\tau}}
\int_{P^{1-}_\beta}u^{\kappa+1}
\noz\\
&\leq2\sum_{i=0}^N\left[(2C\kappa/\varepsilon)^i
\sum_{\beta\in J_{k+i,\tau}}
\lambda(P^{i-}_\beta)u_{P^{i-}_\beta}^{\kappa+1}\right]
+(2C\kappa/\varepsilon)^N\int_{D_N}u^{\kappa+1}
\noz\\
&=:\mathrm{I}_1+\mathrm{I}_2,\noz
\end{align}
where $J_{k,\tau}:=\{\alpha\}$, $P^{0-}_\alpha:=P^-$,
and
\begin{align*}
D_N:&=\bigcup_{\beta\in J_{k+N,\tau}}
\left(\widetilde{P}^{N}_{\beta}\setminus P^{N-}_{\beta}\right)
=\bigcup_{\beta\in J_{k+(N+1),\tau}}P^{(N+1)-}_{\beta}\\
&=Q_\alpha^{k,\tau}\times
\left(t+\delta^{(k+1)p}\frac{1-\delta^{Np}}{1-\delta^p},
t+\delta^{(k+1)p}\frac{1-\delta^{(N+1)p}}{1-\delta^p}\right).
\end{align*}
For $\mathrm{I}_1$, from \eqref{upperdoub} and
the definition of $M$, we deduce that,
for any $i\in\mathbb{Z}_+$,
\begin{align*}
\sum_{\beta\in J_{k+i,\tau}}
\lambda\left(P^{i-}_\beta\right)u_{P^{i-}_\beta}^{\kappa+1}
&=[\lambda(P^-)]^{-\kappa}\sum_{\beta\in J_{k+i,\tau}}
\left[\frac{\lambda(P^-)}{\lambda(P^{i-}_\beta)}\right]^\kappa
\left(\int_{P^{i-}_\beta}u\right)^{\kappa+1}\\
&\lesssim[\lambda(P^-)]^{-\kappa}\sum_{\beta\in J_{k+i},\tau}
(M\delta^{-p})^{i\kappa}
\left(\int_{P^{i-}_\beta}u\right)^{\kappa+1}\\
&\lesssim[\lambda(P^-)]^{-\kappa}
(M\delta^{-p})^{i\kappa}
\left(\int_{\bigcup_{\beta\in J_{k+i,\tau}}
P^{i-}_\beta}u\right)^{\kappa+1},
\end{align*}
which, together with \eqref{1523}, further implies that
\begin{align}\label{951}
\mathrm{I}_1
&\lesssim[\lambda(P^-)]^{-\kappa}\sum_{i=0}^\infty
\left[(2C\kappa/\varepsilon)^i
(M\delta^{-p})^{i\kappa}
\left(\int_{\bigcup_{\beta\in J_{k+i,\tau}}
P^{i-}_\beta}u\right)^{\kappa+1}\right]\\
&\lesssim[\lambda(P^-)]^{-\kappa}\sum_{i=0}^\infty
\left(\int_{\bigcup_{\beta\in J_{k+i,\tau}}
P^{i-}_\beta}u\right)^{\kappa+1}
\lesssim[\lambda(P^-)]^{-\kappa}\left(\int_Pu\right)^{\kappa+1}
\noz\\
&\approx\lambda(P^-)\left(\fint_Pu\right)^{\kappa+1}
\noz
\end{align}
with $\kappa$ sufficiently small such that
$2C\kappa M^\kappa/(\varepsilon\delta^{\kappa p})\leq1$.
Let
$P_1:=Q_\alpha^{k,\tau}\times(t-\delta^{kp},t-\delta^{kp}/2)$,
$$P_2:=Q_\alpha^{k,\tau}\times(t-\delta^{kp}/2,t),$$
and
$P_3:=Q_\alpha^{k,\tau}\times(t+\delta^{kp}/2,t+\delta^{kp})$.
Then, by \eqref{312-1}, we find that
\begin{align*}
\int_Pu=\int_{P_1}u+\int_{P_2}u+\int_{P^+}u
\lesssim\int_{P_3}u+\int_{P^+}u
\approx\int_{P^+}u,
\end{align*}
which, combined with \eqref{951},
further implies that
\begin{align}\label{1539}
\mathrm{I}_1
\lesssim\lambda(P^-)\left(\fint_{P^+}u\right)^{\kappa+1}.
\end{align}
For $\mathrm{I}_2$, choose a $\kappa$
sufficiently small such that $2C\kappa<\varepsilon$.
Since $\lambda(D_N)=\mu(Q_\alpha^{k,\tau})
\times\delta^{(\kappa+N+1)p}\to0$ as $N\to\infty$,
it follows that $\mathrm{I}_2\to0$ as $N\to\infty$.
This, together with both \eqref{1539} and \eqref{952},
implies that \eqref{rhi}
holds true for $u$.
This, combined with the monotone convergence theorem,
further implies that \eqref{rhi} holds true for $\omega$.
This finishes the proof of Step 3 and hence \eqref{rhi}.

From both Proposition~\ref{1627}(iii) and an argument similar to that used in
the above proof of \eqref{rhi}, we deduce \eqref{rhii},
which completes the proof of Theorem~\ref{RHI}.
\end{proof}

\begin{remark}\label{2118}
Let $q\in(1,\infty)$, $\gamma\in[0,1)$, and
$\omega\in A_q^+(\gamma)$.
By Theorem~\ref{RHI} and an argument similar to that used in the proof of
Proposition~\ref{1633}, we find that
there exists a positive constant $\kappa$
such that, for any given $\gamma_1,\gamma_2\in(0,\infty)$, and
for any $\tau\in\{1,...,K\}$, $k\in\mathbb{Z}$,
$\alpha\in I_{\tau,k}$, and $t\in\mathbb{R}$,
\begin{align*}
\left[\fint_{Q_\alpha^{k,\tau}\times
(t-\gamma_1\delta^{kp},t)}\omega^{\kappa+1}\right]^{\frac{1}{\kappa+1}}
\lesssim\fint_{Q_\alpha^{k,\tau}\times
(t-\gamma_1\delta^{kp},t)+\gamma_2\delta^{kp}}\omega
\end{align*}
with the implicit positive constant depends
on both $\gamma_1$ and $\gamma_2$.
Especially, for any parabolic rectangle $P$
and any $\gamma,\gamma'\in[0,1)$ satisfying that $\gamma<\gamma'$,
\begin{align*}
\left[\fint_{P^-(\gamma)}\omega^{\kappa+1}\right]^{\frac{1}{\kappa+1}}
\lesssim\fint_{P^+(\gamma)}\omega
\ \text{and}\
\left[\fint_{P^-(\gamma')}\omega^{\kappa+1}\right]^{\frac{1}{\kappa+1}}
\lesssim\fint_{P^-(\gamma)}\omega.
\end{align*}
All conclusions of the present remark also hold true for $\omega^{1-q'}$
with the time axis reversed.
\end{remark}

The following theorem shows that
Theorem~\ref{RHI} holds true also with parabolic rectangles
replaced by parabolic cylinders.

\begin{theorem}
\label{RHIball}
Let all the symbols be the same as in Theorem~\ref{RHI}.
Then both \eqref{rhi} and \eqref{rhii} still hold true
with parabolic rectangles
replaced by any parabolic cylinders.
\end{theorem}

\begin{proof}
Let $(x,t)\in X\times\mathbb{R}$ and $r\in(0,\infty)$,
and take any parabolic cylinder $R=B(x,r)\times
(t-r^{p},t+r^{p})$. By Lemma~\ref{1517}, we find that
there exists a dyadic cube $Q_\alpha^{k,\tau}$
such that \eqref{1721} holds true,
which, together with Lemma~\ref{1516}(iii), further implies that
$B(x,r)\subset Q_\alpha^{k,\tau} \subset B(x,Cr)$,
$$
r \leq \diam\,\left(Q_\alpha^{k,\tau}\right)\leq 2K_0C_1 \delta^k,
$$
and
$c_1 \delta^k \leq\diam\,(Q_\alpha^{k,\tau})\leq Cr$,
where both $c_1$ and $C_1$ are the same as in Lemma~\ref{1516}
and where both $\delta$ and $C$ the same as in Lemma~\ref{1517}.
From these, Remark~\ref{2118}, and \eqref{upperdoub},
we deduce that there exists a positive constant $\kappa$
such that
\begin{align}\label{1729}
\left(\fint_{R^-}\omega^{\kappa+1}\right)^{\frac{1}{\kappa+1}}
&\leq
\left[\frac{(2K_0C_1)^p\delta^{kp}\mu(Q_\alpha^{k,\tau})}{r^p\mu(B(x,r))}
\fint_{Q_\alpha^{k,\tau}\times
(t-(2K_0C_1)^p\delta^{kp},t)}\omega^{\kappa+1}
\right]^{\frac{1}{\kappa+1}}\\
&\lesssim
\left[\frac{\mu(B(x,Cr))}{\mu(B(x,r))}
\fint_{Q_\alpha^{k,\tau}\times
(t-(2K_0C_1)^p\delta^{kp},t)}\omega^{\kappa+1}\right]^{\frac{1}{\kappa+1}}
\noz\\
&\lesssim \fint_{Q_\alpha^{k,\tau}\times
(t,t+(2K_0C_1)^p\delta^{kp})}\omega\noz\\
&\lesssim\frac{\mu(B(x,Cr))}{\mu(Q_\alpha^{k,\tau})}\fint_{B(x,Cr)\times
(t,t+(2K_0C_1)^p\delta^{kp})}\omega
\noz\\
&\lesssim\frac{\mu(B(x,Cr))}{\mu(B(x,r))}\fint_{B(x,Cr)\times
(t,t+(2K_0C_1)^p\delta^{kp})}\omega
\noz\\
&\lesssim \fint_{B(x,Cr)\times
(t,t+(2K_0C_1Cr/c_1)^p)}\omega
\lesssim\frac{\mu(B(x,\widetilde{C}r))}
{\mu(B(x,Cr))}
\fint_{U^-(\gamma')}\omega
\noz\\
&\lesssim\fint_{U^-(\gamma')}\omega,
\noz
\end{align}
where
$$
U^-(\gamma') := B(x,\widetilde{C}r)\times
\left(t,t+(1-\gamma')\widetilde{C}^p r^{p}\right)
$$
with both
$\widetilde{C}:= \max\{C, 4K_0C_1 C / c_1, 2\} \geq 2$
and $\gamma' := 1-1/2^p$.
By Theorem~\ref{Idp}, we have
$\omega\in A^+_q(\gamma)\subset A^+_q(\gamma')$.
Notice that
\begin{align*}
S := B(x,r)\times
\left(t+(1+\gamma')\widetilde{C}^p r^p ,t+(1+\gamma')
\widetilde{C}^p r^p + r^p\right)
\end{align*}
is contained in
\begin{align*}
U^+(\gamma')
= U^-(\gamma') + (1+\gamma')\widetilde{C}^p r^p
= B(x,\widetilde{C}r)\times
\left(t+(1+\gamma')\widetilde{C}^p r^p ,t+2\widetilde{C}^p r^p\right),
\end{align*}
because $r^p \leq (1-\gamma')\widetilde{C}^p r^p$
by the definition of $\widetilde{C}$.
Thus, we may apply
\eqref{1826}
for both $U^-(\gamma')$ and $S \subset U^+(\gamma')$.
From this and \eqref{upperdoub}, we deduce that
\begin{align*}
\omega(U^-(\gamma')) &\lesssim
\left[\frac{\lambda(U^-(\gamma'))}{\lambda(S)}\right]^q \omega(S)
\lesssim \left[\frac{\mu(B(x,\widetilde{C}r))}{\mu(B(x,r))}\right]^q
\omega(S) \lesssim \omega(S),
\end{align*}
which, combined with \eqref{1729}, implies that
\begin{align*}
\left(\fint_{R^-}\omega^{\kappa+1}\right)^{\frac{1}{\kappa+1}}
\lesssim\fint_{U^-(\gamma')}\omega
\lesssim\frac{\lambda(S)}{\lambda(U^-(\gamma'))}\fint_{S}\omega
\lesssim\fint_{S}\omega.
\end{align*}
By this, an argument similar to that used
in the proof of Case 2) of the proof of Proposition~\ref{1633},
and \eqref{1826},
we further conclude that
\begin{align*}
\left(\fint_{R^-}\omega^{\kappa+1}\right)^{\frac{1}
{\kappa+1}} \lesssim \fint_{R^+}\omega,
\end{align*}
which shows that~\eqref{rhi} holds true for parabolic cylinders.
Similarly, we obtain~\eqref{rhii} for parabolic cylinders,
which completes the proof of Theorem~\ref{RHIball}.
\end{proof}

Using both Theorem~\ref{RHI} and Remark~\ref{2118},
we obtain the following open property
for parabolic Muckenhoupt weight classes.

\begin{theorem}\label{2135}
Let $(X,\rho,\mu)$ be a space of homogeneous type,
$q\in(1,\infty)$, and $\gamma\in[0,1)$.
Assume that $\omega\in A^+_q(\gamma)$.
Then there exists a constant $\varepsilon\in(0,\infty)$
such that $\omega\in A^+_{q-\varepsilon}(\gamma)$.
\end{theorem}

\begin{proof}
From both $\omega\in A^+_q(\gamma)$ and Remark~\ref{2118},
it follows that there exist two positive constants $\kappa$
and $C$ such that, for any parabolic rectangle $P$,
\begin{align}\label{954}
\left[\fint_{P^+(\gamma)}
\omega^{\frac{1+\kappa}{1-q}}\right]^{\frac{q-1}{1+\kappa}}
\lesssim\left[\fint_{P^-(\gamma)}
\omega^{\frac{1}{1-q}}\right]^{q-1}.
\end{align}
Let $\varepsilon:=\frac{\kappa}{1+\kappa}(q-1)\in(0,q-1)$.
By this, \eqref{954}, and the definition of $A^+_q(\gamma)$,
we find that, for any parabolic rectangle $P$
with edge length $\delta^k$,
\begin{align*}
\left[\fint_{P^+(\gamma)}
\omega^{\frac{1}{1-(q-\varepsilon)}}\right]^{q-\varepsilon-1}
&=\left[\fint_{P^+(\gamma)}
\omega^{\frac{1+\kappa}{1-q}}\right]^{\frac{q-1}{1+\kappa}}
\lesssim\left[\fint_{P^-(\gamma)}
\omega^{\frac{1}{1-q}}\right]^{q-1}
\\
&\lesssim\left[\fint_{P^-(\gamma)-(1+\gamma)
\delta^{kp}}\omega\right]^{-1},
\end{align*}
which, together with Remark~\ref{5.3}(iii), further implies that
$\omega\in A^+_{q-\varepsilon}(\gamma)$.
This finishes the proof of Theorem~\ref{2135}.
\end{proof}

By some arguments similar to those
used in \cite[pp.\,295-296, Subsection 4.2]{ks2016'},
we have the following conclusion; we omit the details here.

\begin{theorem}\label{rdp}
Let $(X,\rho,\mu)$ be a space of homogeneous type
and $q\in(1,\infty)$.
Assume that $\omega$ is a weight.
Then the following statements are equivalent.
\begin{enumerate}
\item[\textup{(i)}]
There exist two positive constants $\kappa$
and $C$ such that, for any parabolic rectangle $P$,
\eqref{rhi} holds true.

\item[\textup{(ii)}]
There exist two positive constants $\varepsilon$
and $\widetilde{C}$ such that, for any parabolic rectangle $P$
and any measurable set $E\subset P^-$,
\begin{align*}
\omega(E)\leq\widetilde{C}\left[\frac{\lambda(E)}
{\lambda(P^-)}\right]^\varepsilon
\omega(P^+).
\end{align*}
\end{enumerate}
\end{theorem}

Using Theorems~\ref{5.7} and~\ref{2135},
Proposition~\ref{1627}(ii), and
the Marcinkiewicz interpolation theorem,
we immediately obtain the following conclusion;
we omit the details here.

\begin{theorem}\label{opeboun}
Let $(X,\rho,\mu)$ be a space of homogeneous type,
$q\in(1,\infty)$, $\gamma_1,\gamma_2,\gamma_3\in(0,1)$,
and $\omega$ be a weight.
Then the following statements are equivalent.
\begin{enumerate}
\item[\textup{(i)}]
$\omega\in A^+_q(\gamma_1)$.
\item[\textup{(ii)}]
$\mathcal{M}^{\gamma_2+}$
is bounded from $L^{q}(X\times\mathbb{R},\omega)$ to
$L^{q,\infty}(X\times\mathbb{R},\omega)$.
\item[\textup{(iii)}]
$\mathcal{M}^{\gamma_3+}$
is bounded from $L^{q}(X\times\mathbb{R},\omega)$ to
$L^{q}(X\times\mathbb{R},\omega)$.
\end{enumerate}
\end{theorem}

\begin{remark}\label{2136}
By Proposition~\ref{Mfequiv},
we find that Theorem~\ref{opeboun} also holds true for
the parabolic maximal function $M^{\gamma+}$
related to parabolic cylinders
but with both $\gamma_2$ and $\gamma_3$ satisfying \eqref{1048}.
\end{remark}

\section{Parabolic $A^+_1(\gamma)$ Class
and a Factorization
of $A^+_q(\gamma)$}
\label{section7}

In this section, we introduce the parabolic Muckenhoupt weight classes
$A_1^+(\gamma)$ and $A_1^-(\gamma)$, and give a factorization
of $A^+_q(\gamma)$ with $q\in(1,\infty)$.
Moreover, we establish a characterization of
$A_1^+(\gamma)$ in terms of the small power of maximal functions
multiplied by bounded functions. To this end,
we first introduce the following concept of
parabolic $A^\pm_1(\gamma)$ classes.

\begin{definition}\label{defA1}
Let $(X,\rho,\mu)$ be a space of homogeneous type
and $\gamma\in[0,1)$.
\begin{enumerate}
\item[\textup{(i)}]
A weight $\omega$ is said to belong
to the \emph{parabolic Muckenhoupt weight class} ${A}^+_1(\gamma)$ if
there exists a positive constant $C$ such that,
for $\lambda$-almost every $(x,t)\in X\times\mathbb{R}$,
\begin{align}\label{16}
\mathcal{M}^{\gamma-}\omega(x,t)\leq C\omega(x,t).
\end{align}
\item[\textup{(ii)}]
A weight $\omega$ is said to belong
to the \emph{parabolic Muckenhoupt weight class} ${A}^-_1(\gamma)$ if
there exists a positive constant $C'$ such that,
for $\lambda$-almost every $(x,t)\in X\times\mathbb{R}$,
\begin{align}\label{17}
\mathcal{M}^{\gamma+}\omega(x,t)\leq C'\omega(x,t).
\end{align}
\end{enumerate}
Moreover,
$[\omega]_{{A}^+_1(\gamma)}$
and $[\omega]_{{A}^-_1(\gamma)}$ are defined to be
the smallest constants, respectively, in \eqref{16} and \eqref{17}.
\end{definition}

\begin{remark}\label{2129}
Let $\gamma,\gamma'\in[0,1)$ with $\gamma\leq\gamma'$.
By Definition~\ref{defA1} and the facts that,
for any $f\in L^1_{\loc}(X\times\mathbb{R})$
and $(x,t)\in X\times\mathbb{R}$,
\begin{align*}
\mathcal{M}^{\gamma'-}f(x,t)\leq\frac{1-\gamma}{1-\gamma'}
\mathcal{M}^{\gamma-}f(x,t)
\ \text{and}\
\mathcal{M}^{\gamma'+}f(x,t)\leq\frac{1-\gamma}{1-\gamma'}
\mathcal{M}^{\gamma+}f(x,t),
\end{align*}
we find that ${A}^+_1(\gamma)\subset{A}^+_1(\gamma')$
and ${A}^-_1(\gamma)\subset{A}^-_1(\gamma')$.
\end{remark}

Now, we consider the relation between $A_1^+(\gamma)$ and $A_q^+(\gamma)$
with $q\in(1,\infty)$.
To this end, we need the following technical lemma.

\begin{lemma}\label{2115}
Let $(X,\rho,\mu)$ be a space of homogeneous type,
$\delta\in(0,1)$ the same as in Lemma~\ref{1517},
$\varepsilon_0:=\delta^{\lfloor\log_{\delta}2^{1/p}\rfloor p}/2$,
$\gamma\in[0,1/{\varepsilon_0})$, and $\omega\in{A}^+_1(\gamma)$.
Then there exists a positive constant $C$ such that,
for any parabolic rectangle $P$,
\begin{align*}
\fint_{P^-(\varepsilon_0\gamma)}\omega
\leq C\essinf_{(y,s)\in P^+(\varepsilon_0\gamma)}\omega(y,s).
\end{align*}
\end{lemma}

\begin{proof}
Assume that $P:=P(\tau,k,\alpha,t)$, with $\tau\in\{1,...,K\}$,
$k\in\mathbb{Z}$, $\alpha\in I_{k,\tau}$,
and $t\in\mathbb{R}$, is any given parabolic rectangle.
Let $$(y,s)\in P^+(\varepsilon_0\gamma)=Q_\alpha^{k,\tau}\times
(t+\varepsilon_0\gamma\delta^{kp},t+\delta^{kp}).$$
Let $k':=k+\lfloor\log_{\delta}2^{1/p}\rfloor<k$.
It is easy to show that $\delta^{k'p}\ge2\delta^{kp}$,
which further implies that
\begin{align}\label{1534}
s-\delta^{k'p}<t+\delta^{kp}-\delta^{k'p}<t-\delta^{kp}
\end{align}
and
\begin{align}\label{7.3x}
s-\gamma\delta^{k'p}
>t+\varepsilon_0\gamma\delta^{kp}-\gamma\delta^{k'p}
=t-\varepsilon_0\gamma\delta^{kp}.
\end{align}
Notice that, by Lemma~\ref{1516}, we know that there exists
a $\beta\in I_{k',\tau}$ such that
$y\in Q_\alpha^{k,\tau}\subset Q_\beta^{k',\tau}$.
This, combined with both \eqref{1534} and \eqref{7.3x}, implies that
\begin{align*}
P^-(\varepsilon_0\gamma)=Q_\alpha^{k,\tau}\times
(t-\delta^{kp},t-\varepsilon_0\gamma\delta^{kp})
\subset Q_\beta^{k',\tau}\times
(s-\delta^{k'p},s-\gamma\delta^{k'p}).
\end{align*}
From this, \eqref{upperdoub}, Lemma~\ref{1516}(iii),
Definition~\ref{M^gamma},
and \eqref{16},
we further deduce that,
for $\lambda$-almost every
$(y,s)\in P^+(\varepsilon_0\gamma)$,
\begin{align*}
\fint_{P^-(\varepsilon_0\gamma)}\omega
&\lesssim\frac{\mu(Q_\beta^{k',\tau})}{\mu(Q_\alpha^{k,\tau})}
\fint_{Q_\beta^{k',\tau}\times
(s-\delta^{k'p},s-\gamma\delta^{k'p})}\omega
\\
&\lesssim\frac{\mu(B(z_\beta^{k',\tau},C_1\delta^{k'}))}
{\mu(B(z_\alpha^{k,\tau},c_1\delta^{k}))}
\fint_{Q_\beta^{k',\tau}\times
(s-\delta^{k'p},s-\gamma\delta^{k'p})}\omega\\
&\lesssim\mathcal{M}^{\gamma-}\omega(y,s)\lesssim\omega(y,s),
\end{align*}
where the implicit positive constants are independent of $y$, $s$,
and $P(\tau,k,\alpha,t)$.
This, together with the arbitrariness of $P$,
then finishes the proof of Lemma~\ref{2115}.
\end{proof}

Using Lemma~\ref{2115}, we obtain the following proposition.

\begin{proposition}\label{73}
Let $(X,\rho,\mu)$ be a space of homogeneous type,
$\varepsilon_0:=\delta^{\lfloor\log_{\delta}2^{1/p}\rfloor p}/2$,
$\gamma\in[0,1/{\varepsilon_0})$,
and $\omega\in{A}^+_1(\gamma)$.
\begin{enumerate}
\item[\textup{(i)}]
If $\gamma=0$, then, for any $q\in(1,\infty)$ and $\gamma'\in[0,1)$,
${A}^+_1(0)\subset A_q^+(\gamma')$.
\item[\textup{(ii)}]
If $\gamma\neq0$, then, for any $q\in(1,\infty)$ and $\gamma'\in(0,1)$,
${A}^+_1(\gamma)\subset A_q^+(\gamma')$.
\end{enumerate}
\end{proposition}

\begin{proof}
Assume that $P:=P(\tau,k,\alpha,t)$, with $\tau\in\{1,...,K\}$,
$k\in\mathbb{Z}$, $\alpha\in I_{k,\tau}$,
and $t\in\mathbb{R}$, is any given parabolic rectangle,
and $\omega_0:=\essinf_{(y,s)\in P^+(\varepsilon_0\gamma)}\omega(y,s)$.
By Lemma~\ref{2115},
we find that, for any $q\in(1,\infty)$,
\begin{align*}
\left[\fint_{P^-(\varepsilon_0\gamma)}\omega\right]
\left[\fint_{P^+(\varepsilon_0\gamma)}\omega^{1-q'}\right]^{q-1}
\lesssim\omega_0
\left[\fint_{P^+(\varepsilon_0\gamma)}\omega_0^{1-q'}\right]^{q-1}
\approx1,
\end{align*}
which further implies that $\omega\in A_q^+(\varepsilon_0\gamma)$.
From this, Theorem~\ref{Idp}(i), and Remark~\ref{1511}(i), we deduce that,
if $\gamma=0$, then,
for any $\gamma'\in[0,1)$,
${A}^+_1(0)\subset A_q^+(0)\subset A_q^+(\gamma')$
and, if $\gamma\in(0,1)$, then,
for any $\gamma'\in(0,1)$,
${A}^+_1(\gamma)
\subset A_q^+(\varepsilon_0\gamma)=A_q^+(\gamma')$.
This finishes the proof of Proposition~\ref{73}.
\end{proof}

Now, we establish a factorization of $A_q^+(\gamma)$
with both $q\in(1,\infty)$ and $\gamma\in[0,1)$.

\begin{theorem}\label{74}
Let $(X,\rho,\mu)$ be a space of homogeneous type, $q\in(1,\infty)$,
and $\varepsilon_0:=\delta^{\lfloor\log_{\delta}2^{1/p}\rfloor p}/2$.
Then
\begin{enumerate}
\item[\textup{(i)}]
for any $\gamma\in[0,1/\varepsilon_0)$,
$u\in A_1^+(\gamma)$, $v\in A_1^-(\gamma)$, and $\gamma'\in(0,1$),
$uv^{1-q}\in A_q^+(\gamma')$.
Especially, if $\gamma=0$, then $uv^{1-q}\in A_q^+(0)$;
\item[\textup{(ii)}]
for any $\gamma,\gamma'\in(0,1$)
and $\omega\in A_q^+(\gamma')$,
there exist two weights $u\in A_1^+(\gamma)$ and $v\in A_1^-(\gamma)$
such that
\begin{align}\label{1542}
\omega=uv^{1-q}.
\end{align}
Especially, if $\gamma'=0$,
then \eqref{1542} holds true for any $\gamma\in[0,1)$.
\end{enumerate}
\end{theorem}

\begin{proof}
We first prove (i).
Assume that
$u\in A_1^+(\gamma)$ and $v\in A_1^-(\gamma)$.
By Lemma~\ref{2115}, we find that,
for any parabolic rectangle $P$,
\begin{align}\label{qq}
\esssup_{(x,t)\in P^+(\varepsilon_0\gamma)}
\left[u(x,t)\right]^{-1}
=\left[\essinf_{(x,t)\in P^+(\varepsilon_0\gamma)}u(x,t)\right]^{-1}
\lesssim
\left[\fint_{P^-(\varepsilon_0\gamma)}u\right]^{-1}
\end{align}
and
\begin{align}\label{ww}
\esssup_{(y,s)\in P^-(\varepsilon_0\gamma)}
\left[v(y,s)\right]^{-1}
=\left[\essinf_{(y,s)\in P^-(\varepsilon_0\gamma)}v(y,s)\right]^{-1}
\lesssim
\left[\fint_{P^+(\varepsilon_0\gamma)}v\right]^{-1}.
\end{align}
Combining \eqref{qq} and \eqref{ww}, we conclude that,
for any $q\in(1,\infty)$ and any parabolic rectangle $P$,
\begin{align*}
&\left[\fint_{P^-(\varepsilon_0\gamma)}uv^{1-q}\right]
\left[\fint_{P^+(\varepsilon_0\gamma)}(uv^{1-q})^{1-q'}\right]^{q-1}\\
&\quad\leq\left[\fint_{P^-(\varepsilon_0\gamma)}u\right]
\left\{\esssup_{(y,s)\in P^-(\varepsilon_0\gamma)}
\left[v(y,s)\right]^{-1}\right\}^{q-1}
\left[\fint_{P^+(\varepsilon_0\gamma)}v\right]^{q-1}\\
&\qquad\times
\esssup_{(x,t)\in P^+(\varepsilon_0\gamma)}\left[u(x,t)\right]^{-1}\\
&\quad\lesssim\left[\fint_{P^-(\varepsilon_0\gamma)}u\right]
\left[\fint_{P^+(\varepsilon_0\gamma)}v\right]^{1-q}
\left[\fint_{P^+(\varepsilon_0\gamma)}v\right]^{q-1}
\left[\fint_{P^-(\varepsilon_0\gamma)}u\right]^{-1}\approx1,
\end{align*}
which further implies that $uv^{1-q}\in A_q^+(\varepsilon_0\gamma)$.
This, combined with Theorem~\ref{Idp}, then finishes the proof of (i).

Next, we prove (ii).
Let $q\in[2,\infty)$,
$\gamma'\in[0,1)$, and $\omega\in A_q^+(\gamma')$.
By Theorem~\ref{Idp},
for any $\gamma\in(0,1)$ when $\gamma'\in(0,1)$,
and for any $\gamma\in[0,1)$ when $\gamma'=0$,
we find that $\omega\in A_q^+(\gamma)$.
Now, we define the sublinear operator $T$ by setting,
for any $f\in L^q(X\times\mathbb{R})$,
\begin{align*}
Tf:=\left[\omega^{-1/q}\mathcal{M}^{\gamma-}
\left(f^{q-1}\omega^{1/q}\right)\right]^{1/(q-1)}+
\omega^{1/q}\mathcal{M}^{\gamma+}\left(f\omega^{-1/q}\right).
\end{align*}
From $\omega\in A_q^+(\gamma)$,
Proposition~\ref{1627}(iii), and Theorem~\ref{opeboun},
it follows that $\mathcal{M}^{\gamma+}$
is bounded from $L^{q}(X\times\mathbb{R},\omega)$ to
$L^{q}(X\times\mathbb{R},\omega)$
and $\mathcal{M}^{\gamma-}$
is bounded from $L^{q'}(X\times\mathbb{R},\omega^{1-q'})$ to
$L^{q'}(X\times\mathbb{R},\omega^{1-q'})$.
This, together with the definition of $T$,
implies that, for any $f\in L^q(X\times\mathbb{R})$,
\begin{align}\label{1730}
\left\|Tf\right\|_{L^q(X\times\mathbb{R})}
&\leq\left\|\left[\omega^{-1/q}\mathcal{M}^{\gamma-}
\left(f^{q-1}\omega^{1/q}\right)\right]^{1/(q-1)}
\right\|_{L^q(X\times\mathbb{R})}\\
&\quad+\left\|\omega^{1/q}\mathcal{M}^{\gamma+}
\left(f\omega^{-1/q}\right)\right\|_{L^q(X\times\mathbb{R})}
\noz\\
&=\left\|\mathcal{M}^{\gamma-}
\left(f^{q-1}\omega^{1/q}\right)\right\|_{L^{q'}
(X\times\mathbb{R},\omega^{1-q'})}^{q'/q}
+\left\|\mathcal{M}^{\gamma+}
\left(f\omega^{-1/q}\right)\right\|_{L^q(X\times\mathbb{R},\omega)}
\noz\\
&\lesssim\left\|f^{q-1}\omega^{1/q}\right\|_{L^{q'}
(X\times\mathbb{R},\omega^{1-q'})}^{q'/q}
+\left\|\left(f\omega^{-1/q}\right)\right\|_{L^q(X\times\mathbb{R},\omega)}\noz\\
&\approx\left\|f\right\|_{L^q(X\times\mathbb{R})},\noz
\end{align}
where the implicit positive constants depend on $[\omega]_{A_q^+(\gamma)}$.
Choose a function $g\in L^q(X\times\mathbb{R})$ satisfying
$\|g\|_{L^q(X\times\mathbb{R})}=1$.
From this and \eqref{1730}, we deduce that
\begin{align*}
&\left\|\sum_{i=1}^\infty\left[2\|T\|_{L^q(X\times\mathbb{R})\to
L^q(X\times\mathbb{R})}\right]^{-i}T^ig\right\|_{L^q(X\times\mathbb{R})}\\
&\quad\leq\sum_{i=1}^\infty\left[2\|T\|_{L^q(X\times\mathbb{R})\to
L^q(X\times\mathbb{R})}\right]^{-i}
\left\|T^ig\right\|_{L^q(X\times\mathbb{R})}\leq1,
\end{align*}
which further implies that the series
\begin{align}\label{1733}
\sum_{i=1}^\infty\left[2\|T\|_{L^q(X\times\mathbb{R})\to
L^q(X\times\mathbb{R})}\right]^{-i}T^ig
\end{align}
converges in $L^q(X\times\mathbb{R})$
and hence $\lambda$-almost everywhere,
where, for any $i\in\mathbb{N}$, $T^i$ denotes the $i$-th
iterate of $T$.
We denote the series of \eqref{1733} pointwisely by $\phi$.
Define $u:=\omega^{1/q}\phi^{q-1}$
and $v:=\omega^{-1/q}\phi$.
By this, it is easy to see that $\omega=uv^{1-q}$.

Next, we prove that $u\in A_1^+(\gamma)$ and $v\in A_1^-(\gamma)$.
Since $q\ge2$, it follows that $T$ is a sublinear operator.
From this, the definition of $\phi$,
and $Tg\ge0$,
we deduce that
\begin{align}\label{1654}
T\phi
&\leq2\|T\|_{L^q(X\times\mathbb{R})\to
L^q(X\times\mathbb{R})}\sum_{i=1}^\infty
\left[2\|T\|_{L^q(X\times\mathbb{R})\to
L^q(X\times\mathbb{R})}\right]^{-(i+1)}T^{i+1}g\\
&=2\|T\|_{L^q(X\times\mathbb{R})\to
L^q(X\times\mathbb{R})}\left[\phi-\frac{Tg}
{2\|T\|_{L^q(X\times\mathbb{R})\to
L^q(X\times\mathbb{R})}}\right]
\noz\\
&\leq2\|T\|_{L^q(X\times\mathbb{R})\to
L^q(X\times\mathbb{R})}\phi\noz
\end{align}
$\lambda$-almost everywhere on $X\times\mathbb{R}$.
Notice that $\phi=(\omega^{-1/q}u)^{1/(q-1)}=\omega^{1/q}v$.
This, combined with both \eqref{1654}
and the definition of $T$, implies that
\begin{align*}
&2\|T\|_{L^q(X\times\mathbb{R})\to
L^q(X\times\mathbb{R})}(\omega^{-1/q}u)^{1/(q-1)}\\
&\quad=2\|T\|_{L^q(X\times\mathbb{R})\to
L^q(X\times\mathbb{R})}\phi
\ge T\phi\\
&\quad=T((\omega^{-1/q}u)^{1/(q-1)})
\ge\left(\omega^{-1/q}\mathcal{M}^{\gamma-}u\right)^{1/(q-1)}
\end{align*}
$\lambda$-almost everywhere on $X\times\mathbb{R}$
and, similarly,
\begin{align*}
2\|T\|_{L^q(X\times\mathbb{R})\to
L^q(X\times\mathbb{R})}\omega^{1/q}v
\ge\omega^{1/q}\mathcal{M}^{\gamma+}v
\end{align*}
$\lambda$-almost everywhere on $X\times\mathbb{R}$,
which further implies that
$$\mathcal{M}^{\gamma-}u\leq\left[2\|T\|_{L^q(X\times\mathbb{R})\to
L^q(X\times\mathbb{R})}\right]^{q-1}u\
\mathrm{and}\ \mathcal{M}^{\gamma+}v\leq2\|T\|_{L^q(X\times\mathbb{R})\to
L^q(X\times\mathbb{R})}v$$
$\lambda$-almost everywhere on $X\times\mathbb{R}$.
By this, we conclude that $u\in A_1^+(\gamma)$ and $v\in A_1^-(\gamma)$,
which completes the proof of (ii) with $q\in[2,\infty)$.

From Proposition~\ref{1627}(iii) and an argument similar to that used in
the proof of (ii) in the case that $q\in[2,\infty)$
with $\omega$ and $q$ replaced, respectively, by
$\omega^{1-q'}$ and $q'$,
we deduce that, for any $q\in(1,2)$, (ii) holds true,
which completes the proof of (ii) and hence Theorem~\ref{74}.
\end{proof}

Next, we establish a characterization of $A_1^+(\gamma)$. Recall that
a measure $\nu$ on $X\times\mathbb{R}$
is Borel regular if open subsets of $X\times\mathbb{R}$ are
$\nu$-measurable and every set $A\subset X\times\mathbb{R}$
is contained in a Borel set $E$ such that
$\nu(A)=\nu(E)$. For any
$f\in L^1_\loc(X\times\mathbb{R})$, the
\emph{Hardy--Littlewood maximal function}
$\CM f$ of $f$ is defined by setting,
for any $(x,t)\in X\times\mathbb{R}$,
\begin{align*}
\CM f(x,t)
:=\sup_{R\ni(x,t)}
\frac{1}{\lambda(R)}\int_{R}\left|f\right|,
\end{align*}
where the supremum is taken over
all the parabolic cylinders $R$ of $X\times\mathbb{R}$ that contain $(x,t)$.

\begin{theorem}\label{76}
Let $(X,\rho,\mu)$ be a space of homogeneous type, $q\in(1,\infty)$,
and $\varepsilon_0:=\delta^{\lfloor\log_{\delta}2^{1/p}\rfloor p}/2$.
Assume that $$\left\{\mathfrak{D}^\tau:=
\left\{Q_\alpha^{k,\tau}:\ k\in\mathbb{Z},\,
\alpha\in I_{k,\tau}\right\}:\ \tau\in\{1,...,K\}\right\}$$
is the adjacent system of dyadic cubes
in Lemma~\ref{1517}.
\begin{enumerate}
\item[\textup{(i)}]
Let $\nu$ be a locally finite nonnegative measure
defined on a $\sigma$-algebra which contains
$$
\left\{B(x,l)\times(t-l^p,t):\
(x,t)\in X\times\mathbb{R},\,l\in(0,\infty)\right\}
$$
such that, for $\lambda$-almost every $(x,t)\in X\times\mathbb{R}$,
$\mathcal{M}^-\nu(x,t)<\infty$.
Then, for any $\varepsilon\in[0,1)$ and $\gamma\in[0,1)$,
$$\omega:=(\mathcal{M}^-\nu)^\varepsilon\in A_1^+(\gamma).$$
\item[\textup{(ii)}]
Let $\gamma\in[0,1/\varepsilon_0)$
and $\omega\in A_1^+(\gamma)$.
Then there exists a Borel regular measure $\nu$ in (i),
$\varepsilon\in[0,1)$,
and a function $\phi$ with $\phi,\phi^{-1}\in L^\infty(X\times\mathbb{R})$
such that, for any $\gamma'\in({\gamma},1)$,
$$
\omega=\phi\times\left(\mathcal{M}^{\gamma'-}\nu\right)^\varepsilon.
$$
\end{enumerate}
\end{theorem}

\begin{proof}
We first prove (i).
If $\varepsilon=0$, it is easy to show that $\omega=1\in A_1^+(\gamma)$
and hence (i) holds true.
Now, we assume that $\varepsilon\in(0,1)$ and
$R:=R(x,t,r)$, with both $(x,t)\in X\times\mathbb{R}$
and $r\in(0,\infty)$, is any given parabolic cylinder.
Let
\begin{align*}
\widetilde{R}^-:=R^-(x,t,2K_0r)
\quad\text{and}\quad
\nu=\nu\mid_{\widetilde{R}^-}+\nu\mid_{{(\widetilde{R}^-)}^\complement}
=:\nu_1+\nu_2.
\end{align*}
On the one hand, since, for any $f\in L^1_{\loc}(X\times\mathbb{R})$,
${M}^-f\lesssim\mathcal{M}f$ pointwisely,
and the Hardy--Littlewood maximal operator $\mathcal{M}$
is bounded from $L^1(X\times\mathbb{R})$ to $L^{1,\infty}(X\times\mathbb{R})$
(see, for instance, \cite[Theorem 3.5]{Cowe77}),
it follows that the sublinear
operator ${M}^-$ is bounded
from $L^1(X\times\mathbb{R})$ to $L^{1,\infty}(X\times\mathbb{R})$.
By this, $\varepsilon\in(0,1)$, and the Kolmogorov inequality
(see, for instance, \cite[Lemma 5.16]{jd}), we find that
\begin{align*}
\left[\int_{R^-}\left({M}^-\nu_1\right)^\varepsilon\right]^\frac{1}{\varepsilon}
\lesssim\left[\lambda(R^-)\right]^{\frac{1}{\varepsilon}-1}
\|\nu_1\|_{L^1(X\times\mathbb{R})}.
\end{align*}
From this, \eqref{upperdoub},
and the definition of ${M}^{-}$,
we deduce that
\begin{align}\label{nu1}
\fint_{R^-}\left({M}^-\nu_1\right)^\varepsilon
\lesssim\left[\lambda(R^-)\right]^{-\varepsilon}
\left[\nu_1\left(\widetilde{R}^-\right)\right]^\varepsilon
\lesssim\left[\lambda\left(\widetilde{R}^-\right)\right]^{-\varepsilon}
\left[\nu_1\left(\widetilde{R}^-\right)\right]^\varepsilon
\lesssim\left({M}^-\nu\right)^\varepsilon.
\end{align}
On the other hand, let $(y,s)\in R^-$ and $l\in(0,\infty)$.
By Definition~\ref{2057}(iii), we find that, for any
$z\in B(y,l)$, $\rho(z,x)\leq K_0[\rho(z,y)+\rho(y,x)]<K_0(l+r)$
and hence
\begin{align}\label{2141}
B(y,l)\subset B(x,K_0(l+r)).
\end{align}
If $l\leq r$, then both $(y,s)\in R^-$ and \eqref{2141} imply that
$$
R^-(y,s,l)=B(y,l)\times(s-l^p,s)\subset
B(x,2K_0r)\times(t-2r^p,t)\subset\widetilde{R}^-.
$$
By this, we further conclude that $R^-(y,s,l)
\cap(\widetilde{R}^-)^\complement=\emptyset$
and hence
\begin{align}\label{2150}
\fint_{R^-(y,s,l)}\nu_2=0.
\end{align}
If $l>r$, from \eqref{2141}, we deduce that
$$R^-(y,s,l)=B(y,l)\times(s-l^p,s)\subset
B(x,2K_0l)\times(t-r^p-l^p,t)\subset R^-(x,t,2K_0l).$$
This, together with \eqref{2150},
Definition~\ref{2057}(iii),
and \eqref{upperdoub}, implies that, for any $(y,s)\in P^-$,
\begin{align*}
{M}^-\nu_2(y,s)
&=\sup_{l\in(0,\infty)}
\fint_{R^-(y,s,l)}\nu_2
=\sup_{l\in(r,\infty)}
\fint_{R^-(y,s,l)}\nu_2\\
&\leq\sup_{l\in(r,\infty)}
\frac{\lambda(R^-(x,t,2K_0l))}{\lambda(R^-(y,s,l))}
\fint_{R^-(x,t,2K_0l)}\nu_2\\
&\leq\sup_{l\in(0,\infty)}
\frac{\mu(B(x,2K_0l))(2K_0l)^p}{\mu(B(y,l))l^p}
\fint_{R^-(x,t,2K_0l)}\nu_2\\
&\leq\sup_{l\in(0,\infty)}
\frac{\mu(B(y,3K_0^2l))(2K_0l)^p}{\mu(B(y,l))l^p}
\fint_{R^-(x,t,2K_0l)}\nu_2
\lesssim{M}^-\nu(x,t),
\end{align*}
which, combined with both \eqref{nu1}
and Proposition~\ref{Mfequiv}, further implies that,
for any $(x,t)\in X\times\mathbb{R}$,
\begin{align*}
\mathcal{M}^-\left(\left[\mathcal{M}^-\nu\right]^\varepsilon\right)(x,t)
&\approx{M}^-\left(\left[{M}^-\nu\right]^\varepsilon\right)(x,t)
\approx\sup_{l\in(0,\infty)}
\fint_{R^-(x,t,l)}
\left({M}^-\nu\right)^\varepsilon\\
&\lesssim\sup_{l\in(0,\infty)}
\left[\fint_{R^-(x,t,l)}
\left({M}^-\nu_1\right)^\varepsilon
+\fint_{R^-(x,t,l)}
\left({M}^-\nu_2\right)^\varepsilon\right]\\
&\lesssim\left[{M}^-\nu(x,t)\right]^\varepsilon
\approx\left[\mathcal{M}^-\nu(x,t)\right]^\varepsilon.
\end{align*}
Thus, for any $\varepsilon\in(0,1)$,
$(\mathcal{M}^-\nu)^\varepsilon\in A_1^+(0)$. By this and
Remark~\ref{2129}, we further conclude that, for any $\gamma\in[0,1)$,
$(\mathcal{M}^-\nu)^\varepsilon\in A_1^+(\gamma)$,
which completes the proof of (i).

Next, we prove (ii).
By both $\omega\in A_1^+(\gamma)$ and Proposition~\ref{73},
we find that, for any $q\in(1,\infty)$
and $\gamma'\in(0,1)$, $\omega\in A_q^+(\gamma')$.
From this, Remark~\ref{2118},
and the definitions of both $\mathcal{M}^{\gamma-}$ and $A_1^+(\gamma)$,
we deduce that there exists a positive
constant $\kappa$ such that, for
$\lambda$-almost every $(x,t)\in X\times\mathbb{R}$,
any parabolic rectangle $P(\tau,k,\alpha,t)\ni(x,t)$,
and any $\gamma'\in({\gamma},1)$,
\begin{align}\label{953}
\left[\fint_{P^-(\gamma')}\omega^{1+\kappa}\right]^{\frac{1}{1+\kappa}}
&\lesssim\fint_{P^-(\gamma)}\omega
\lesssim\mathcal{M}^{\gamma-}\omega(x,t)
\lesssim\omega(x,t).
\end{align}
Define $\nu:=\omega^{1+\kappa}$.
Since $\nu$ is Borel regular,
then, by the Lebesgue differentiation
theorem (see, for instance, \cite[Theorem 1.8]{H2001})
and \eqref{953}, we conclude that,
for $\lambda$-almost every $(x,t)\in X\times\mathbb{R}$,
\begin{align*}
\omega(x,t)
&=\left[\lim_{\substack{k\to\infty\\P(\tau,k,\alpha,t)\ni(x,t)}}
\fint_{P^-(\gamma')}\nu\right]^{\frac{1}{1+\kappa}}
\leq\left(\mathcal{M}^{\gamma'-}\nu\right)^{\frac{1}{1+\kappa}}\\
&=\sup_{P\ni (x,t)}
\left[\fint_{P^-(\gamma')}\omega^{1+\kappa}\right]^{\frac{1}{1+\kappa}}
\lesssim\omega(x,t),
\end{align*}
which further implies that
\begin{align*}
\frac{\omega}{(\mathcal{M}^{\gamma'-}\nu)^{\frac{1}{1+\kappa}}}\approx1
\end{align*}
$\lambda$-almost everywhere.
This finishes the proof of (ii) and hence Theorem~\ref{76}.
\end{proof}

\section{A Characterization of Parabolic
${\mathrm{BMOs}}$\\ on Spaces of Homogeneous Type}
\label{section8}

The main target of this section is to
establish a relation between the parabolic
Muckenhoupt weight class and the parabolic space BMO,
and further establish
a characterization
of the parabolic space $\BMO$ in terms of parabolic maximal functions
in the sense of Coifman and Rochberg.
To this end, we first introduce the following
concept of the parabolic BMO.
Recall that, for any $r\in\mathbb{R}$, $r_+:=\max\{0,\,r\}$.

\begin{definition}\label{2101}
Let $(X,\rho,\mu)$ be a space of homogeneous type.
\begin{enumerate}
\item[\textup{(i)}]
The \emph{parabolic space} $\PBMO^+(X\times\mathbb{R})$
is defined to be the set of all the
functions $u\in L^1_{\loc}(X\times\mathbb{R})$ satisfying that
there exist constants $\gamma\in(0,1)$ and $a_R\in(0,\infty)$,
where $a_R$ may depend on the parabolic cylinder $R$, such that
\begin{align}\label{PBMO+}
\sup_{\{R:=R(x,t,l):\ (x,t)\in X\times\mathbb{R},\,
t\in(0,\infty)\}}\left[\fint_{R^+(\gamma)}(u-a_R)_++
\fint_{R^-(\gamma)}(a_R-u)_+\right]<\infty.
\end{align}
\item[\textup{(ii)}]
The \emph{parabolic space} $\PBMO^-(X\times\mathbb{R})$
is defined to be the set of all the
functions $u\in L^1_{\loc}(X\times\mathbb{R})$ satisfying that
there exist constants $\gamma\in(0,1)$ and $b_R\in(0,\infty)$,
where $b_R$ may depend on the parabolic cylinder $R$, such that
\begin{align}\label{PBMO-}
\sup_{\{R:=R(x,t,l):\ (x,t)\in X\times\mathbb{R},\,
t\in(0,\infty)\}}\left[\fint_{R^-(\gamma)}(u-b_R)_++
\fint_{R^+(\gamma)}(b_R-u)_+\right]<\infty.
\end{align}
\end{enumerate}
\end{definition}

The following proposition directly follows from
Definition~\ref{2101}; we omit the details here.

\begin{proposition}\label{83}
Let $(X,\rho,\mu)$ be a space of homogeneous type.
\begin{enumerate}
\item[\textup{(i)}]
Let $u,v\in\PBMO^+(X\times\mathbb{R})$.
Then, for any $\alpha,\beta\in(0,\infty)$,
$\alpha u+\beta v\in \PBMO^+(X\times\mathbb{R})$.
\item[\textup{(ii)}]
$u\in\PBMO^+(X\times\mathbb{R})$
if and only if $-u\in \PBMO^-(X\times\mathbb{R})$.
\item[\textup{(iii)}]
$L^\infty(X\times\mathbb{R})\subset \PBMO^+(X\times\mathbb{R})$.
\end{enumerate}
\end{proposition}

To prove a characterization of
$\PBMO^+(X\times\mathbb{R})$,
we need the following John--Nirenberg-type lemma,
which is obtained by combining \cite[Lemma 3.4]{a1988}
with a chaining argument
(see \cite[Lemma 3.1]{s2016} for the argument in the Euclidean space case);
we omit the details here.

\begin{lemma}\label{J-N}
Let $(X,\rho,\mu)$ be a space of homogeneous type,
$u\in\PBMO^+(X\times\mathbb{R})$, and $\gamma\in(0,1)$.
Then there exist two constants $A,B\in(0,\infty)$,
depending only on $K_0$, $C_{\mu}$, $\gamma$, and $u$,
such that, for any parabolic cylinder $R:=R(x,t,l)$
and any $\xi\in(0,\infty)$,
\begin{align*}
\lambda\left(R^+(\gamma)\cap
\left\{(x,t):\ \left(u-a_R\right)_+>\xi\right\}\right)
\leq Ae^{-B\xi}\lambda(R)
\end{align*}
and
\begin{align*}
\lambda\left({R}^-(\gamma)\cap
\left\{(x,t):\ \left(a_R-u\right)_+>\xi\right\}\right)
\leq Ae^{-B\xi}\lambda(R).
\end{align*}
\end{lemma}

Using Lemma~\ref{J-N} and Proposition~\ref{83}
and following the proof of \cite[Lemma 7.4]{ks2016},
we obtain the following conclusion;
we omit the details here.

\begin{theorem}\label{85}
Let $(X,\rho,\mu)$ be a space of homogeneous type
and $\gamma\in(0,1)$.
Then, for any $q\in(1,\infty)$,
\begin{align*}
\PBMO^+(X\times\mathbb{R})=\left\{-\xi \log\omega:\
\omega\in A_q^+(\gamma),\ \xi\in(0,\infty)\right\}
\end{align*}
and
\begin{align*}
\PBMO^-(X\times\mathbb{R})=\left\{-\xi \log\omega:\
\omega\in A_q^-(\gamma),\ \xi\in(0,\infty)\right\}.
\end{align*}
\end{theorem}

\begin{remark}
\begin{enumerate}
\item
Combining Theorem~\ref{85} and Remark~\ref{1511}(i),
we find that the parabolic spaces $\PBMO^{\pm}(X\times\mathbb{R})$
are independent of the choice of $\gamma\in(0,1)$,
namely, if \eqref{PBMO+} [resp., \eqref{PBMO-}]
holds true for some $\gamma\in(0,1)$,
then \eqref{PBMO+} [resp., \eqref{PBMO-}]
also holds true for all $\gamma\in(0,1)$.
\item
From both Lemma~\ref{1516}(iii) and \eqref{1721},
it follows that, as subsets of $X\times\mathbb{R}$,
a parabolic cylinder and a parabolic rectangle
are comparable in measure and in diameter.
Thus, Lemma~\ref{J-N} holds true for parabolic rectangles.
Using this
and Remark~\ref{5.3}(iii),
we obtain Theorem~\ref{85}
in which both $\PBMO^+(X\times\mathbb{R})$
and $A_q^+(\gamma)$ are defined by parabolic rectangles
instead of parabolic cylinders.
This, together with Remark~\ref{5.3}(iv),
further implies that,
if we replace the parabolic cylinders
by parabolic rectangles in Definition~\ref{2101},
then this new definition still recovers the same class of functions.
\end{enumerate}
\end{remark}

By an argument similar to that used in the proof of
\cite[Theorem 7.5]{ks2016}, Theorems~\ref{85},
\ref{74}, and~\ref{76},
and Proposition~\ref{83},
we obtain the following
characterization
of $\PBMO^+(X\times\mathbb{R})$ in terms of parabolic maximal functions
in the sense of Coifman and Rochberg;
we omit the details here.

\begin{theorem}\label{86}
Let $(X,\rho,\mu)$ be a space of homogeneous type.
\begin{enumerate}
\item[\textup{(i)}]
Let $f\in \PBMO^+(X\times\mathbb{R})$.
Then there exists a $\gamma\in(0,1)$,
two constants $\alpha,\beta\in(0,\infty)$,
a function $b\in L^\infty(X\times\mathbb{R})$,
and two nonnegative Borel regular measures $\nu_1$ and $\nu_2$
such that
\begin{align*}
f=-\alpha\log\mathcal{M}^{\gamma-}\nu_1
+\beta\log\mathcal{M}^{\gamma+}\nu_2+b.
\end{align*}
\item[\textup{(ii)}]
Let $\alpha,\beta\in(0,\infty)$, $b\in L^\infty(X\times\mathbb{R})$,
and $\nu_1$ and $\nu_2$ be two measures in Theorem~\ref{76}(i).
Then
\begin{align*}
-\alpha\log\mathcal{M}^{-}\nu_1
+\beta\log\mathcal{M}^{+}\nu_2+b\in\PBMO^+(X\times\mathbb{R}).
\end{align*}
\end{enumerate}
\end{theorem}

\medskip

\noindent\textbf{Data Availability}\quad Data sharing not applicable to this article as
no datasets were generated or analysed during
the current study.

\medskip

\noindent\textbf{Declarations}

\medskip

\noindent\textbf{Conflict of Interest}\quad The authors declare that there is no
conflict of interest.

\bigskip

\noindent Juha Kinnunen and
Kim Myyryl\"{a}inen

\medskip

\noindent Department of Mathematics, Aalto University,
P.O. Box 11100, FI-00076 Aalto, Finland

\smallskip

\noindent{\it E-mails:} \texttt{juha.k.kinnunen@aalto.fi} (J. Kinnunen)

\noindent\phantom{{\it E-mails:} }\texttt{kim.myyrylainen@aalto.fi} (K. Myyryl\"{a}inen)

\bigskip

\noindent Dachun Yang (Corresponding author) and Chenfeng Zhu

\medskip

\noindent Laboratory of Mathematics
and Complex Systems (Ministry of Education of China),
School of Mathematical Sciences,
Beijing Normal University,
Beijing 100875, The People's Republic of China

\smallskip

\noindent{\it E-mails:} \texttt{dcyang@bnu.edu.cn} (D. Yang)

\noindent\phantom{{\it E-mails:} }\texttt{cfzhu@mail.bnu.edu.cn} (C. Zhu)

\end{document}